\documentclass{amsart}

\usepackage{amsthm,amsfonts,amsmath,amssymb}
\usepackage{lscape}
\usepackage{lunadiagrams}
\usepackage{hyperref}
\usepackage{pdfpages}

\theoremstyle{plain}

\newtheorem{theorem}{Theorem}[section]
\newtheorem{lemma}[theorem]{Lemma}     

\newtheorem{proposition}[theorem]{Proposition}

\theoremstyle{definition}

\newtheorem{definition}[theorem]{Definition}

\theoremstyle{remark}

\newtheorem{remark}[theorem]{Remark}
\newtheorem{example}[theorem]{Example}

\DeclareMathOperator{\Pic}{Pic}

\DeclareMathOperator{\Hom}{Hom}

\DeclareMathOperator{\height}{ht}
\DeclareMathOperator{\supp}{supp}

 \newcommand{\calL}{\mathcal L}
 
\newcommand{\calO}{\mathcal O} 
\newcommand{\calX}{\mathcal X}

\newcommand{\mC}{\mathbb C} \newcommand{\mN}{\mathbb N}
\newcommand{\mP}{\mathbb P} \newcommand{\mQ}{\mathbb Q}
 \newcommand{\mZ}{\mathbb Z}

\newcommand{\gog}{\mathfrak g}

\newcommand{\gok}{\mathfrak k}

\newcommand{\gop}{\mathfrak p}

\newcommand{\sfA}{\mathsf A} \newcommand{\sfB}{\mathsf B}
\newcommand{\sfC}{\mathsf C} \newcommand{\sfD}{\mathsf D}

\newcommand{\gra}{\alpha} \newcommand{\grb}{\beta}    \newcommand{\grg}{\gamma}
\newcommand{\grd}{\delta} \newcommand{\grl}{\lambda}  \newcommand{\grs}{\sigma}
 \newcommand{\grf}{\varphi}\newcommand{\gro}{\omega}

\newcommand{\grG}{\Gamma} \newcommand{\grD}{\Delta}  
 \newcommand{\grS}{\Sigma}

\newcommand{\lra}        {\longrightarrow}
\newcommand{\isocan}     {\simeq}
\newcommand{\vuoto}      {\varnothing}

\renewcommand{\geq}      {\geqslant}
\renewcommand{\leq}      {\leqslant}
\newcommand{\senza}      {\smallsetminus}

\newcommand{\ol}         {\overline}

\newcommand{\wt}         {\widetilde}

\newsavebox{\kdwzero}
\newsavebox{\kdwone}
\newsavebox{\kdwtwo}
\savebox{\kdwzero}{\put(-240,240){0}}
\savebox{\kdwone}{\put(-240,240){1}}
\savebox{\kdwtwo}{\put(-240,240){2}}

\title[Spherical nilpotent orbits in complex symmetric pairs]{Regular functions on spherical nilpotent\\orbits in complex symmetric pairs:\\classical non-Hermitian cases}

\email{bravi@mat.uniroma1.it}

\curraddr{\textsc{Dipartimento di Matematica\\ Sapienza Universit\`a di Roma\\ 
Piazzale Aldo Moro 5\\ 00185 Roma, Italy}}

\email{rocco.chirivi@unisalento.it}

\curraddr{\textsc{Dipartimento di Matematica e Fisica "Ennio De Giorgi"\\
Via per Arnesano\\ 73047 Monteroni di Lecce (LE), Italy}}

\email{jacopo.gandini@sns.it}

\curraddr{\textsc{Scuola Normale Superiore\\
Piazza dei Cavalieri 7\\ 56126 Pisa, Italy}}

\author{Paolo Bravi, Rocco Chiriv\`\i, Jacopo Gandini}

\begin{document}

\begin{abstract}
Given a classical semisimple complex algebraic group $G$ and a symmetric pair $(G,K)$ of non-Hermitian type, we study the closures of the spherical nilpotent $K$-orbits in the isotropy representation of $K$. For all such orbit closures, we study the normality and we describe the $K$-module structure of the ring of regular functions of the normalizations.
\end{abstract}

\maketitle

\section*{Introduction}
Let $G$ be a connected semisimple complex algebraic group, and let $K$ be the fixed point subgroup of an algebraic involution $\theta$ of $G$. Then $K$ is a reductive group, which is connected if $G$ is simply-connected.

The Lie algebra $\mathfrak g$ of $G$ splits into the sum of eigenspaces of $\theta$, 
\[\mathfrak g=\mathfrak k \oplus \mathfrak p,\]
where the Lie algebra $\mathfrak k$ of $K$ is the eigenspace of eigenvalue 1, and $\mathfrak p$ is the eigenspace of eigenvalue $-1$. The adjoint action of $G$ on $\mathfrak g$, once restricted to $K$, leaves $\mathfrak k$ and $\mathfrak p$ stable. 

Therefore $\mathfrak p$ provides an interesting representation of $K$, called the isotropy representation, where one may want to study the geometry of the $K$-orbits. With this aim, one looks at the so-called nilpotent cone $\mathcal N_\mathfrak p \subset \mathfrak p$, which consists of the elements whose $K$-orbit closure contains the origin. In this case, $\mathcal N_\mathfrak p$ actually consists of the nilpotent elements of $\mathfrak g$ which belong to $\mathfrak p$. By a fundamental result of Kostant and Rallis \cite{KR}, as in the case of the adjoint action of $G$ on $\mathfrak g$, there are finitely many nilpotent $K$-orbits in $\mathfrak p$.

Provided $K$ is connected, we restrict our attention to the spherical nilpotent $K$-orbits in $\mathfrak p$. Here spherical means with an open orbit for a Borel subgroup of $K$, or equivalently with a ring of regular functions which affords a multiplicity-free representation of $K$. The classification of these orbits is known and due to King \cite{Ki04}.

In the present paper, we begin a systematic study of the closures of the spherical nilpotent $K$-orbits in $\mathfrak p$. In particular, we analyze their normality, and describe the $K$-module structure of the coordinate rings of their normalizations. This is done by making use of the technical machinery of spherical varieties, which is recalled in Section~1.


Here we will deal with the case where $(G,K)$ is a classical symmetric pair with $K$ semisimple, the other cases will be treated in forthcoming papers. The semisimplicity of $K$ is equivalent to the fact that $\mathfrak p$ is a simple $K$-module, in which case $G/K$ is also called a symmetric space of non-Hermitian type.

Let $G_{\mathbb R}$ be a real form of $G$ with Lie algebra $\mathfrak g_\mathbb R$ and Cartan decomposition $\mathfrak g_\mathbb R = \mathfrak k_\mathbb R +\mathfrak p_\mathbb R$,  so that $\theta$ is induced by the corresponding Cartan involution of $G_{\mathbb R}$. Then $K$ is the complexification of a maximal compact subgroup $K_{\mathbb R} \subset G_{\mathbb R}$, and the Kostant-Sekiguchi-\makebox[0pt]{\rule{3pt}{0pt}\rule[4pt]{3pt}{0.8pt}}Dokovi\'c correspondence \cite{D87,S87} establishes a bijection between the set of the nilpotent $G_{\mathbb R}$-orbits in $\mathfrak g_{\mathbb R}$ and the set of the nilpotent $K$-orbits in $\mathfrak p$. Let us briefly recall how it works, more details and references can be found in \cite{CoMG}.


Every non-zero nilpotent element $e\in\mathfrak g_\mathbb R$ lies in an $\mathfrak{sl}(2)$-triple $\{h,e,f\}\subset\mathfrak g_\mathbb R$. Every $\mathfrak{sl}(2)$-triple $\{h,e,f\}\subset\mathfrak g_\mathbb R$ is conjugate to a \textit{Cayley triple} $\{h',e',f'\}\subset\mathfrak g_\mathbb R$, that is, an $\mathfrak{sl}(2)$-triple with $\theta(h')=-h'$, $\theta(e')=-f'$ and $\theta(f')=-e'$. To a Cayley triple in $\mathfrak g_\mathbb R$ one can associate its \textit{Cayley transform}
\[\{h,e,f\}\mapsto \{i(e-f), \tfrac12 (e+f+ih), \tfrac12 (e+f-ih)\}:\]
this is a normal triple in $\mathfrak g$, that is, an $\mathfrak{sl}(2)$-triple $\{h',e',f'\}$ with $h'\in\mathfrak k$ and $e',f'\in\mathfrak p$. By \cite{KR}, any non-zero nilpotent element $e\in\mathfrak p$ lies in a normal triple $\{h,e,f\}\subset\mathfrak g$, and any two normal triples with the same nilpositive element $e$ are conjugated under $K$. Then the desired bijective correspondence is constructed as follows: let $\mathcal O \subset \mathfrak g_\mathbb R$ be an adjoint nilpotent orbit, choose an element $e \in \mathcal O$ belonging to a Cayley triple $\{h,e,f\}$, consider its Cayley transform $\{h',e',f'\}$ and let $\mathcal O' = Ke'$: then $\mathcal O' \subset \mathfrak p$ is the nilpotent $K$-orbit corresponding to $\mathcal O$. 

Among the nice geometrical properties of the Kostant-Sekiguchi-\makebox[0pt]{\rule{3pt}{0pt}\rule[4pt]{3pt}{0.8pt}}Dokovi\'c correspondence, we just recall here one result concerning sphericality: the spherical nilpotent $K$-orbits in $\mathfrak p$ correspond to the adjoint nilpotent $G_\mathbb R$-orbits in $\mathfrak g_\mathbb R$ which are multiplicity free as Hamiltonian $K_\mathbb R$-spaces \cite{Ki02}.


In accordance with the philosophy of the orbit method (see e.g.\ \cite{AHV}), the unitary representations of $G_{\mathbb R}$ should be parametrized by the (co-)adjoint orbits of $G_{\mathbb R}$. In particular one is interested in the so-called unipotent representations of $G_{\mathbb R}$, namely those which should be attached to nilpotent orbits.
The $K$-module structure of the ring of regular functions on a nilpotent $K$-orbit in $\mathfrak p$ (which we compute in our spherical cases) should give information on the corresponding unitary representation of $G_\mathbb R$. Unitary representations that should be attached to the spherical nilpotent $K$-orbits are studied in \cite{HL} (when $G$ is a classical group) and \cite{Sab} (when $G$ is the special linear group). When $G$ is the symplectic group, for particular spherical nilpotent $K$-orbits, such representations are constructed in \cite{Wo} and \cite{Ya}.

The normality and the $K$-module structure of the coordinate ring of the closure of a spherical nilpotent $K$-orbit in $\mathfrak p$ have been studied in several particular cases, with different methods, by Nishiyama \cite{Ni1}, \cite{Ni2}, by Nishiyama, Ochiai and Zhu \cite{NOZ}, and by Binegar \cite{Bi}.


In Appendix \ref{A} we report the list of the spherical nilpotent $K$-orbits in $\mathfrak p$ for all symmetric pairs $(\mathfrak g,\mathfrak k)$ of classical non-Hermitian type.

In the classical cases, the adjoint nilpotent orbits in real simple algebras are classified in terms of signed partitions, as explained in \cite[Chapter~9]{CoMG}. In the list, every orbit is labelled with its corresponding signed partition.

For every orbit we provide an explicit description of a representative $e\in\mathfrak p$, as element of a normal triple $\{h,e,f\}$, and the centralizer of $e$, which we denote by $K_e$. All these data can be directly computed using King's paper on the classification of the spherical nilpotent $K$-orbits \cite{Ki04} (but we point out a missing case therein, see Remark~\ref{remark:missingcase}).

The first datum which is somewhat new in this work is the Luna spherical system associated with $\mathrm N_K(K_e)$, the normalizer of $K_e$ in $K$, which is a wonderful subgroup of $K$. It is equal to $K_{[e]}$, the stabilizer of the line through $e$,  and notice that $K_{[e]}/K_e\cong\mathbb C^\times$.

The Luna spherical systems are used to deduce the normality or non-normality of the $K$-orbits, and to compute the corresponding $K$-modules of regular functions.

Appendix \ref{B} consists of two sets of tables, where we summarize our results on the spherical nilpotent $K$-orbits in $\mathfrak p$. Given such an orbit $\mathcal O = Ke$, in the first set (Tables~2--11) we describe the normality of its closure $\overline{\mathcal O}$, and if $\widetilde{\mathcal O} \lra \overline{\mathcal O}$ denotes the normalization, we describe the $K$-module structure of $\mathbb C[\widetilde{\mathcal O}]$ by giving a set of generators of its weight semigroup $\grG(\widetilde{\mathcal O})$ (that is, the set of the highest weights occurring in $\mathbb C[\widetilde{\mathcal O}]$). The second set (Tables~12--20) contains the Luna spherical systems of $\mathrm N_K(K_e)$.

In Section~1 we compute the Luna spherical systems. In Section~2 we study the multiplication of sections of globally generated line bundles on the corresponding wonderful varieties, which turns out to be always surjective in all cases except one. In Section~3 we deduce our results on normality and semigroups. 
  
\subsection*{Notation}

Simple roots of irreducible root systems are denoted by $\alpha_1,\alpha_2,\ldots$ and enumerated as in Bourbaki, when belonging to different irreducible components they are denoted by $\alpha_1,\alpha_2,\ldots$, $\alpha'_1,\alpha'_2,\ldots$, $\alpha''_1,\alpha''_2,\ldots$, and so on.  For the fundamental weights we adopt the same convention, they are denoted by $\omega_1,\omega_2,\ldots$, $\omega'_1,\omega'_2,\ldots$, $\omega''_1,\omega''_2,\ldots$, and so on. In the tables for the orthogonal cases at the end of the paper we use a variation of the fundamental weights $\varpi_1,\varpi_2,\ldots$ which is explained in Appendix~\ref{B}. 

By $V(\lambda)$ we denote the simple module of highest weight $\lambda$, the acting group will be clear from the context.

\section{Spherical systems}

In this section we compute the Luna spherical systems given in the tables at the end of the paper, in Appendix~\ref{B}. 

First, let us briefly explain what a Luna spherical system is, see e.g.\ \cite{BL} for a plain introduction.

\subsection{Luna spherical systems} 

Recall that a subgroup $H$ of $K$ is called wonderful if the homogeneous space $K/H$ admits an open equivariant embedding in a wonderful $K$-variety. 
A $K$-variety is called wonderful if it is smooth, complete, with an open $K$-orbit whose complement is union of $D_1,\ldots,D_r$ smooth prime $K$-stable divisors with non-empty transversal crossings such that two points $x,x'$ lie in the same $K$-orbit if and only if
\[\{i : x\in D_i\}=\{i : x'\in D_i\}.\]
The wonderful embedding of $K/H$ is unique up to equivariant isomorphism, and is a projective spherical $K$-variety. The number $r$ of the prime $K$-stable divisors is called the rank of $X$.

Let us fix, inside $K$, a maximal torus $T$ and a Borel subgroup $B$ containing $T$. 
This choice yields a root system $R$ and a set of simple roots $S$ in $R$. Let us also denote by $(\,,\,)$ the scalar product in the Euclidean space spanned by $R$, by $\alpha^\vee$ the coroot associated with $\alpha$, and by $\langle\,,\,\rangle$ the usual Cartan pairing 
\[\langle\alpha^\vee,\lambda\rangle=2\frac{(\alpha,\lambda)}{(\alpha,\alpha)}.\]  

For any spherical $K$-variety $X$, the set of colors, which is denoted by $\Delta_X$, is the set of prime $B$-stable non-$K$-stable divisors of $X$. 
It is a finite set. In our case, if $X$ is the wonderful embedding of $K/H$, the colors of $K/H$ are just the irreducible components of the complement of the open $B$-orbit, and the colors of $X$ are just the closures of the colors of $K/H$, so that the two sets $\Delta_X$ and $\Delta_{K/H}$ are naturally identified.

For any spherical $K$-variety $X$ one can also define another finite set, the set of spherical roots, usually denoted by $\Sigma_X$. Here we recall its definition only in the wonderful case. Suppose $X$ is the wonderful embedding of $K/H$. By definition $X$ contains a unique closed $K$-orbit, therefore every Borel subgroup of $K$ fixes in $X$ a unique point. Let us call $z$ the point fixed by $B^-$, the opposite of the Borel subgroup $B$. For all $K$-stable prime divisors $D_i$, let $\sigma_i$ be $T$-eigenvalue occurring in the normal space of $D_i$ at $z$
\[\frac{\mathrm T_zX}{\mathrm T_zD_i}.\]
Then the set of spherical roots is the set $\Sigma_X=\{\sigma_1,\ldots,\sigma_r\}$, also denoted by $\Sigma_{K/H}$. The spherical roots are linearly independent and the corresponding reflections
\[\gamma\mapsto\gamma-2\frac{(\sigma_i,\gamma)}{(\sigma_i,\sigma_i)}\sigma_i\]
generate a finite group of orthogonal transformations which is called the little Weyl group of $X$.
In our case, in which the center of $K$ acts trivially, the spherical roots are elements of $\mathbb NS$, that is, linear combinations with non-negative integer coefficients of simple roots.

The Picard group of a wonderful variety $X$ is freely generated by the equivalence classes of the colors of $X$. 
Expressing the classes of the $K$-stable divisors in terms of the basis given by the classes of colors
\[[D_i]=\sum_{D\in\Delta_{K/H}}c_{K/H}(D,\sigma_i)[D]\]
we get a $\mathbb Z$-bilinear pairing, which is also called Cartan pairing,
\[c_{K/H}\colon\mathbb Z\Delta_{K/H}\times\mathbb Z\Sigma_{K/H}\to\mathbb Z.\] 
It is known to satisfy quite strong restrictions, as follows.

For any simple root $\alpha\in S$, the set of colors moved by $\alpha$, which is denoted by $\Delta_{K/H}(\alpha)$, is the set of colors that are not stable under the action of the minimal parabolic subgroup $P_{\{\alpha\}}$. Any simple root $\alpha$ moves at most two colors, and more precisely there are exactly four cases:
\begin{itemize}
\item[Case {\bf p})] $\alpha$ moves no colors;
\item[Case {\bf a})] $\alpha$ moves two colors, this happens if and only if $\alpha\in\Sigma_{K/H}$, and in this case we have 
\begin{enumerate}
\item $\Delta_{K/H}(\alpha)=\{D\in\Delta_{K/H}:c_{K/H}(D,\alpha)=1\}$,
\item $c_{K/H}(D,\sigma)\leq1$ for all $D\in\Delta_{K/H}(\alpha)$ and $\sigma\in\Sigma_{K/H}$, 
\item $\sum_{D\in\Delta_{K/H}(\alpha)}c_{K/H}(D,\sigma)=\langle\alpha^\vee,\sigma\rangle$ for all $\sigma\in\Sigma_{K/H}$;
\end{enumerate}
\item[Case {\bf 2a})] $\alpha$ moves one color and $2\alpha\in\Sigma_{K/H}$, in this case if $D\in\Delta_{K/H}(\alpha)$ 
we have $c_{K/H}(D,\sigma)=\frac12\langle\alpha^\vee,\sigma\rangle$ for all $\sigma\in\Sigma_{K/H}$;  
\item[Case {\bf b})] $\alpha$ moves one color and $2\alpha\not\in\Sigma_{K/H}$, in this case if $D\in\Delta_{K/H}(\alpha)$
we have $c_{K/H}(D,\sigma)=\langle\alpha^\vee,\sigma\rangle$ for all $\sigma\in\Sigma_{K/H}$.
\end{itemize}

The set of simple roots moving no colors is denoted by $S^\mathrm p_{K/H}$. 

The set of colors $\Delta_{K/H}$ is a disjoint union of subsets $\Delta_{K/H}^\mathrm a$, $\Delta_{K/H}^\mathrm{2a}$, $\Delta_{K/H}^\mathrm b$ 
which consist of colors moved by simple roots of type (a), (2a), (b), respectively. The set $\Delta_{K/H}^\mathrm a$ is also denoted by $\mathrm A_{K/H}$.
\begin{itemize}
\item[Case {\bf a})] A color in $\mathrm A_{K/H}$ may be moved by several simple roots.
\item[Case {\bf 2a})] A color in $\Delta_{K/H}^\mathrm{2a}$ is moved by a unique simple root.
\item[Case {\bf b})] A color in $\Delta_{K/H}^\mathrm b$ may be moved by at most two simple roots, 
in this case two simple roots $\alpha$ and $\beta$ move the same color if and only if $\alpha$ and $\beta$ are orthogonal and $\alpha+\beta\in\Sigma_{K/H}$.
\end{itemize}

Notice that the full Cartan pairing $c_{K/H}\colon\mathbb Z\Delta_{K/H}\times\mathbb Z\Sigma_{K/H}\to\mathbb Z$ is determined by its restriction to
$\mathrm A_{K/H}\times\Sigma_{K/H}$. 

If $H$ is a wonderful subgroup of $K$, the triple $(S^\mathrm p_{K/H},\Sigma_{K/H},\mathrm A_{K/H})$, 
endowed with the map $c_{K/H}\colon \mathrm A_{K/H}\times\Sigma_{K/H}\to\mathbb Z$, 
is called the spherical system of $H$. 

\subsection{Luna diagrams} 

In Appendix \ref{B}, we present the spherical systems of the wonderful subgroups $H=\mathrm N_K(K_e)$ of $K$ by providing the sets of spherical roots $\Sigma_{K/H}$ and the Luna diagrams. The Luna diagram of a spherical system consists of the Dynkin diagram of $K$ decorated with some extra symbols from which one can read off all the data of the spherical system. Let us briefly explain how it works, here we only explain how to read off the missing data (the set $S^\mathrm p_{K/H}$ and the map $c_{K/H}\colon\mathrm A_{K/H}\times\Sigma_{K/H}$), see e.g.\ \cite{BL} for a complete description. 

Every circle (shadowed or not) represents a color. Circles corresponding to the same color are joined by a line. 
The colors moved by a simple root are close to the corresponding vertex of the Dynkin diagram: 
\begin{itemize}
\item[Case {\bf p})] no circle is placed in correspondence of the vertex,
\item[Case {\bf a})] two circles are placed one above and one below the vertex,
\item[Case {\bf 2a})] one circle is placed below the vertex,
\item[Case {\bf b})] one circle is placed around the vertex. 
\end{itemize}
Therefore, the set $S^\mathrm p$ is given by the vertices with no circles. 
It is worth saying that in general $S^\mathrm p$ is included in $\{\alpha\in S: \langle\alpha^\vee,\sigma\rangle=0\ \forall\,\sigma\in\Sigma\}$.   

To read off the map $c\colon\mathrm A\times\Sigma\to\mathbb Z$ one has to know that an arrow 
(it looks more like a pointer but it has a source and a target) 
starting from a circle $D$ above a vertex $\alpha$ and pointing towards a spherical root $\sigma$ non-orthogonal to $\alpha$
means that $c(D,\sigma)=-1$. 
Vice versa, the Luna diagram is organized in order that the colors $D$ corresponding to circles that lie above the vertices 
have $c(D,\sigma)\geq-1$ for all $\sigma\in\Sigma$, 
so if the there is no arrow starting from a circle $D$ above a vertex $\alpha$ and pointing towards a spherical root $\sigma$ non-orthogonal to $\alpha$
(with $D\not\in\Delta(\sigma)$) this means that $c(D,\sigma)=0$. These together with the properties of the Cartan pairing for colors of type (a), explained above, allows to recover the map $c\colon\mathrm A\times\Sigma\to\mathbb Z$.

The two colors moved by $\alpha\in S\cap\Sigma$ will be denoted by $D_\alpha^+$ and $D_\alpha^-$, the former refers to the circle placed above the vertex while the latter refers to the circle placed below. The color moved by a simple root $\alpha\not\in\Sigma$ will be denoted by $D_\alpha$.

As an example we show in detail how to recover the map $c\colon\mathrm A\times\Sigma\to\mathbb Z$ for the first case of the list where a non-empty set $\mathrm A_{K/H}$ occurs, the case \ref{sss:CCacy} with $q>2$.
The group $K$ is of type $\mathsf C_p\times\mathsf C_q$, with $p$ and $q$ greater than 2,
the set of spherical roots is 
\[\Sigma=\{\alpha_1,\alpha_2,\alpha'_1,\alpha'_2,\ \alpha'_2+2(\alpha'_3+\ldots+\alpha'_{q-1})+\alpha'_q\}\]
and the Luna diagram is as follows.
\[\begin{picture}(24600,2850)(-300,-1350)
\multiput(0,0)(13500,0){2}{
\multiput(0,0)(1800,0){3}{\usebox{\edge}}
\put(5400,0){\usebox{\susp}}
\put(9000,0){\usebox{\leftbiedge}}
\multiput(0,0)(1800,0){2}{\usebox{\aone}}
}
\put(3600,0){\usebox{\wcircle}}
\put(17100,0){\usebox{\gcircle}}
\multiput(0,-900)(15300,0){2}{\line(0,-1){450}}
\put(0,-1350){\line(1,0){15300}}
\multiput(0,900)(13500,0){2}{\line(0,1){600}}
\put(0,1500){\line(1,0){13500}}
\multiput(1800,900)(13500,0){2}{\line(0,1){300}}
\put(1800,1200){\line(1,0){11600}}
\put(13600,1200){\line(1,0){1700}}
\put(1800,600){\usebox{\tow}}
\put(13500,600){\usebox{\toe}}
\end{picture}\]
Here the set $S^\mathrm p$ is given by the simple roots $\alpha_i$ for all $4\leq i\leq p$ and $\alpha'_i$ for all $4\leq i\leq q$.
The elements of $\mathrm A$, i.e.\ the colors of type (a), are five: 
\[D_{\alpha_2}^-,\quad 
D_{\alpha_2}^+=D_{\alpha'_2}^+,\quad 
D_{\alpha_1}^-=D_{\alpha'_2}^-,\quad 
D_{\alpha_1}^+=D_{\alpha'_1}^+,\quad 
D_{\alpha'_1}^-.\] 
We know that for all colors $D$ of type (a) $c(D,\sigma)=1$ if $\sigma\in S$ and $D\in\Delta(\sigma)$, and $c(D,\sigma)\leq0$ otherwise. 
Therefore, let us show how to determine $c(D_{\alpha_2}^-,\sigma)$ for all $\sigma\in\Sigma$. First $D_{\alpha_2}^-\in\Delta(\alpha_2)$ then $c(D_{\alpha_2}^-,\alpha_2)=1$.
Since there is an arrow from $D_{\alpha_2}^+$ to $\alpha_1$ $c(D_{\alpha_2}^+,\alpha_1)=-1$, furthermore $c(D_{\alpha_2}^-,\alpha_1)+c(D_{\alpha_2}^+,\alpha_1)=\langle\alpha_2^\vee,\alpha_1\rangle=-1$, thus we have $c(D_{\alpha_2}^-,\alpha_1)=0$. The other spherical roots $\sigma$ are orthogonal to $\alpha_2$, so $c(D_{\alpha_2}^-,\sigma)+c(D_{\alpha_2}^+,\sigma)=0$,
if $c(D_{\alpha_2}^-,\sigma)$ is $<0$ then $c(D_{\alpha_2}^+,\sigma)$ must be $>0$ but this happens only if $D_{\alpha_2}^+\in\Delta(\sigma)$. Therefore, $c(D_{\alpha_2}^-,\alpha'_2)=-1$ while it is zero on the other two spherical roots $c(D_{\alpha_2}^-,\alpha'_1)=c(D_{\alpha_2}^-,\alpha'_2+2(\alpha'_3+\ldots+\alpha'_{q-1})+\alpha'_q)=0$. The entire map $c\colon\mathrm A\times\Sigma\to\mathbb Z$ is as follows.
\[\begin{array}{r|rrrrr}
 & \alpha_1 & \alpha_2 & \alpha'_1 & \alpha'_2 & \sigma_5 \\
\hline
D_{\alpha_2}^- &  0 &  1 &  0 & -1 &  0 \\
D_{\alpha_2}^+ & -1 &  1 &  0 &  1 &  0 \\
D_{\alpha_1}^- &  1 & -1 & -1 &  1 &  0 \\
D_{\alpha_1}^+ &  1 &  0 &  1 & -1 &  0 \\
D_{\alpha'_1}^- & -1 &  0 &  1 &  0 & -1 
\end{array}\]

\subsection{Operations on spherical systems}

Here we briefly recall the definition and the essential properties of some combinatorial operations on spherical systems which correspond to geometric operations on wonderful varieties, see e.g.\ \cite{BL} for some more details and references.

\subsubsection{Subsystems}

All (irreducible) $K$-subvarieties of a wonderful $K$-variety $X$ are wonderful, they are exactly the $K$-orbit closures of $X$, 
and are in correspondence with the subsets of $\Sigma_X$. 
If $D_1,\ldots,D_r$ are the $K$-stable prime divisors of $X$, 
recall that the spherical roots $\sigma_1,\ldots,\sigma_r$ are $T$-eigenvalues occurring respectively in the normal spaces of $D_i$ at $z$,
$\mathrm T_zX/\mathrm T_zD_i$. Therefore, every $K$-subvariety $X'$ of $X$ is the intersection of some $K$-stable prime divisors
\[X'=\bigcap_{i\in I}D_i\]
for some $I\subset\{1,\ldots,r\}$. Its spherical system is thus given by 
\begin{itemize}
\item $S^\mathrm p_{X'}=S^\mathrm p_X$, 
\item $\Sigma_{X'}=\{\sigma_i:i\not\in I\}$,
\item $\mathrm A_{X'}=\bigcup_{\alpha\in S\cap\Sigma_{X'}}\Delta_X(\alpha)$  
with the map $c_X$ restricted to $\mathbb Z\mathrm A_{X'}\times\mathbb Z\Sigma_{X'}$.
\end{itemize}

\subsubsection{Quotients}

Let $X_1$ and $X_2$ be the wonderful embeddings of $K/H_1$ and $K/H_2$, respectively.
If $H_1$ is included in $H_2$ with connected quotient $H_2/H_1$, 
there exists a surjective equivariant morphism from $X_1$ to $X_2$ with connected fibers.

In terms of sperical systems this is equivalent to an operation called quotient, as follows.

A subset $\Delta'$ of $\Delta_{X_1}$ is called distinguished if there exists a linear combination with positive coefficients
\[D'\in\sum_{D\in\Delta'}n_DD\]
such that $c_{X_1}(D',\sigma)\geq0$ for all $\sigma\in\Sigma_{X_1}$.

If $\Delta'$ is distinguished, the monoid 
\[(\mathbb N\Sigma_{X_1})/\Delta'=\{\sigma\in\mathbb N\Sigma_{X_1}:c_{X_1}(D,\sigma)=0\ \forall\,D\in\Delta'\}\] is known to be free \cite{Bra}.
Therefore, we can consider the following triple, which is called the quotient of the spherical system of $X_1$ by $\Delta'$:
\begin{itemize}
\item $S^\mathrm p_{X_1}/\Delta'=\{\alpha\in S: \Delta_{X_1}(\alpha)\subset\Delta'\}$,
\item $\Sigma_{X_1}/\Delta'$, the basis of $(\mathbb N\Sigma_{X_1})/\Delta'$,
\item $\mathrm A_{X_1}/\Delta'=\bigcup_{\alpha\in S\cap(\Sigma_{X_1}/\Delta')}\Delta_{X_1}(\alpha)$ 
endowed with the map $c_{X_1}$ restricted to $\mathbb Z(\mathrm A_{X_1}/\Delta')\times\mathbb Z(\Sigma_{X_1}/\Delta')$.
\end{itemize}

If $X_1$ and $X_2$ are wonderful $K$-varieties with a surjective equivariant morphism with connected fibers $\varphi\colon X_1\to X_2$,
then $\Delta'_{\varphi}=\{D\in\Delta_{X_1}: \varphi(D)=X_2\}$ is distinguished and the spherical system of $X_2$ is equal to the quotient of the spherical system of $X_1$ by $\Delta'_{\varphi}$.

If $X_1$ is a wonderful $K$-variety, every distinguished subset $\Delta'$ of $\Delta_{X_1}$ corresponds in this way to a surjective equivariant 
morphism with connected fibers onto a wonderful variety whose spherical system is equal to the quotient of the spherical system of $X_1$ by $\Delta'$.   

\subsubsection{Parabolic inductions}

Let $Q$ be a parabolic subgroup of $K$, with Levi decomposition $Q=L\,Q^\mathrm u$. 
A wonderful $K$-variety $X$ is said to be obtained by parabolic induction from the wonderful $L$-variety $Y$ if
\[X\cong K\times_Q Y,\]
where $Q^\mathrm u$ acts trivially on $Y$. 

Further, since $Y$ is a wonderful $L$-variety, the radical of $L$ acts trivially on $Y$, as well.   

Clearly, if the wonderful $K$-variety $X$ is obtained by parabolic induction from the wonderful embedding of $L/M$,
then $X$ is the wonderful embedding of $K/(M\,Q^\mathrm u)$. 

In terms of spherical systems this corresponds to the following situation. 

Assume that $Q$ contains $B^-$ and $L$ contains $T$, denote by $S_L$ the subset of $S$ generating the root subsystem of $L$.

The wonderful $K$-variety $X$ is obtained by parabolic induction from a wonderful $L$-variety $Y$ if and only if
\[S^\mathrm p_X\cup\{\mathrm{supp}\,\sigma:\ \forall\,\sigma\in\Sigma_X\}\subset S_L.\]
In this case, the spherical system of $Y$,
after the above inclusion, is equal to the triple $(S^\mathrm p_X,\Sigma_X,\mathrm A_X)$.

In plain words, the spherical system of $X$ is obtained from the spherical system of $Y$ 
by letting the extra simple roots in $S\smallsetminus S_L$ move one extra color each so that they are all of type (b).

\subsubsection{Localizations}

Let $Q$ be a parabolic subgroup of $K$, containing $B^-$, and let $Q=L\,Q^\mathrm u$ be its Levi decomposition, with $L$ containing $T$.
Denote by $L^\mathrm r$ the radical of $L$, and by $S_L$ the subset of $S$ generating the root subsystem of $L$.

Let $X$ be a wonderful $K$-variety. 
Consider the subset of $X$ of points fixed by $L^\mathrm r$ 
and take its connected component which contains $z$, the unique point fixed by $B^-$.
It is a wonderful $L$-variety $Y$ called $L$-localization of $X$. 
The spherical system of $Y$ is obtained from the spherical system of $X$ as follows:
\begin{itemize}
\item $S^\mathrm p_Y=S^\mathrm p_X\cap S_L$,
\item $\Sigma_Y=\{\sigma\in\Sigma_X:\mathrm{supp}\,\sigma\subset S_L\}$,
\item $\mathrm A_Y=\bigcup_{\alpha\in S_L\cap\Sigma_X}\Delta_X(\alpha)$
with the map $c_{X}$ restricted to $\mathbb Z\mathrm A_{Y}\times\mathbb Z\Sigma_{Y}$.
\end{itemize}
In this case the spherical system of $Y$ is said to be obtained from the spherical system of $X$ by localization in $S_L$. 

\subsection{Luna's classification of wonderful varieties}

Here we recall the statement of Luna's theorem of the classification of wonderful varieties, \cite{luna_typeA, CF, BP16}.

In our case the center of $K$ always acts trivially, so here we assume for convenience that $K$ is a semisimple complex algebraic group of adjoint type. Let $T$, $B$ and $S$ as above.

Every spherical root of any wonderful $K$-variety is the spherical root of a wonderful $K$-variety of rank 1, and the wonderful varieties of rank 1 are well-known. In particular, the set $\Sigma(K)$ of the spherical roots of all the wonderful $K$-varieties is finite and is described by the following. 

\begin{theorem}
Every spherical root $\sigma$ of any wonderful $K$-variety, for any semisimple complex algebraic group $K$ of adjoint type belongs to Table \ref{tab:sphroots}.
\end{theorem}

\begin{table}\caption{spherical roots}\label{tab:sphroots}
\[\begin{array}{ll}
\mbox{type of support} & \mbox{spherical root}  \\
\hline
\mathsf A_1 & \alpha \\
\mathsf A_1 & 2\alpha \\
\mathsf A_1\times \mathsf A_1 & \alpha+\alpha' \\
\mathsf A_m & \alpha_1+\ldots+\alpha_m \\
\mathsf A_3 & \alpha_1+2\alpha_2+\alpha_3 \\
\mathsf B_m & \alpha_1+\ldots+\alpha_m \\
\mathsf B_m & 2(\alpha_1+\ldots+\alpha_m) \\
\mathsf B_3 & \alpha_1+2\alpha_2+3\alpha_3 \\
\mathsf C_m & \alpha_1+2(\alpha_2+\ldots+\alpha_{m-1})+\alpha_m \\
\mathsf D_m & 2(\alpha_1+\ldots+\alpha_{m-2})+\alpha_{m-1}+\alpha_m \\ 
\mathsf F_4 & \alpha_1+2\alpha_2+3\alpha_3+2\alpha_4 \\
\mathsf G_2 & 2\alpha_1+\alpha_2 \\
\mathsf G_2 & 4\alpha_1+2\alpha_2 \\
\mathsf G_2 & \alpha_1+\alpha_2
\end{array}\]
\end{table}

There is an abstract notion of Luna spherical system, the following. 

\begin{definition}\label{def:sphsystem}
A triple $(S^\mathrm p, \Sigma, \mathrm A)$, where $S^\mathrm p$ is a subset of $S$, $\Sigma$ is a subset of $\Sigma(K)$ without proportional elements and $\mathrm A$ is a finite set endowed with a map $c\colon \mathrm A\times\Sigma\to\mathbb Z$, is called a \textit{spherical $K$-system} if the following axioms hold.
\begin{itemize}
\item[$\mathrm A1)$] For all $D\in\mathrm A$, $c(D,\sigma)\leq1$ for all $\sigma\in\Sigma$, and $c(D,\sigma)=1$ only if $\sigma\in S$.
\item[$\mathrm A2)$] For all $\alpha\in S\cap\Sigma$, $\{D\in\mathrm A:c(D,\alpha)=1\}$ has cardinality 2 and 
for all $\sigma\in\Sigma$
\[\sum_{D\,:\,c(D,\alpha)=1}c(D,\sigma)=\langle\alpha^\vee,\sigma\rangle.\]
\item[$\mathrm A3$)] For all $D\in\mathrm A$ there exists $\alpha\in S\cap\Sigma$ with $c(D,\alpha)=1$.
\item[$\Sigma1$)] For all $\alpha\in S$ such that $2\alpha\in\Sigma$, $\frac12\langle\alpha^\vee,\sigma\rangle\in\mathbb Z_{\leq0}$ for all $\sigma\in\Sigma\smallsetminus\{2\alpha\}$.
\item[$\Sigma2$)] For all $\alpha$ and $\beta$ in $S$ such that $\alpha$ and $\beta$ are orthogonal and $\alpha+\beta\in\Sigma$, $\langle\alpha^\vee,\sigma\rangle=\langle\beta^\vee,\sigma\rangle$ for all $\sigma\in\Sigma$.
\item[$\mathrm S$)] For all $\sigma\in\Sigma$,
\begin{itemize}
\item if $\sigma=\alpha_1+\ldots+\alpha_m$ with $\mathrm{supp}\,\sigma$ of type $\mathsf B_m$,
\[\{\alpha_2,\ldots,\alpha_{m-1}\}\subset S^\mathrm p\subset \{\alpha\in S:\langle\alpha^\vee,\sigma\rangle=0\},\]
\item if $\sigma=\alpha_1+2(\alpha_2+\ldots+\alpha_{m-1})+\alpha_m$ with $\mathrm{supp}\,\sigma$ of type $\mathsf C_m$,
\[\{\alpha_3,\ldots,\alpha_{m}\}\subset S^\mathrm p\subset \{\alpha\in S:\langle\alpha^\vee,\sigma\rangle=0\},\]
\item otherwise
\[\{\alpha\in\mathrm{supp}\,\sigma:\langle\alpha^\vee,\sigma\rangle=0\}\subset S^\mathrm p\subset \{\alpha\in S:\langle\alpha^\vee,\sigma\rangle=0\}.\]
\end{itemize}
\end{itemize}
\end{definition}

The following is known as Luna's theorem of classification of wonderful varieties.

\begin{theorem}\label{thm:Luna}
The map which associates to a wonderful $K$-variety $X$ its spherical system $(S^\mathrm p_X,\Sigma_X,\mathrm A_X)$ is a bijection between the set of wonderful $K$-varieties up to equivariant isomorphism and the set of spherical $K$-systems. 
\end{theorem} 

\subsection{The spherical systems of the list} 

Here we show that the spherical systems given in the tables of Appendix~\ref{B} are indeed the spherical sytems associated with $\mathrm N_K(K_e)$, the normalizers of the centralizers of the representatives $e$ given in Appendix~\ref{A}.

For all $K$, every spherical system given in the tables satisfies the axioms of Definition~\ref{def:sphsystem}, 
so by Theorem~\ref{thm:Luna} it is equal to the spherical system associated with a (uniquely determined up to conjugation) wonderful subgroup of $K$. Here we compute this wonderful subgroup for any spherical system of Appendix~\ref{B}.

\subsubsection{Parabolic inductions and trivial factors}	\label{sss:parabolic induction}

In all the spherical systems of Appendix~\ref{B} the set $(\supp\Sigma)\cup S^\mathrm p$ is properly contained in $S$, 
therefore the corresponding wonderful $K$-varieties $X$ can be obtained by parabolic induction from wonderful $L$-varieties $Y$, where $L$ is properly contained in $K$. We set $S_L=(\supp\Sigma)\cup S^\mathrm p$.

Furthermore, in general $\supp\Sigma$ and $S^\mathrm
p\smallsetminus\supp\Sigma$ are orthogonal, so that $L$ is a direct product
$L_1\times L_2$, where $S_{L_1}=\supp\Sigma$ and $S_{L_2}=S^\mathrm
p\smallsetminus\supp\Sigma$, with $L_2$ acting trivially on $Y$.  In many
cases $S^\mathrm p\smallsetminus\supp\Sigma$ is non-empty.

Notice that the above decomposition $L=L_1\times L_2$ is not uniquely determined, but here the center of $L$ acts trivially on $Y$, so we do not care of which part of the center of $L$ is contained in the two factors $L_1$ and $L_2$. 

In the following we will compute, in all our cases, the wonderful subgroups associated with the spherical systems obtained by localization in $S_{L_1}=\supp\Sigma$.

\subsubsection{Trivial cases} \label{sss:trivial cases}

In the cases \ref{sss:symm1} ($r=1$), \ref{sss:symm2} ($r=1$), \ref{sss:Dao} ($r=1$), \ref{sss:symm4.1} ($r=1$), \ref{sss:trvial1}, \ref{sss:trvial2}, \ref{sss:BDaa} ($r=1$), \ref{sss:symm8} ($r=1$) and \ref{sss:DDaa} ($r=1$) the set $\Sigma$ is empty, so the spherical system obtained by localization in $\supp\Sigma$ is trivial. More explicitly, the parabolic subgroups $Q$ of $K$ given in Appendix~\ref{A} are the wonderful subgroups associated with the given spherical $K$-systems. 

\subsubsection{Symmetric cases} \label{sss:symmetric cases}

In the cases \ref{sss:symm1}, \ref{sss:symm2}, \ref{sss:Dao}, \ref{sss:symm4.1}, \ref{sss:CCaac} ($q=1$), \ref{sss:CCaac2} ($p=1$), \ref{sss:BDaa}, \ref{sss:BDady} ($r=0$), \ref{sss:BDbay} ($r=0$),  \ref{sss:symm8}, \ref{sss:BBaby} ($r=0$), \ref{sss:BBbay} ($r=0$), \ref{sss:DDaa}, \ref{sss:DDady} ($r=0$) and \ref{sss:DDday} ($r=0$) the spherical system obtained by localization in $\supp\Sigma$ is the spherical system of a symmetric subgroup $\mathrm N_{L_1}(L_1^\theta)$ of $L_1$,  where $L_1^\theta$ is the fixed point subgroup of an involution $\theta$ of $L_1$.

The wonderful symmetric subgroups and their spherical systems are well-known, see e.g.\ \cite{BP15}.      
More precisely, 
\begin{itemize}
\item in the case \ref{sss:symm1} we get the case 6 of \cite{BP15}; 
\item in the cases \ref{sss:symm2}, \ref{sss:Dao}, \ref{sss:BDbay} ($r=0,\ p=1$), \ref{sss:BBaby} ($r=0,\ q=1$) and \ref{sss:BBbay} ($r=0,\ p=1$) we get the case 5 of \cite{BP15};
\item in the cases \ref{sss:symm4.1}, \ref{sss:CCaac} ($q=1$), \ref{sss:CCaac2} ($p=1$), \ref{sss:BDaa},  \ref{sss:BDady} ($r=0,\ q=2$), \ref{sss:symm8}, \ref{sss:DDaa}, \ref{sss:DDady} ($r=0,\ q=2$) and 
\ref{sss:DDday} ($r=0,\ p=2$) we get the case 2 of \cite{BP15};
\item in the cases \ref{sss:BDbay} ($r=0,\ p>1$), \ref{sss:BBaby} ($r=0,\ q>1$) and \ref{sss:BBbay} ($r=0,\ p>1$)  we get the case 9 of \cite{BP15};
\item in the cases \ref{sss:BDady} ($r=0,\ q>2$), \ref{sss:DDady} ($r=0,\ q>2$) and \ref{sss:DDday} ($r=0,\ p>2$)  we get the case 15 of \cite{BP15}.
\end{itemize}

\subsubsection{Other reductive cases} \label{sss:reductive}

In the cases \ref{sss:CCaac} ($q>1$), \ref{sss:CCaac2} ($p>1$), \ref{sss:CCabx} and \ref{sss:CCbax} the spherical system obtained by localization in $\supp\Sigma$ is the spherical system of a wonderful reductive (but not symmetric) subgroup of $L_1$. More precisely,
\begin{itemize}
\item in the cases \ref{sss:CCaac} ($q>1$) and \ref{sss:CCaac2} ($p>1$) we get the case 42 of \cite{BP15};
\item in the cases \ref{sss:CCabx} and \ref{sss:CCbax} we get the case 46 ($p=5$) of \cite{BP15}.
\end{itemize}

\subsubsection{Morphisms of type $\mathcal L$} \label{sss:quotients}

Notice that in all the above cases the Levi subgroup $L$ such that $S_L=(\supp\Sigma)\cup S^\mathrm p$ is equal to $K_h$, the centralizer of $h$ given in the list of Appendix~\ref{A}. In the remaining cases this is no longer true, but we have the following situation. 

In the remaining cases, 
\ref{sss:CCacy}, \ref{sss:CCcay},  
\ref{sss:BDady} ($r>0$),  
\ref{sss:BDbay} ($r>0$),  
\ref{sss:BBaby} ($r>0$), \ref{sss:BBbay} ($r>0$),  
\ref{sss:DDady} ($r>0$) and  
\ref{sss:DDday} ($r>0$),  
the given spherical $K$-system $(S^\mathrm p,\Sigma,\mathrm A)$ admits a distinguished set of colors $\Delta'$ such that the corresponding quotient 
\[(S^\mathrm p/\Delta',\Sigma/\Delta',\mathrm A/\Delta')\] 
is the spherical system of a wonderful $K$-variety which is obtained by parabolic induction from a wonderful $K_h$-variety. Indeed, $S_{K_h}= (\supp(\Sigma/\Delta'))\cup (S^\mathrm p/\Delta')$.  

Such distinguished set of colors $\Delta'$ is minimal, that is, does not contain any proper non-empty distinguished subset. Moreover, the corresponding quotient has higher defect, which means the following.

The defect of a spherical system is defined as the non-negative integer given by the difference between the number of colors and the number of spherical roots.

In all our cases, we have
\begin{equation}\label{eqn:defect}
\mathrm{card}(\Delta\smallsetminus\Delta')-\mathrm{card}(\Sigma/\Delta')>\mathrm{card}\,\Delta-\mathrm{card}\,\Sigma.
\end{equation}  
Therefore, the set $\Delta'$ corresponds to a minimal surjective equivariant morphism with connected fibers of type $\mathcal L$ in the sense of \cite[Proposition~2.3.5]{BL}. In particular, the minimal quotients of higher defect have been studied in \cite[Section~5.3]{BP}. Let us recall their description.

Let $H_1$ be the wonderful subgroup associated with the spherical $K$-system $(S^\mathrm p,\Sigma,\mathrm A)$,
let $\Delta'$ be a distinguished subset satisfying the condition \eqref{eqn:defect} and let $H_2$ be the wonderful subgroup of $K$ associated with the quotient of $(S^\mathrm p,\Sigma,\mathrm A)$ by $\Delta'$. We can assume $H_1\subset H_2$. Recall that the quotient $H_2/H_1$ is connected. 

Under the condition \eqref{eqn:defect} we have that $H_1^\mathrm u$ is properly contained in $H_2^\mathrm u$. Take Levi decompositions $H_1=L_{H_1}H_1^\mathrm u$ and $H_2=L_{H_2}H_2^\mathrm u$ with $L_{H_1}\subset L_{H_2}$, then $\mathrm{Lie}\,H_2^\mathrm u/\mathrm{Lie}\,H_1^\mathrm u$ is a simple $L_{H_1}$-module and $L_{H_1}$ and $L_{H_2}$ differ only by their connected center. 

The defect of a spherical system is equal to the dimension of the connected center of the associated wonderful subgroup, so the codimension of $L_{H_1}$ in $L_{H_2}$ is equal to
\[
d=\mathrm{card}(\Delta\smallsetminus\Delta')-\mathrm{card}(\Sigma/\Delta')-(\mathrm{card}\,\Delta-\mathrm{card}\,\Sigma).
\]

The quotient $\mathrm{Lie}\,H_2^\mathrm u/\mathrm{Lie}\,H_1^\mathrm u$ can be described as follows. There exist $d+1$ $L_{H_2}$-submodules of $\mathrm{Lie}\,H_2^\mathrm u$, $W_0,\ldots,W_d$, isomorphic as $L_{H_1}$-modules but not as $L_{H_2}$-modules. Denoting by $V$ the $L_{H_2}$-complement of $W_0\oplus \ldots\oplus W_d$ in $\mathrm{Lie}\,H_2^\mathrm u$, as $L_{H_1}$-module, 
\[\mathrm{Lie}\,H_1^\mathrm u=W\oplus V,\]
where $W$ is a co-simple $L_{H_1}$-submodule of $W_0\oplus \ldots\oplus W_d$ which projects non-trivially on every summand $W_0,\ldots,W_d$. 

As said above, in our cases we always have $H_2\subset Q$, with $Q=K_h\,Q^\mathrm u$ given in the list of Appendix~\ref{A}, $L_{H_2}\subset K_h$ and $H_2^\mathrm u=Q^\mathrm u$. 

One can say something more about the inclusion of the $W_0,\ldots,W_d$ in $\mathrm{Lie}\,Q^\mathrm u$. One has to consider the set $S_{\Delta'}$, whose general definition involves the notion of external negative color (see \cite[Section~2.3.5]{BL} and \cite[Section~5.2]{BP}). Without going into technical details, in our cases it holds
\[S_{\Delta'}=(\supp\Sigma) \smallsetminus (\supp(\Sigma/\Delta')).\] 
Moreover, $\mathrm{card}\,S_{\Delta'}=d+1$, say $S_{\Delta'}=\{\beta_0,\ldots,\beta_d\}$.
Assuming $Q$ contains $B^-$, we have that $W_0,\ldots,W_d$ are respectively included in the simple $L$-submodules $V(-\beta_0),\ldots,V(-\beta_d)$ containing the root spaces of $-\beta_0,\ldots,-\beta_d$. 

In our cases the integer $d+1$, the cardinality of $S_{\Delta'}$, is always equal to $2$ or $3$. 

In the following, for all the remaining cases, we describe the quotient of $(S^\mathrm p,\Sigma,\mathrm A)$ by $\Delta'$, and $L_{H_2}$ in $K_h$. The knowledge of $S_{\Delta'}$ will be enough to uniquely determine the modules $W_0,\ldots,W_d$. 

\begin{remark}
Actually, the results contained in \cite{BP} allow to reduce the computation of the wonderful subgroup associated with a spherical system to the computation of the wonderful subgroups associated with somewhat smaller spherical systems. In particular, Section~5.3 therein allows to reduce the computation of the wonderful subgroup associated with a spherical system with a quotient of higher defect to the computation of the wonderful subgroups associated with some spherical subsystems. Moreover, many of the spherical systems under consideration have a tail, see Section~6 therein, and these cases can also be reduced to some smaller cases. Similar general considerations could be done for the cases obtained by \lq\lq collapsing\rq\rq\ the tails. We prefer to avoid as far as possible the technicalities and give a direct explicit description of our wonderful subgroups even if they are somewhat already known. 
\end{remark}

\subsubsection{Type $\mathsf B$}	\label{sss:typeB}

\paragraph{\textbf{a) Tail case.}}
Localizing the spherical systems of the cases 
\ref{sss:BDbay} ($0<r<p$), \ref{sss:BBaby} ($0<r<q$) and \ref{sss:BBbay} ($0<r<p$)
in $\supp\Sigma$ we obtain the following spherical system,
which we label as $\mathsf{a^y}(s,s)+\mathsf b'(t)$, 
for a group of semisimple type $\mathsf A_s\times\mathsf B_{s+t}$ with $t\geq1$.
\begin{itemize}
\item[]$S^\mathrm p=\{\alpha'_{s+2},\ldots,\alpha'_{s+t}\}$.
\item[]$\Sigma=\{\alpha_1,\ldots,\alpha_s,\ \alpha'_1,\ldots,\alpha'_s,\ 2(\alpha'_{s+1}+\ldots+\alpha'_{s+t})\}$.
\item[]$\mathrm A=\{D_1,\ldots,D_{2s+1}\}$ 
with $\Delta=\mathrm A\cup\{D_{2s+2}\}$ and full Cartan pairing as follows:
\begin{itemize}
\item[]$\alpha_1=D_1+D_2-D_3$,
\item[]$\alpha_i=-D_{2i-2}+D_{2i-1}+D_{2i}-D_{2i+1}$ for $2\leq i\leq s$,
\item[]$\alpha'_i=-D_{2i-1}+D_{2i}+D_{2i+1}-D_{2i+2}$ for $1\leq i\leq s$,
\item[]$2(\alpha'_{s+1}+\ldots+\alpha'_{s+t})=-2D_{2s+1}+2D_{2s+2}$.
\end{itemize}
\end{itemize}

If $t=1$ the Luna diagram is as follows,
\[\begin{picture}(15900,3600)(-300,-1800)
\multiput(0,0)(8100,0){2}{
\put(0,0){\usebox{\edge}}
\put(1800,0){\usebox{\susp}}
\multiput(0,0)(1800,0){2}{\usebox{\aone}}
\put(5400,0){\usebox{\aone}}
}
\put(9900,0){\usebox{\susp}}
\put(13500,0){\usebox{\rightbiedge}}
\put(15300,0){\usebox{\aprime}}
\multiput(0,900)(8100,0){2}{\line(0,1){900}}
\put(0,1800){\line(1,0){8100}}
\multiput(1800,900)(8100,0){2}{\line(0,1){600}}
\put(1800,1500){\line(1,0){6200}}
\put(8200,1500){\line(1,0){1700}}
\multiput(5400,900)(8100,0){2}{\line(0,1){300}}
\put(5400,1200){\line(1,0){2600}}
\put(8200,1200){\line(1,0){1600}}
\put(10000,1200){\line(1,0){3500}}
\multiput(1800,-900)(6300,0){2}{\line(0,-1){900}}
\put(1800,-1800){\line(1,0){6300}}
\multiput(3600,-1500)(0,300){3}{\line(0,1){150}}
\put(3600,-1500){\line(1,0){4400}}
\put(8200,-1500){\line(1,0){1700}}
\put(9900,-1500){\line(0,1){600}}
\put(5400,-900){\line(0,-1){300}}
\put(5400,-1200){\line(1,0){2600}}
\put(8200,-1200){\line(1,0){1600}}
\put(10000,-1200){\line(1,0){1700}}
\multiput(11700,-1200)(0,300){2}{\line(0,1){150}}
\multiput(0,600)(1800,0){2}{\usebox{\toe}}
\multiput(9900,600)(3600,0){2}{\usebox{\tow}}
\end{picture}\]
while if $t>1$ it is as follows,
\[\begin{picture}(22800,3600)(-300,-1800)
\multiput(0,0)(8100,0){2}{
\put(0,0){\usebox{\edge}}
\put(1800,0){\usebox{\susp}}
\multiput(0,0)(1800,0){2}{\usebox{\aone}}
\put(5400,0){\usebox{\aone}}
}
\multiput(9900,0)(7200,0){2}{\usebox{\susp}}
\multiput(13500,0)(1800,0){2}{\usebox{\edge}}
\put(20700,0){\usebox{\rightbiedge}}
\put(15300,0){\usebox{\gcircletwo}}
\multiput(0,900)(8100,0){2}{\line(0,1){900}}
\put(0,1800){\line(1,0){8100}}
\multiput(1800,900)(8100,0){2}{\line(0,1){600}}
\put(1800,1500){\line(1,0){6200}}
\put(8200,1500){\line(1,0){1700}}
\multiput(5400,900)(8100,0){2}{\line(0,1){300}}
\put(5400,1200){\line(1,0){2600}}
\put(8200,1200){\line(1,0){1600}}
\put(10000,1200){\line(1,0){3500}}
\multiput(1800,-900)(6300,0){2}{\line(0,-1){900}}
\put(1800,-1800){\line(1,0){6300}}
\multiput(3600,-1500)(0,300){3}{\line(0,1){150}}
\put(3600,-1500){\line(1,0){4400}}
\put(8200,-1500){\line(1,0){1700}}
\put(9900,-1500){\line(0,1){600}}
\put(5400,-900){\line(0,-1){300}}
\put(5400,-1200){\line(1,0){2600}}
\put(8200,-1200){\line(1,0){1600}}
\put(10000,-1200){\line(1,0){1700}}
\multiput(11700,-1200)(0,300){2}{\line(0,1){150}}
\multiput(0,600)(1800,0){2}{\usebox{\toe}}
\multiput(9900,600)(3600,0){2}{\usebox{\tow}}
\end{picture}\]
but the combinatorics is the same, so from now on we just report the diagram for $t>1$.

Consider the quotient by $\Delta'=\{D_{2i}:1\leq i\leq s\}$.

$\Sigma/\Delta'=\{\alpha_2+\alpha'_1,\ldots,\alpha_s+\alpha'_{s-1},\ 2(\alpha'_{s+1}+\ldots+\alpha'_{s+t})\}$.
\[\begin{picture}(22800,2850)(-300,-1350)
\multiput(0,0)(8100,0){2}{
\put(0,0){\usebox{\edge}}
\put(1800,0){\usebox{\susp}}
\multiput(0,0)(1800,0){2}{\usebox{\wcircle}}
\put(5400,0){\usebox{\wcircle}}
}
\multiput(9900,0)(7200,0){2}{\usebox{\susp}}
\multiput(13500,0)(1800,0){2}{\usebox{\edge}}
\put(20700,0){\usebox{\rightbiedge}}
\put(15300,0){\usebox{\gcircletwo}}
\multiput(1800,-300)(6300,0){2}{\line(0,-1){1050}}
\put(1800,-1350){\line(1,0){6300}}
\multiput(3600,-1050)(0,300){3}{\line(0,1){150}}
\put(3600,-1050){\line(1,0){4400}}
\put(8200,-1050){\line(1,0){1700}}
\put(9900,-1050){\line(0,1){750}}
\put(5400,-300){\line(0,-1){450}}
\put(5400,-750){\line(1,0){2600}}
\put(8200,-750){\line(1,0){1600}}
\put(10000,-750){\line(1,0){1700}}
\multiput(11700,-750)(0,300){2}{\line(0,1){150}}
\end{picture}\]
It is a spherical system obtained by parabolic induction from the direct product of case 2 and the rank one case 9 (resp.\ the rank one case 4) if $t>1$ (resp.\ $t=1$), the labels referring to \cite{BP15}.

We have $S_{\Delta'}=\{\alpha_1,\alpha'_s\}$.

\paragraph{\textbf{b) Collapsed tail.}}
Localizing the spherical systems of the cases 
\ref{sss:BDbay} ($r=p$), \ref{sss:BBaby} ($r=q$) and \ref{sss:BBbay} ($r=p$)
in $\supp\Sigma$ we obtain the following spherical system,
which is labeled as $\mathsf{ab^y}(s,s)$ or S-6 in \cite{Bra}, 
for a group of semisimple type $\mathsf A_s\times\mathsf B_s$.
\begin{itemize}
\item[]$S^\mathrm p=\emptyset$.
\item[]$\Sigma=\{\alpha_1,\ldots,\alpha_s,\ \alpha'_1,\ldots,\alpha'_s\}$.
\item[]$\mathrm A=\{D_1,\ldots,D_{2s+1}\}=\Delta$ 
with Cartan pairing as follows:
\begin{itemize}
\item[]$\alpha_1=D_1+D_2-D_3$,
\item[]$\alpha_i=-D_{2i-2}+D_{2i-1}+D_{2i}-D_{2i+1}$ for $2\leq i\leq s$,
\item[]$\alpha'_i=-D_{2i-1}+D_{2i}+D_{2i+1}-D_{2i+2}$ for $1\leq i\leq s-2$,
\item[]$\alpha'_{s-1}=-D_{2s-3}+D_{2s-2}+D_{2s-1}-D_{2s}-D_{2s+1}$,
\item[]$\alpha'_s=-D_{2s-1}+D_{2s}+D_{2s+1}$.
\end{itemize}
\end{itemize}

The Luna diagram is as follows.
\[\begin{picture}(17700,4200)(-300,-2100)
\put(0,0){\usebox{\edge}}
\put(1800,0){\usebox{\susp}}
\put(5400,0){\usebox{\edge}}
\put(9900,0){\usebox{\edge}}
\put(11700,0){\usebox{\susp}}
\put(15300,0){\usebox{\rightbiedge}}
\multiput(0,0)(9900,0){2}{\multiput(0,0)(5400,0){2}{\multiput(0,0)(1800,0){2}{\usebox{\aone}}}}
\multiput(0,900)(9900,0){2}{\line(0,1){1200}}
\put(0,2100){\line(1,0){9900}}
\multiput(1800,900)(9900,0){2}{\line(0,1){900}}
\put(1800,1800){\line(1,0){8000}}
\put(10000,1800){\line(1,0){1700}}
\multiput(5400,900)(9900,0){2}{\line(0,1){600}}
\put(5400,1500){\line(1,0){4400}}
\put(10000,1500){\line(1,0){1600}}
\put(11800,1500){\line(1,0){3500}}
\multiput(7200,900)(9900,0){2}{\line(0,1){300}}
\put(7200,1200){\line(1,0){2600}}
\put(10000,1200){\line(1,0){1600}}
\put(11800,1200){\line(1,0){3400}}
\put(15400,1200){\line(1,0){1700}}
\multiput(1800,-900)(8100,0){2}{\line(0,-1){1200}}
\put(1800,-2100){\line(1,0){8100}}
\multiput(3600,-1800)(0,300){3}{\line(0,1){150}}
\put(3600,-1800){\line(1,0){6200}}
\put(10000,-1800){\line(1,0){1700}}
\put(11700,-1800){\line(0,1){900}}
\put(5400,-900){\line(0,-1){600}}
\put(5400,-1500){\line(1,0){4400}}
\put(10000,-1500){\line(1,0){1600}}
\put(11800,-1500){\line(1,0){1700}}
\multiput(13500,-1500)(0,300){2}{\line(0,1){150}}
\multiput(7200,-900)(8100,0){2}{\line(0,-1){300}}
\put(7200,-1200){\line(1,0){2600}}
\put(10000,-1200){\line(1,0){1600}}
\put(11800,-1200){\line(1,0){1600}}
\put(13600,-1200){\line(1,0){1700}}
\multiput(0,600)(1800,0){2}{\usebox{\toe}}
\put(5400,600){\usebox{\toe}}
\put(11700,600){\usebox{\tow}}
\multiput(15300,600)(1800,0){2}{\usebox{\tow}}
\end{picture}\]

Consider the quotient by $\Delta'=\{D_{2i}:1\leq i\leq s\}$.

$\Sigma/\Delta'=\{\alpha_2+\alpha'_1,\ldots,\alpha_s+\alpha'_{s-1}\}$.
\[\begin{picture}(17700,2550)(-300,-1650)
\put(0,0){\usebox{\edge}}
\put(1800,0){\usebox{\susp}}
\put(5400,0){\usebox{\edge}}
\put(9900,0){\usebox{\edge}}
\put(11700,0){\usebox{\susp}}
\put(15300,0){\usebox{\rightbiedge}}
\multiput(0,0)(9900,0){2}{\multiput(0,0)(5400,0){2}{\multiput(0,0)(1800,0){2}{\usebox{\wcircle}}}}
\multiput(1800,-300)(8100,0){2}{\line(0,-1){1350}}
\put(1800,-1650){\line(1,0){8100}}
\multiput(3600,-1350)(0,300){3}{\line(0,1){150}}
\put(3600,-1350){\line(1,0){6200}}
\put(10000,-1350){\line(1,0){1700}}
\put(11700,-1350){\line(0,1){1050}}
\put(5400,-300){\line(0,-1){750}}
\put(5400,-1050){\line(1,0){4400}}
\put(10000,-1050){\line(1,0){1600}}
\put(11800,-1050){\line(1,0){1700}}
\multiput(13500,-1050)(0,300){2}{\line(0,1){150}}
\multiput(7200,-300)(8100,0){2}{\line(0,-1){450}}
\put(7200,-750){\line(1,0){2600}}
\put(10000,-750){\line(1,0){1600}}
\put(11800,-750){\line(1,0){1600}}
\put(13600,-750){\line(1,0){1700}}
\end{picture}\]
It is a spherical system obtained by parabolic induction from the case 2 of \cite{BP15}. 
We have $S_{\Delta'}=\{\alpha_1,\alpha'_s\}$.

\subsubsection{Type $\mathsf C$} \label{sss:typeC}

\paragraph{\textbf{a) Tail case.}}
Localizing the spherical systems of the cases 
\ref{sss:CCacy} ($q>2$) and \ref{sss:CCcay} ($p>2$)
in $\supp\Sigma$ we obtain the following spherical system,
which we label as $\mathsf{a^y}(2,2)+\mathsf c(t)$, 
for a group of semisimple type $\mathsf A_2\times\mathsf C_{t+1}$ with $t\geq2$.
\begin{itemize}
\item[]$S^\mathrm p=\{\alpha'_4,\ldots,\alpha'_{t+1}\}$.
\item[]$\Sigma=\{\alpha_1,\alpha_2,\ \alpha'_1,\alpha'_2,\ \alpha'_2+2(\alpha'_3+\ldots+\alpha'_{t})+\alpha'_{t+1}\}$.
\item[]$\mathrm A=\{D_1,\ldots,D_5\}$ 
with $\Delta=\mathrm A\cup\{D_6\}$ and full Cartan pairing as follows:
\begin{itemize}
\item[]$\alpha_1=-D_2+D_3+D_4-D_5$,
\item[]$\alpha_2=D_1+D_2-D_3$,
\item[]$\alpha'_1=-D_3+D_4+D_5$,
\item[]$\alpha'_2=-D_1+D_2+D_3-D_4-D_6$,
\item[]$\sigma_5=-D_5+D_6$.
\end{itemize}
\end{itemize}

The Luna diagram is as follows.
\[\begin{picture}(15600,2850)(-300,-1350)
\put(0,0){\usebox{\edge}}
\multiput(0,0)(1800,0){2}{\usebox{\aone}}
\put(4500,0){
\multiput(0,0)(1800,0){3}{\usebox{\edge}}
\put(5400,0){\usebox{\susp}}
\put(9000,0){\usebox{\leftbiedge}}
\multiput(0,0)(1800,0){2}{\usebox{\aone}}
}
\put(8100,0){\usebox{\gcircle}}
\multiput(0,-900)(6300,0){2}{\line(0,-1){450}}
\put(0,-1350){\line(1,0){6300}}
\multiput(0,900)(4500,0){2}{\line(0,1){600}}
\put(0,1500){\line(1,0){4500}}
\multiput(1800,900)(4500,0){2}{\line(0,1){300}}
\put(1800,1200){\line(1,0){2600}}
\put(4600,1200){\line(1,0){1700}}
\put(1800,600){\usebox{\tow}}
\put(4500,600){\usebox{\toe}}
\end{picture}\]

Consider the quotient by $\Delta'=\{D_2,D_4\}$.

$\Sigma/\Delta'=\{\alpha_1+\alpha'_2,\ \alpha'_2+2(\alpha'_3+\ldots+\alpha'_{t})+\alpha'_{t+1}\}$.
\[\begin{picture}(15600,1650)(-300,-750)
\put(0,0){\usebox{\edge}}
\multiput(0,0)(1800,0){2}{\usebox{\wcircle}}
\put(4500,0){
\multiput(0,0)(1800,0){3}{\usebox{\edge}}
\put(5400,0){\usebox{\susp}}
\put(9000,0){\usebox{\leftbiedge}}
\multiput(0,0)(1800,0){2}{\usebox{\wcircle}}
}
\put(8100,0){\usebox{\gcircle}}
\multiput(0,-300)(6300,0){2}{\line(0,-1){450}}
\put(0,-750){\line(1,0){6300}}
\end{picture}\]
It is a spherical system obtained by parabolic induction from the case 42 of \cite{BP15},
already considered in Section~\ref{sss:reductive}. 
We have $S_{\Delta'}=\{\alpha_2,\alpha'_1\}$.

\paragraph{\textbf{b) Collapsed tail.}}
Localizing the spherical systems of the cases 
\ref{sss:CCacy} ($q=2$) and \ref{sss:CCcay} ($p=2$)
in $\supp\Sigma$ we obtain the spherical system $\mathsf{ab^y}(2,2)$ 
for a group of semisimple type $\mathsf A_2\times\mathsf B_2$, 
a particular case of the spherical system obtained above in Section~\ref{sss:typeB}. 

\subsubsection{Type $\mathsf D$} \label{sss:typeD}

\paragraph{\textbf{a) Tail case.}}
Localizing the spherical systems of the cases 
\ref{sss:BDady} ($0<r<q-1$), \ref{sss:DDady} ($0<r<q-1$) and \ref{sss:DDday} ($0<r<p-1$)
in $\supp\Sigma$ we obtain the following spherical system 
for a group of semisimple type $\mathsf A_s\times\mathsf D_{s+t}$ with $t\geq2$.
\begin{itemize}
\item[]$S^\mathrm p=\{\alpha'_{s+2},\ldots,\alpha'_{s+t}\}$.
\item[]$\Sigma=\{\alpha_1,\ldots,\alpha_s,\ \alpha'_1,\ldots,\alpha'_s,\ 2(\alpha'_{s+1}+\ldots+\alpha'_{s+t-2})+\alpha'_{s+t-1}+\alpha'_{s+t}\}$.
\item[]$\mathrm A=\{D_1,\ldots,D_{2s+1}\}$ 
with $\Delta=\mathrm A\cup\{D_{2s+2}\}$ and full Cartan pairing as follows:
\begin{itemize}
\item[]$\alpha_1=D_1+D_2-D_3$,
\item[]$\alpha_i=-D_{2i-2}+D_{2i-1}+D_{2i}-D_{2i+1}$ for $2\leq i\leq s$,
\item[]$\alpha'_i=-D_{2i-1}+D_{2i}+D_{2i+1}-D_{2i+2}$ for $1\leq i\leq s$,
\item[]$\sigma_{2s+1}=-2D_{2s+1}+2D_{2s+2}$.
\end{itemize}
\end{itemize}

It is the case 60 of \cite{BP5}, labeled as $\mathsf{a^y}(s,s)+\mathsf d(t)$.

\paragraph{\textbf{b) Collapsed tail.}}
Localizing the spherical systems of the cases 
\ref{sss:BDady} ($r=q-1$), \ref{sss:DDady} ($r=q-1$) and \ref{sss:DDday} ($r=p-1$)
in $\supp\Sigma$ we obtain the following spherical system 
for a group of semisimple type $\mathsf A_s\times\mathsf D_{s+1}$.
\begin{itemize}
\item[]$S^\mathrm p=\emptyset$.
\item[]$\Sigma=\{\alpha_1,\ldots,\alpha_s,\ \alpha'_1,\ldots,\alpha'_s,\alpha'_{s+1}\}$.
\item[]$\mathrm A=\{D_1,\ldots,D_{2s+2}\}=\Delta$ 
with Cartan pairing as follows:
\begin{itemize}
\item[]$\alpha_1=D_1+D_2-D_3$,
\item[]$\alpha_i=-D_{2i-2}+D_{2i-1}+D_{2i}-D_{2i+1}$ for $2\leq i\leq s-1$,
\item[]$\alpha_s=-D_{2s-2}+D_{2s-1}+D_{2s}-D_{2s+1}-D_{2s+2}$,
\item[]$\alpha'_i=-D_{2i-1}+D_{2i}+D_{2i+1}-D_{2i+2}$ for $1\leq i\leq s-1$,
\item[]$\alpha'_s=-D_{2s-1}+D_{2s}+D_{2s+1}-D_{2s+2}$,
\item[]$\alpha'_{s+1}=-D_{2s-1}+D_{2s}-D_{2s+1}+D_{2s+2}$.
\end{itemize}
\end{itemize}

It is the case 40 of \cite{BP5}, labeled as $\mathsf{ad^y}(s,s+1)$
or S-10 in \cite{Bra}, and considered also in \cite[Section~5]{BGM}
as the spherical system of the comodel wonderful variety of cotype
$\mathsf D_{2(s+1)}$.

\section{Projective normality}

This section is devoted to prove the following result, that we need in
order to study the singularities of closures of spherical nilpotent $K$-orbits
in $\gop$.

\begin{theorem}	\label{teo: projnorm}
Let $(\mathfrak g,\mathfrak k)$ be a classical symmetric pair of non-Hermitian type, let $\calO \subset \gop$ be a spherical nilpotent $K$-orbit. If $(\mathfrak g,\mathfrak k) = (\mathfrak{sp}(2p+2q), \mathfrak{sp}(2p)+\mathfrak{sp}(2q))$, assume that the signed partition of $\mathcal O$ is neither $(+3^4,+1^{2p-8})$ nor $(-3^4,-1^{2q-8})$ (Cases \ref{sss:CCabx} and \ref{sss:CCbax} in Appendix \ref{A}). Let $X$ be the wonderful $K$-variety associated to $\calO$, then the multiplication of sections
$$
	m_{\calL,\calL'} \colon \grG(X, \calL) \otimes \grG(X, \calL') \lra \grG(X, \calL \otimes \calL') 
$$ is surjective for all globally generated line bundles $\calL, \calL' \in \Pic(X)$.
\end{theorem}

We point out that multiplication is not surjective if $(\mathfrak g,\mathfrak k) = (\mathfrak{sp}(2p+2q), \mathfrak{sp}(2p)+\mathfrak{sp}(2q))$ and $\calO$ is the spherical nilpotent orbit corresponding to the signed partitions $(+3^4,+1^{2p-8})$ or $(-3^4,-1^{2q-8})$), see Example~\ref{example:counterexample} below. These cases will be treated separately in Section~\ref{ss:symplecticcases} with an {\it ad hoc} argument.

Let us briefly recall here some generalities about the multiplication of sections of line bundles on a wonderful variety, for more details and references see \cite{BGM}.

Let $X$ be a wonderful $K$-variety with set of spherical roots $\grS$
and set of colors $\grD$. The classes of colors form a free basis for
the Picard group of $X$, and for the semigroup of globally generated
line bundles. Therefore the Picard group of $X$ is identified with $\mZ \grD$, and the semigroup of globally generated line bundles is identified with $\mN\grD$. Given $E,F \in \mN \grD$ we will also write $m_{E,F}$ meaning $m_{\calL_E, \calL_F}$.

Given $D \in \mZ \grD$ we denote by $\calL_D \in
\Pic(X)$ the corresponding line bundle, and we fix $s_D \in
\grG(X,\calL_D)$ a section whose associated divisor is $D$. 
Recall that every line
bundle on $X$ has a unique
$K$--linearization. 
Then $s_D$
is a highest weight vector, and we denote by $V_D \subset
\grG(X,\calL_D)$ the $K$-submodule generated by $s_D$. Since $X$ is
a spherical variety, $\grG(X,\calL_D)$ is a multiplicity-free
$K$-module, hence $V_D$ is uniquely determined and $s_D$ is uniquely
determined up to a scalar factor.

By identifying $\grS$ with the set of $K$-stable prime divisors of $X$,
every $\grs \in \mZ \grS$ determines a line bundle $\calL_{\grs} \in
\Pic(X)$, and the map $\mZ \grS \lra \Pic(X)$ is injective. The line
bundle $\calL_\grs$ is effective if and only if $\grs \in \mN\grS$,
and for all $\grs \in \mZ \grS$ we fix a section $s^\grs \in
\grG(X,\calL_\grs)$ whose associated divisor is $\grs$. Such a section
is a highest weight vector of weight $0$, and is uniquely determined
up to a scalar factor.

By identifying $\Pic(X)$ with $\mZ \grD$, we regard $\mZ \grS$ as a
sublattice of $\mZ \grD$. This defines a partial order $\leq_\grS$ on
$\mZ \grD$ as follows: if $D,E \in \mZ \grD$, then $D
\leq_\grS E$ if and only if $E-D \in \mN\grS$. This allows to describe
the space of global sections of $\calL_E$ as follows 
$$
\grG(X, \calL_E) = \bigoplus_{F \in \mN \grD \; : \; F \leq_\grS E} s^{E-F} V_F
$$

In particular, if $E \in \mN \grD$, we have that $\grG(X, \calL_D)$ is
an irreducible $K$-module if and only if $E$ is minuscule in $\mN\grD$
w.r.t.\ $\leq_\grS$ or zero, that is, if $F \in \mN \grD$ and $F \leq_\grS E$
then it must be $F = E$.

To any line bundle $\calL_E$ on $X$, we attach two characters $\xi_E$ and $\gro_E$ as follows. Let $H$ be the stabilizer of a point $x_0$ in the open orbit of $X$, fix a maximal torus $T$ and a Borel subgroup $B$ such that $T\subset B$, and let
$y_0$ be the point fixed by the opposite Borel of $B$. Then we denote $\xi_E \in \Hom(H,\mathbb C^\times)$ the character given by the action of $H$ over the fiber $\calL_{E,x_0}$, and by $\omega_E \in \Hom(T,\mathbb C^\times)$ the character given by the action of $T$ over the fiber $\calL_{E,y_0}$.

If $E\in \mN\Delta$ then the set of sections $V_E\subset\Gamma(X,\calL_E)$
does not vanish on the closed orbit of $X$, so it defines a regular map
$\phi_E\colon X\lra \mP(V_E^*)$. We choose a non-zero element $h_E \in
V_E^*$ in the line $\phi_E(x_0)$. 
Notice that $V_E$ is the irreducible
module of highest weight $\omega_E$ and that $h_E$ is determined by
the condition $g\cdot h_E = \xi_E(g)h_E$
for all $g\in H$.

For $D \in \grD$, the weight $\omega_D$ is combinatorially described as follows:
if $D \in \grD^{\mathrm{2a}}$ and $\gra \in S$ is such that $D \in \grD(\gra)$, then $\gro_D = 2 \gro_\gra$, otherwise $\gro_D = \sum\gro_\gra$ for all $\gra \in S$ such that $D \in \grD(\gra)$.

\subsection{General reductions}

By making use of quotients and parabolic inductions, it is possible to reduce the study of the multiplication maps. We recall such reductions from \cite{BGM}.

\begin{lemma}[{\cite[Corollary 1.4]{BGM}}]	\label{lem: projnorm quotients}
Let $X$ be a wonderful variety with set of colors $\grD$, let $X'$ be a quotient of $X$ by a distinguished subset $\grD_0 \subset \grD$ with set of colors $\grD'$ and identify $\grD'$ with $\grD \senza \grD_0$. If $D \in \mN \grD$ and $\supp(D) \cap \grD_0 = \vuoto$ and if $\calL_D \in \Pic(X)$ and $\calL'_D \in \Pic(X')$ are the line bundles corresponding to $D$ regarded as an element in $\mN \grD$ and in $\mN\grD'$, then $\grG(X, \calL_D) = \grG(X', \calL'_D)$.

In particular, if $m_{D,E}$ is surjective for all $D,E \in \mN\grD$, then  $m_{D', E'}$ is surjective for all $D', E' \in \mN \grD'$.
\end{lemma}

\begin{lemma}[{\cite[Proposition 1.6]{BGM}}] \label{lem: projnorm parabolic induction}
Let $X$ be a wonderful variety and suppose that $X$ is the parabolic
induction of a wonderful variety $X'$. Then for all $\calL,\calL'$ in
$\Pic(X)$ the multiplication $m_{\calL, \calL'}$ is surjective if and
only if the multiplication $m_{\calL|_{X'}, \calL'|_{X'}}$ is surjective.
\end{lemma}

We now explain how to reduce the study of the multiplications with
respect to wonderful subvarieties.

\begin{lemma}	\label{lem: projnorm localizzazioni}
Let $X$ be a wonderful variety and let $X' \subset X$ be a
wonderful subvariety. If $m_{\calL, \calL'}$ is surjective for all
globally generated $\calL, \calL' \in \Pic(X)$, then $m_{\calL,
  \calL'}$ is surjective for all globally generated $\calL, \calL' \in
\Pic(X')$.
\end{lemma}

\begin{proof}
Denote by $\grS$ and $\grD$ the set of spherical roots and the set of
colors of $X$, and by $\grS'$ and $\grD'$ those of $X'$. The
restriction of line bundles induces a map $\rho \colon \mN\grD \lra \mN
\grD'$, and the restriction of sections $\grG(X, \calL_D) \lra
\grG(X', \calL_{\rho(D)})$ is surjective for all $D \in
\mN\grD$. Given $E,F \in \mN \grD$, the surjectivity
of $m_{\rho(E), \rho(F)}$ follows then from the surjectivity of $m_{E,F}$.

Set
$$
	\grD'_0 = \{D \in \grD' \; : \; c(D,\grs) \leq 0  \quad \forall \grs \in \grS'\}.
$$ 
Notice that every $D \in \mN \grD'_0$ is minuscule
w.r.t. $\leq_{\grS'}$ or zero, namely $\grG(X', \calL_D) = V_D$ for all $D \in
\mN\grD'$. Indeed if $D \in \mN \grD'_0$ and $D-\grs \in \mN \grD$ for
some $\grs \in \mN\grS$, then it follows that $-\grs \in \mN\grD$,
hence both $\grs$ and $-\grs$ define effective divisors on $X'$. On
the other hand the cone of effective divisors of $X'$ contains no line
since $X'$ is complete, therefore it must be $\grs = 0$.

Let $D \in \grD$, reasoning as in \cite[\S 1.13]{Ga} by the
combinatorial description of $\rho$ it follows that for all $D \in
\grD$ there exists $D' \in (\grD' \smallsetminus \grD'_0) \cup \{0\}$
such that $\rho(D) - D' \in \mN \grD'_0$, and conversely for all $D'
\in \grD' \smallsetminus \grD'_0$ there exists $D \in \grD$ with
$\rho(D) - D' \in \mN \grD'_0$.

Let now $E, F \in \mN\grD'$, then by the previous discussion there exist
$E', F' \in \mN\grD'_0$ such that $E+E', F+F' \in \rho(\mN\grD)$. On
the other hand since $E', F' \in \mN \grD'_0$ we have $\grG(X,
\calL_{E+E'+F+F'}) = \grG(X, \calL_{E+F}) V_{E'+F'}$ and
$$
	\mathrm{Im}(m_{E+E', F+F'}) = \mathrm{Im}(m_{E,F}) V_{E'} V_{F'} = \mathrm{Im}(m_{E,F}) V_{E'+F'}.
$$
Therefore the surjectivity of $m_{E,F}$ follows from that of $m_{E+E', F+F'}$.
\end{proof}

A strategy to prove the surjectivity of the multiplication map was described in \cite{CM_projective-normality} for wonderful symmetric varieties and in \cite{BGM} for
general wonderful varieties. Such a strategy reduces the proof of the surjectivity of the multiplication maps for all pair of globally generated line bundles to a finite number of computations, which arise in correspondence to the so-called \textit{fundamental low tiples}. 

Recall from \cite{BGM} that a triple $(D,E,F) \in (\mathbb{N}\Delta)^3$ with $F \leq_\Sigma D+E$ is called a \textit{low triple} if, for all $D',E' \in \mN\Delta$ such that
$D' \leq_\Sigma D$, $E' \leq_\Sigma E$ and $F \leq_\Sigma D' + E'$, it holds $D' = D$ and $E' = E$. The triple $(D,E,F)$ is called a \emph{fundamental triple} if $D,E \in \grD$.

To determine the low triples is useful the notion of covering
difference.  Let $E, F \in \mN \grD$ with $E <_\grS F$ and
suppose that $E$ is maximal in $\mN \grD$ with this property:
then we say that $F$ covers $E$ and we call $F - E$ a
\emph{covering difference} in $\mN \grD$.

For all $E = \sum_{D \in \grD} k_D D \in \mZ\grD$, define its
\emph{positive part} $E^+= \sum_{k_D > 0} k_D D $, its \emph{negative
  part} $E^-= E^+-E$ and its \emph{height} $\height(E) = \sum_{D \in
  \grD} k_D$.  Notice that $\grg \in \mN \grS$ is a covering
difference in $\mN \grD$ if and only $\grg^+$ covers $\grg^-$.

As noticed in \cite[Section~2.1, Remark]{BGM}, the covering differences in $\mN\grD$ are finitely many, therefore there is always a bound for the height of the positive part of a covering difference. In all the examples we know (included those we will deal with in the present paper) this bound can be taken to be 2. 

Let $(D,E,F)$ be a low triple and suppose that $m_{D,E}$ is surjective, then it is a straightforward consequence of the definition that $s^{D+E-F} V_F \subset V_D V_E$. On the other hand we have the following.

\begin{lemma}[{\cite[Lemma 2.3]{BGM}}]\label{lem:riduzionetriple}
Let $X$ be a wonderful variety and let $n$ be such that $\height(\grg^+) \leq n$
for every covering difference $\grg$. If $s^{D+E-F} V_F \subset V_D V_E$
for all low triples $(D,E,F)$ with $\height(D+E)\leq n$, then the multiplication maps $m_{D,E}$ are surjective for all $D,E \in \mN\grD$.
\end{lemma}

To verify that $s^{D+E-F} V_F \subset V_D V_E$ we will make use
of the following.

\begin{lemma}[{\cite[Lemma~19]{CLM}}] \label{lemma: supporto moltiplicazione}
Let $D,E,F \in \mN\grD$ be such that $D \leq_\grS E+F$. Then
$s^{E+F-D} V_D \subset V_E V_F$ if and only if the projection of $h_E
\otimes h_F\in V(\gro^*_E) \otimes V(\gro^*_F)$ onto the isotypic
component of highest weight $\gro^*_D$ is non-zero.
\end{lemma}

\begin{example} \label{example:counterexample}
Let $\mathfrak g = \mathfrak{sp}(2p+2q)$ and $\mathfrak k= \mathfrak{sp}(2p) + \mathfrak{sp}(2q)$. If $p\geq4$ consider the spherical nilpotent $K$-orbit $\mathcal O$ defined by the signed partition $(+3^4,+1^{2p-8})$ (or similarly the one defined by $(-3^4,-1^{2q-8})$ if $q\geq4$). Let $X$ be the corresponding wonderful $K$-variety, then there are elements $D,E \in \mN\grD$ such that $m_{D,E}$ is not surjective.

Indeed, the spherical system of $X$ is the following:
\[\begin{picture}(17700,3000)(-300,-1500)
\multiput(0,0)(1800,0){2}{\usebox{\edge}}
\multiput(0,0)(1800,0){3}{\usebox{\aone}}
\put(3600,0){\usebox{\edge}}
\put(5400,0){\usebox{\wcircle}}
\put(5400,0){\usebox{\edge}}
\put(7200,0){\usebox{\susp}}
\put(10800,0){\usebox{\leftbiedge}}
\put(15300,0){\usebox{\leftbiedge}}
\multiput(15300,0)(1800,0){2}{\usebox{\aone}}
\multiput(0,-900)(15300,0){2}{\line(0,-1){600}}
\put(0,-1500){\line(1,0){15300}}
\multiput(1800,-900)(15300,0){2}{\line(0,-1){300}}
\put(1800,-1200){\line(1,0){13400}}
\put(15400,-1200){\line(1,0){1700}}
\multiput(0,900)(3600,0){2}{\line(0,1){600}}
\put(17100,900){\line(0,1){600}}
\put(0,1500){\line(1,0){17100}}
\put(3900,-600){\line(1,0){10050}}
\put(13950,-600){\line(0,1){1200}}
\put(13950,600){\line(1,0){1050}}
\put(0,600){\usebox{\toe}}
\put(3600,600){\usebox{\tow}}
\put(15300,600){\usebox{\toe}}
\put(17100,600){\usebox{\tow}}
\end{picture}
\]

%
%
%

Label the spherical roots and the colors of $X$ as follows:
$$\grs_1 = \gra_2, \qquad \grs_2 = \gra'_2, \qquad \grs_3 = \gra_1, \qquad \grs_4 = \gra'_1, \qquad \grs_5 = \gra_3$$
$$	D_1 = D_{\gra_2}^+, \quad D_2 = D_{\gra_2}^-, \quad D_3 = D_{\gra_1}^+, \quad D_4 = D_{\gra_1}^-, \quad D_5 = D_{\gra_3}^-, \quad D_6 = D_{\gra_4}$$
Then the Cartan pairing of $X$ is expressed as follows:
\begin{align*}
& \grs_1 = D_1 + D_2 -D_3\\
& \grs_2 = -D_1 +D_2 +D_3 -D_4 -D_5 \\
& \grs_3 = -D_2 +D_3 +D_4 -D_5 \\
& \grs_4 = -D_3 +D_4 +D_5\\
& \grs_5 = -D_2 +D_3 -D_4 +D_5 - D_6
\end{align*}

Consider the triple $(D_3, D_3, D_1+D_2+D_6)$: then $2D_3 -D_1-D_2 - D_6= \grs_2+ \grs_3 + \grs_4 +\grs_5$, and the triple is easily shown to be low. On the other hand if $V_{D_1+D_2+D_6} \subset V_{D_3}^2$, then it would be $V(2\gro_2+\gro_4+\gro_2') \subset V(\gro_1+\gro_3+\gro_2')^{\otimes 2}$, which is not the case. Therefore $m_{D_3, D_3}$ is not surjective.
\end{example}

\subsection{Basic cases}
We show in this section that in order to prove Theorem \ref{teo:
  projnorm}, we are reduced to the study of three special families of
wonderful varieties.

Following Section \ref{sss:parabolic induction}, 
by Lemma \ref{lem: projnorm parabolic induction}, 
the surjectivity of the multiplications on
$X$ is reduced to that one on a wonderful $L_1$-variety $Y$, where
$L_1$ is the Levi subgroup of $K$ corresponding to the set of simple
roots in $\supp \grS$. More precisely, $Y$ is the localization of $X$ at
the subset $\supp \grS\subset S$, and the wonderful varieties arising in this
way are described in Sections \ref{sss:trivial cases},
\ref{sss:symmetric cases}, \ref{sss:reductive} (only the cases \ref{sss:CCaac} and \ref{sss:CCaac2}), \ref{sss:typeB},
\ref{sss:typeC}, \ref{sss:typeD}.

Analyzing all the possible cases, we now show that to prove the
surjectivity of the multiplications for $Y$ we are reduced to the
following three families:

\begin{itemize}
\item[]$\mathsf{a}^\mathsf{y}(2, 2)+\mathsf{c}(t)$, $t\geq2$,
\[
\begin{picture}(16600,2850)(-300,-1350)
\multiput(0,0)(5500,0){2}{
\multiput(0,0)(1800,0){1}{\usebox{\edge}}
\multiput(0,0)(1800,0){2}{\usebox{\aone}}
}
\put(7300,0){\usebox{\edge}}
\put(9100,0){\usebox{\gcircle}}
\put(9100,0){\usebox{\edge}}
\put(10900,0){\usebox{\susp}}
\put(14500,0){\usebox{\leftbiedge}}

\multiput(0,-900)(7300,0){2}{\line(0,-1){450}}
\put(0,-1350){\line(1,0){7300}}
\multiput(0,900)(5500,0){2}{\line(0,1){600}}
\put(0,1500){\line(1,0){5500}}
\multiput(1800,900)(5500,0){2}{\line(0,1){300}}
\put(1800,1200){\line(1,0){3600}}
\put(5600,1200){\line(1,0){1700}}
\put(1800,600){\usebox{\tow}}
\put(5500,600){\usebox{\toe}}
\end{picture}\]

\item[]$\mathsf{a^y}(s,s) + \mathsf{b}'(t)$, $s,t\geq1$,
\[\begin{picture}(23800,3600)(-300,-1800)
\multiput(0,0)(9100,0){2}{
\put(0,0){\usebox{\edge}}
\multiput(0,0)(1800,0){2}{\usebox{\aone}}
\put(5400,0){\usebox{\aone}}
}
\put(1800,0){\usebox{\susp}}
\multiput(10900,0)(7200,0){2}{\usebox{\susp}}
\multiput(14500,0)(1800,0){2}{\usebox{\edge}}
\put(21700,0){\usebox{\rightbiedge}}
\put(16300,0){\usebox{\gcircletwo}}
\multiput(0,900)(9100,0){2}{\line(0,1){900}}
\put(0,1800){\line(1,0){9100}}
\multiput(1800,900)(9100,0){2}{\line(0,1){600}}
\put(1800,1500){\line(1,0){7200}}
\put(9200,1500){\line(1,0){1700}}
\multiput(5400,900)(9100,0){2}{\line(0,1){300}}
\put(5400,1200){\line(1,0){3600}}
\put(9200,1200){\line(1,0){1600}}
\put(11000,1200){\line(1,0){3500}}
\multiput(1800,-900)(7300,0){2}{\line(0,-1){900}}
\put(1800,-1800){\line(1,0){7300}}
\multiput(3600,-1500)(0,300){3}{\line(0,1){150}}
\put(3600,-1500){\line(1,0){5400}}
\put(9200,-1500){\line(1,0){1700}}
\put(10900,-1500){\line(0,1){600}}
\put(5400,-900){\line(0,-1){300}}
\put(5400,-1200){\line(1,0){3600}}
\put(9200,-1200){\line(1,0){1600}}
\put(11000,-1200){\line(1,0){1700}}
\multiput(12700,-1200)(0,300){2}{\line(0,1){150}}
\multiput(0,600)(1800,0){2}{\usebox{\toe}}
\multiput(10800,600)(3600,0){2}{\usebox{\tow}}
\end{picture}\]

\item[]$\mathsf{ab^y}(s,s)$, $s\geq2$.
\[\begin{picture}(17700,4200)(-300,-2100)
\put(0,0){\usebox{\edge}}
\put(1800,0){\usebox{\susp}}
\put(5400,0){\usebox{\edge}}
\put(9900,0){\usebox{\edge}}
\put(11700,0){\usebox{\susp}}
\put(15300,0){\usebox{\rightbiedge}}
\multiput(0,0)(9900,0){2}{\multiput(0,0)(5400,0){2}{\multiput(0,0)(1800,0){2}{\usebox{\aone}}}}
\multiput(0,900)(9900,0){2}{\line(0,1){1200}}
\put(0,2100){\line(1,0){9900}}
\multiput(1800,900)(9900,0){2}{\line(0,1){900}}
\put(1800,1800){\line(1,0){8000}}
\put(10000,1800){\line(1,0){1700}}
\multiput(5400,900)(9900,0){2}{\line(0,1){600}}
\put(5400,1500){\line(1,0){4400}}
\put(10000,1500){\line(1,0){1600}}
\put(11800,1500){\line(1,0){3500}}
\multiput(7200,900)(9900,0){2}{\line(0,1){300}}
\put(7200,1200){\line(1,0){2600}}
\put(10000,1200){\line(1,0){1600}}
\put(11800,1200){\line(1,0){3400}}
\put(15400,1200){\line(1,0){1700}}
\multiput(1800,-900)(8100,0){2}{\line(0,-1){1200}}
\put(1800,-2100){\line(1,0){8100}}
\multiput(3600,-1800)(0,300){3}{\line(0,1){150}}
\put(3600,-1800){\line(1,0){6200}}
\put(10000,-1800){\line(1,0){1700}}
\put(11700,-1800){\line(0,1){900}}
\put(5400,-900){\line(0,-1){600}}
\put(5400,-1500){\line(1,0){4400}}
\put(10000,-1500){\line(1,0){1600}}
\put(11800,-1500){\line(1,0){1700}}
\multiput(13500,-1500)(0,300){2}{\line(0,1){150}}
\multiput(7200,-900)(8100,0){2}{\line(0,-1){300}}
\put(7200,-1200){\line(1,0){2600}}
\put(10000,-1200){\line(1,0){1600}}
\put(11800,-1200){\line(1,0){1600}}
\put(13600,-1200){\line(1,0){1700}}
\multiput(0,600)(1800,0){2}{\usebox{\toe}}
\put(5400,600){\usebox{\toe}}
\put(11700,600){\usebox{\tow}}
\multiput(15300,600)(1800,0){2}{\usebox{\tow}}
\end{picture}
\]
\end{itemize}

In the cases of Section \ref{sss:trivial cases} the wonderful variety
$X$ is a flag variety, therefore the surjectivity of the
multiplication of globally generated line bundles holds trivially
since the space of sections of a globally generated line bundle on a
flag variety is an irreducible $K$-module.

In the cases of Section \ref{sss:symmetric cases} the wonderful
variety $Y$ is the wonderful compactification of an adjoint symmetric
variety, and the surjectivity of the multiplication of globally
generated line bundles holds thanks to \cite{CM_projective-normality}.

In the cases \ref{sss:CCaac} and \ref{sss:CCaac2} of Section~\ref{sss:reductive} (up to switching the two factors
of $K$) the surjectivity of the multiplications of $Y$ is reduced to
that one of the wonderful variety $Z$ with spherical system
$\mathsf{a}^\mathsf{y}(2,2)+\mathsf c(t)$ where $t \geq 2$. 
More precisely, start with
$Z$ and consider the set of colors $\{D_{\gra_1}^+, D_{\gra_2}^+\}$,
it is distinguished and the corresponding quotient is a
parabolic induction of $Y$. Therefore the surjectivity of the
multiplications of $Y$ follows from that of $Z$ thanks to 
Lemma~\ref{lem: projnorm quotients} and Lemma~\ref{lem: projnorm parabolic induction}.

In the cases of Section \ref{sss:typeB} (a) $Y$ is the wonderful
variety with spherical system $\mathsf a^\mathsf y(s,s)+\mathsf b'(t)$ 
where $s\geq0$ and $t \geq1$, but if $s=0$ it is just an adjoint symmetric variety. 
In the cases of Section~\ref{sss:typeB} (b) $Y$ is the
wonderful variety with spherical system $\mathsf{ab}^\mathsf y(s,s)$ where $s\geq2$.

In the cases of Section \ref{sss:typeC} (a) $Y$ is the wonderful
variety with spherical system $\mathsf a^\mathsf y(2,2)+\mathsf c(t)$ where $t \geq
2$, whereas in the cases of Section~\ref{sss:typeC} (b) $Y$ is the
wonderful variety with spherical system $\mathsf{ab}^\mathsf y(2,2)$.

In the cases of Section \ref{sss:typeD} (a) $Y$ is the wonderful
variety with spherical system $\mathsf a^\mathsf y(s,s)+\mathsf d(t)$ 
where $s\geq0$ and $t \geq2$. The surjectivity of the multiplications in this case can be
reduced to that of a comodel wonderful variety, which is known by
\cite[Theorem~5.2]{BGM}. Let indeed $Z$ be the comodel wonderful
variety of cotype $\sfD_{2(s+t)}$, this is the wonderful variety with
the following spherical system for a group of semisimple type $\sfA_{s+t-1} \times
\sfD_{s+t}$. 
\[\begin{picture}(17550,4500)(-300,-2100)
\put(0,0){\usebox{\edge}}
\put(1800,0){\usebox{\susp}}
\put(5400,0){\usebox{\edge}}
\put(600,0){
\put(9300,0){\usebox{\edge}}
\put(11100,0){\usebox{\susp}}
\put(14700,0){\usebox{\bifurc}}
}
\multiput(0,0)(1800,0){2}{\usebox{\aone}}
\multiput(5400,0)(1800,0){2}{\usebox{\aone}}
\put(600,0){
\multiput(9300,0)(1800,0){2}{\usebox{\aone}}
\put(14700,0){\usebox{\aone}}
\multiput(15900,-1200)(0,2400){2}{\usebox{\aone}}
}
\put(7200,-2100){\line(0,1){1200}}
\put(7200,-2100){\line(1,0){8100}}
\put(15300,-2100){\line(0,1){1200}}
\put(5400,-900){\line(0,-1){900}}
\put(5400,-1800){\line(1,0){1700}}
\put(7300,-1800){\line(1,0){6200}}
\multiput(13500,-1800)(0,300){3}{\line(0,1){150}}
\multiput(3600,-1500)(0,300){3}{\line(0,1){150}}
\put(3600,-1500){\line(1,0){1700}}
\put(5500,-1500){\line(1,0){1600}}
\put(7300,-1500){\line(1,0){4400}}
\put(11700,-1500){\line(0,1){600}}
\multiput(1800,-900)(8100,0){2}{\line(0,-1){300}}
\put(1800,-1200){\line(1,0){1700}}
\multiput(3700,-1200)(1800,0){2}{\line(1,0){1600}}
\put(7300,-1200){\line(1,0){2600}}
\put(7200,2400){\line(0,-1){1500}}
\put(7200,2400){\line(1,0){10050}}
\put(17250,2400){\line(0,-1){3000}}
\multiput(17250,-600)(0,2400){2}{\line(-1,0){450}}
\multiput(5400,2100)(9900,0){2}{\line(0,-1){1200}}
\put(5400,2100){\line(1,0){1700}}
\put(7300,2100){\line(1,0){8000}}
\multiput(1800,1500)(9900,0){2}{\line(0,-1){600}}
\put(1800,1500){\line(1,0){3500}}
\put(5500,1500){\line(1,0){1600}}
\put(7300,1500){\line(1,0){4400}}
\multiput(0,1200)(9900,0){2}{\line(0,-1){300}}
\put(0,1200){\line(1,0){1700}}
\put(1900,1200){\line(1,0){3400}}
\put(5500,1200){\line(1,0){1600}}
\put(7300,1200){\line(1,0){2600}}
\multiput(0,600)(1800,0){2}{\usebox{\toe}}
\put(5400,600){\usebox{\toe}}
\put(11700,600){\usebox{\tow}}
\put(15300,600){\usebox{\tow}}
\put(16500,-600){\usebox{\tonw}}
\put(16500,1800){\usebox{\tosw}}
\end{picture}\]
Consider the wonderful subvariety of $Z$ associated to
$\Sigma \smallsetminus \{\alpha_{s+1}, \ldots, \alpha_{s+t-1}\}$, then the set of
colors $\{D_{\alpha_{s+1}'}^-,$ $D_{\alpha_{s+2}'}^\pm, \ldots ,
D_{\alpha_{s+t}'}^\pm\}$ is distinguished, and the corresponding
quotient is a parabolic induction of $Y$.  Therefore the surjectivity
of the multiplications of $Y$ follows from that of $Z$ thanks to 
Lemma~\ref{lem: projnorm quotients} and Lemma~\ref{lem: projnorm parabolic induction}.

Finally, in the cases of Section \ref{sss:typeD} (b) $Y$ is the
comodel wonderful variety of cotype $\sfD_{2(s+1)}$, and the
surjectivity of the multiplications for this variety follows by
\cite[Theorem~5.2]{BGM}.


\subsection{Projective normality of $\mathsf{a}^\mathsf{y}(2, 2)+\mathsf{c}(t)$}\label{ss:ay22ct} 
Consider the wonderful variety $X$ for a semisimple group $G$ of type $\sfA_2
\times \sfC_{t+1}$ with $t\geq 2$ defined by the following spherical system.
\[
\begin{picture}(16600,2850)(-300,-1350)
\multiput(0,0)(5500,0){2}{
\multiput(0,0)(1800,0){1}{\usebox{\edge}}
\multiput(0,0)(1800,0){2}{\usebox{\aone}}
}
\put(7300,0){\usebox{\edge}}
\put(9100,0){\usebox{\gcircle}}
\put(9100,0){\usebox{\edge}}
\put(10900,0){\usebox{\susp}}
\put(14500,0){\usebox{\leftbiedge}}

\multiput(0,-900)(7300,0){2}{\line(0,-1){450}}
\put(0,-1350){\line(1,0){7300}}
\multiput(0,900)(5500,0){2}{\line(0,1){600}}
\put(0,1500){\line(1,0){5500}}
\multiput(1800,900)(5500,0){2}{\line(0,1){300}}
\put(1800,1200){\line(1,0){3600}}
\put(5600,1200){\line(1,0){1700}}
\put(1800,600){\usebox{\tow}}
\put(5500,600){\usebox{\toe}}
\end{picture}\]

The spherical system associated to this Luna diagram is described in Section~\ref{sss:typeC}.
For convenience we number the five spherical roots in the following way:
$$ \sigma_1 = \alpha_2, \quad \sigma_2 = \alpha'_2, \quad \sigma_3 =
\alpha_1, \quad \sigma_4 = \alpha'_1, \quad \sigma_5 = \alpha'_2 +
\sum_{i=3}^t 2\alpha'_i + \alpha'_{t+1}.$$ 
There are six colors that we label in the following way:
\begin{align*} D_1 &= D_{\alpha_2}^-,& \quad D_2 &= D_{\alpha_2}^+,& \quad D_3 &=
D_{\alpha_1}^-, \\ D_4 &= D_{\alpha_1}^+,& \quad D_5 &=
D_{\alpha'_1}^-,& \quad D_6 &= D_{\alpha'_3}.
\end{align*} 
The weights of these colors are the following:
\begin{align*}
\omega_{D_1}&=\omega_2, &\quad 
\omega_{D_2}&=\omega_2+\omega_2', &\quad 
\omega_{D_3}&=\omega_1+\omega_2', \\ 
\omega_{D_4}&=\omega_1+\omega_1', &\quad 
\omega_{D_5}&=\omega_1', &\quad 
\omega_{D_6}&=\omega_3'.
\end{align*}

Notice that the $G$-stable divisor of $X$ corresponding to $\grs_5$ is a parabolic induction of a comodel wonderful variety of cotype $\sfA_5$ (see \cite[Section~5]{BGM}). Therefore we can restrict our study to the covering differences and the low triples of $X$ which contain $\grs_5$.

\begin{lemma}	\label{lemma:coveringXXX}
Let $\grg \in \mN\grS$ be a covering difference in $\mN\grD$ with
$\grs_5 \in \supp_\grS \grg$, then either $\grg = \grs_5 = -D_5 + D_6$
or $\grg = \grs_2 + \grs_4 + \grs_5 = -D_1 + D_2$.
Every other covering difference $\grg \in \mN \grS$ verifies $\height(\grg^+) = 2$.
\end{lemma}

\begin{proof}
Denote $\gamma = \sum a_i \sigma_i$, then we have
\begin{align} \label{gamma}
\gamma = (a_1 - a_2)D_1 + (a_1+a_2-a_3) D_2 + (-a_1 + a_2 +a_3 -a_4) D_3 + \notag \\
+(-a_2 + a_3 + a_4) D_4 + (-a_3 + a_4 - a_5) D_5 + (-a_2 + a_5) D_6 
\end{align}
Suppose that $a_5 \neq 0$. If $D_5 \in \supp(\grg^-)$ then $\gamma^- +
\sigma_5 \in \mN\grD$, and if $D_6 \in \supp(\grg^+)$ then $\gamma^+ -
\sigma_5 \in \mN\grD$. Therefore if $\grg \neq \grs_5$ it must be $D_5
\not \in \supp(\grg^-)$ and $D_6 \not \in \supp(\grg^+)$, namely $a_3
+ a_5 \leq a_4$ and $a_5 \leq a_2$. It follows that $a_2 > 0$ and $a_4
> 0$, suppose that $\grs \neq \grs_2 + \grs_4 + \grs_5 = -D_1 +
D_2$. Then $a_1 + a_4 \leq a_2 + a_3$ since $\grg^- + \grs_4 \not \in
\mN \grD$, and $a_2 \leq a_1$ since $\grg^- + \grs_2 + \grs_4 + \grs_5
\not \in \mN\grD$. Therefore we get $a_1 + (a_4 - a_3) \leq a_2 \leq
a_1$, which is absurd since $a_4 - a_3 \geq a_5 > 0$.

As already noticed, the $G$-stable divisor of $X$ corresponding to $\grs_5$ is a parabolic induction of a comodel wonderful variety of cotype $\sfA_5$. Therefore the covering differences $\grg$ with $\grs_5 \not \in \supp_\grS \grg$ coincide with those studied in \cite[Proposition~3.2]{BGM}, and they all satisfy $\height(\grg^+) = 2$.
\end{proof}

\begin{lemma}\label{lem:tripleXX}
Let $(D,E,F)$ be a low fundamental triple, denote $\grg = D+E-F$ and
suppose that $\grs_5 \in \supp_\grS \grg$. Then we have the following
possibilities:
\begin{itemize}
	\item[-] $(D_2,D_3, D_1+ D_4 + D_5)$, $\grg = \grs_2 + \grs_5$;
	\item[-] $(D_3,D_3, D_1 +2 D_5)$, $\grg = \grs_2 + \grs_3 + \grs_5$;
	\item[-] $(D_2,D_2, D_4 +D_5)$, $\grg = \grs_1 + \grs_2 + \grs_5$;
	\item[-] $(D_2,D_3,2 D_5)$, $\grg = \grs_1 + \grs_2 + \grs_3 + \grs_5$;
	\item[-] $(D_3,D_4, D_1+D_5)$, $\grg = \grs_2 + \grs_3 + \grs_4 + \grs_5$;
	\item[-] $(D_4, D_4, D_1)$, $\grg = \grs_2 + \grs_3 + 2\grs_4 + \grs_5$.
\end{itemize}
\end{lemma}

\begin{proof}
By Lemma \ref{lemma:coveringXXX}, $\grs_5 = -D_5 + D_6$ and $\grs_2 + \grs_4 + \grs_5 = -D_1 + D_2$ are the unique covering differences $\grg$ with $\height(\grg^+) = 1$. Therefore $D_1, D_3, D_4, D_5$ are minuscule in $\mN\grD$.

Let $(D,E,F)$ be a fundamental triple with $\supp(F) \cap \supp(D+E) = \vuoto$, denote $\grg = D+E-F = \sum a_i \grs_i$ and suppose $a_5 > 0$. Notice that, if $(D,E,F)$ is a low triple, then $D_6 \not \in \supp(\grg^+)$: suppose indeed $D = D_6$, then $D_5 <_\grS D$ and $F \leq_\grS D_5 +E$. Therefore, if $(D,E,F)$ is a low triple, then (\ref{gamma}) implies $0 < a_5 \leq a_2$.

Suppose $a_4 = 0$. Then for every covering difference $\grs \leq \grg$ it holds $\height(\grs^+) = 2$, therefore $(D,E,F)$ is necessarily a low triple.

To classify such fundamental triples, suppose $D_2 \not \in \supp(\grg^+)$. Then $c(D_2,\grg) \leq 0$, hence $a_1 + a_2 \leq a_3$ and we get $2 \leq 2a_2 \leq a_2 + a_3 - a_1$. Being $\height(\grg^+) = 2$, it follows then $D = E = D_3$. Equivalently, we have the equality $c(D_3,\grg) = -a_1 + a_2 + a_3 = 2$, and the inequalities $c(D_2,\grg) \leq 0$, $c(D_4,\grg) \leq 0$ imply $2a_1-a_3 +2 \leq a_3 \leq a_1-a_3 +2$. It follows $a_1 = 0$ and $a_2 = a_3 = 1$, and the inequality $0 < a_5 \leq a_2$ imply $a_5 = 1$. Therefore $\grg = \grs_2+\grs_3+\grs_5$ and $F = D_1 +2 D_5$.

Similarly, suppose $a_4 = 0$ and $D_3 \not \in \supp(\grg^+)$. Then $c(D_3,\grg) \leq 0$, hence $a_2 + a_3 \leq a_1$, and we get $2 \leq 2a_2 \leq a_1 + a_2 - a_3$. Being $\height(\grg^+) = 2$, it follows then $D = E = D_2$. Equivalently, $c(D_2,\grg) = a_1 + a_2 -a_3 =2$, and the inequalities $c(D_1,\grg) \leq 0$, $c(D_3,\grg) \leq 0$ imply $a_2 +a_3 \leq a_1 \leq a_2$. It follows $a_3 = 0$ and $a_1 = a_2 =1$, and the inequality $0 < a_5 \leq a_2$ imply $a_5 = 1$. Therefore $\grg = \grs_1+\grs_2+\grs_5$ and $F = D_4 +D_5$.

Suppose now $a_4 = 0$ and $\grg^+ = D_2+D_3$. Then the equalities $c(D_2,\grg) = c(D_3,\grg) = 1$ imply $a_3 - a_1 = a_2 -1 = 1 - a_2$, and it follows $a_1 = a_3$ and $a_2 = 1$. Therefore the inequality $0 < a_5 \leq a_2$ implies $a_5 = 1$, and the inequality $c(D_1,\grg) \leq 0$ implies $a_1 \leq a_2$. Therefore either $\grg = \grs_2 + \grs_5$ and $F = D_1 + D_4 + D_5$, or $\grg = \grs_1 + \grs_2 + \grs_3 + \grs_5$ and $F = 2D_5$.

Suppose finally $a_4 > 0$. Notice that, if $(D,E,F)$ is a low triple, then $D_2 \not \in \supp(\grg^+)$: indeed $\grs_2 + \grs_4 + \grs_5 \leq \grg$, and if e.g.\ $D = D_2$ then $D_1 <_\grS D$ and $F \leq_\grS D_1 +E$. Therefore $c(D_2,\grg) \leq 0$, hence $0 < a_1+a_2 \leq a_3$. It follows $c(D_4, \grg) = -a_2 +a_3+a_4 \geq a_1 + a_4 > 0$, therefore $D_4 \in \supp(\grg^+)$. Being $\height(\grg^+) = 2$, in particular it must be $a_4 \leq 2$.

Suppose $a_4 = 1$. Then $c(D_3,\grg) = -a_1 +a_2 +a_3 -a_4 \geq 2a_2 -a_4 > 0$, hence $\grg^+ = D_3 + D_4$. Therefore $c(D_3,\grg) = c(D_4,\grg) = 1$ and we get the equalities $a_2+a_3 = a_1 +2$ and $a_2 = a_3$. The inequality $c(D_2, \grg) \leq 0$ implies then $a_1 +a_2 \leq a_2$, hence $a_1 = 0$, $a_2 = a_3 = 1$, and $a_5 = 1$ thanks to the inequality $0 < a_5 \leq a_2$. Therefore $\grg = \grs_2+\grs_3+\grs_4+\grs_5$ and $F = D_1+D_5$, and $(D,E,F)$ is a low triple since $D_3,D_4$ are both minuscule.

Suppose now $a_4 = 2$. Then $c(D_4, \grg) = -a_2 +a_3+a_4 \geq a_1 + a_4 \geq 2$, and being $\height(\grg^+) = 2$ it follows $\grg^+ = 2D_4$, and moreover we get $a_1 = 0$ and $a_2 = a_3$. By the inequalities $c(D_2, \grg) \leq 0$, $c(D_3, \grg) \leq 0$ we get then $a_1 + a_2 \leq a_3$ and  $-a_1+a_2+a_3 -a_4 \leq 0$. On the other hand $c(D_3,\grg) = -a_1+a_2+a_3- a_4 \geq 2 a_2-a_4 = 2a_2 -2 \geq 0$, therefore $c(D_3,\grg) = 0$ and it follows $a_2 = 1$, and $a_5 = 1$ as well thanks to the inequality $0 < a_5 \leq a_2$. Therefore $\grg = \grs_2 + \grs_3 +2\grs_4 +\grs_5$ and $F = D_1$, and $(D,E,F)$ is a low triple since $D_4$ is minuscule.
\end{proof}

To prove the projective normality of $X$ we now apply Lemma \ref{lemma: supporto moltiplicazione}. This requires some 
computations. We first need an explicit description of the invariants. Let $V = \mC^3$ with standard basis given by $e_1,e_2,e_3$.
Let $W=\mC^{2n}$ where $n=t+1$ and we choose a basis $e'_1,\dots,e'_n,e'_{-n},\dots, e'_{-1}$ and fix a symplectic 
form such that $\omega(e'_i,e'_j)=\grd_{i,-j}$ for $i>0$.

We set
$\mathsf\Lambda^2_0 W=\{ \gra \in \mathsf\Lambda^2 W : \langle \omega , \gra\rangle=0\}$ and $\omega^*= \sum_{i=1}^n e'_i\wedge e'_{-i}$. Let 
$\grf_1, \grf_2,\grf_3$ be the basis of $V^*$ dual to $e_1,e_2,e_3$. Notice that the isomorphism from  $\mathsf\Lambda^2 V$ to $V^*$ sending
$e_1 \wedge e _2$ to $\grf_3$,
$e_1 \wedge e _3$ to $-\grf_2$ and
$e_2 \wedge e _3$ to $\grf_1$ is $G$-equivariant.

We set $G=\mathrm{SL}(V^*)\times \mathrm{Sp}(W,\omega)$, so that we can take $H$ as the stabilizer of the line spanned by the vector 
$e=e_1\otimes e'_{-2} -e_2\otimes e'_{2} - e_3 \otimes e'_{1}$.

We denote by $h_i$ the vector $h_{D_i}\in V_{D_i}^*$. In coordinates the vectors $h_i$ are given as follows

\begin{itemize}
\item[-] $V^*_{D_4} = V\otimes W$ and $h_4=e$;

\item[-] $V^*_{D_5} = W$ and $h_5=e'_{1}$;

\item[-] $V^*_{D_3} = V\otimes \mathsf\Lambda^2_0 W$ and 
$h_3=e_1\otimes(e'_{1}\wedge e'_{-2}) -e_2 \otimes (e'_{1}\wedge e'_{2})$;

\item[-] $V^*_{D_2} = V^*\otimes \mathsf\Lambda^2_0 W$ and 
$$h_2=
\grf_3 \otimes \left( e'_{2} \wedge e'_{-2} - \frac 1n \omega^* \right)
-\grf_2 \otimes (e'_{1} \wedge e'_{-2}) 
-\grf_1 \otimes (e'_{1} \wedge e'_{2});
$$

\item[-] $V^*_{D_1} = V^*$ and $h_1=\grf_3$.
\end{itemize}

We can now prove 

\begin{proposition}
The multiplication $m_{D,E}$ is surjective for all $D,E \in \mN\grD$.
\end{proposition}

\begin{proof}
By Lemma \ref{lemma:coveringXXX} every covering difference $\grg \in \mN \grS$ satisfies $\height(\grg^+) \leq 2$, therefore by Lemma~\ref{lem:riduzionetriple} it is enough to check that $s^{D+E-F}V_F\subset V_D\cdot V_E$ for all low fundamental triples $(D,E,F)$.

Suppose that $\grs_5 \notin \supp_\grS(D+E-F)$ and let $X'$ be the $G$-stable divisor of $X$ corresponding to $\grs_5$. Then $X'$ is a parabolic induction of a comodel wonderful variety of cotype $\sfA_5$, hence the inclusion $s^\grg V_F \subset V_D\cdot V_E$ follows by Lemma~\ref{lem: projnorm parabolic induction} together with \cite[Theorem~5.2]{BGM}.

By Lemma \ref{lemma: supporto moltiplicazione} we are reduced to prove that for all low triples $(D_i,D_j,F)$ listed in Lemma~\ref{lem:tripleXX} 
the projection of $h_i\otimes h_j$ onto the isotypic component of type 
$V^*_{F}$ in $V_{D_i}^*\otimes V_{D_j}^*$ is non-zero. 

$(D_2,D_3,D_1+D_4+D_5)$. 
We have $V^*_{D_1+D_4+D_5}=\mathfrak{sl}(V)\otimes \mathsf S^2W$,
the equivariant map 
$$
\pi\colon \big(V^* \otimes \mathsf\Lambda_0^2 W\big)  \otimes \big(V \otimes \mathsf\Lambda_0^2 W\big) 
\lra \mathfrak{sl}(V)\otimes \mathsf S^2W 
$$
given by
\begin{align*}
\pi&\bigg(\big(\grf\otimes a\wedge b \big)\otimes \big(v\otimes c\wedge d \big) \bigg) = \\
&=\bigg(\grf\otimes v - \frac 13 \grf(v) \mathrm{Id} \bigg) \otimes \bigg( \omega(a,c) bd - \omega(b,c) ad -\omega(a,d)bc+\omega(b,d)ac \bigg)
\end{align*}
and  
$$
\pi(h_2\otimes h_3)=
(\grf_3\otimes e_1)\otimes e'_1 e'_{-2}
-(\grf_3\otimes e_2)\otimes e'_1 e'_2 
+(\grf_1\otimes e_1 +\grf_2\otimes e_2)\otimes (e'_1)^2
\neq 0.
$$

$(D_3,D_3, D_1 +2 D_5)$.
We have $V^*_{D_1 +2 D_5}=\mathsf\Lambda^2 V\otimes \mathsf S^2W$,
the equivariant map
$$
\pi\colon \big(V \otimes \mathsf\Lambda_0^2 W\big)  \otimes \big(V \otimes \mathsf\Lambda_0^2 W\big) 
\lra \mathsf\Lambda^2 V\otimes \mathsf S^2W 
$$
given by
\begin{align*}
\pi&\Big(\big(u\otimes a\wedge b \big)\otimes \big(v\otimes c\wedge d \big) \Big) = \\
&=\Big(u\wedge v\Big) \otimes \Big( \omega(a,c) bd - \omega(b,c) ad -\omega(a,d)bc+\omega(b,d)ac \Big)
\end{align*}
and  
$$
\pi(h_3\otimes h_3)=
2(e_1\wedge e_2)\otimes (e'_1)^2
\neq 0.
$$

$(D_2,D_2, D_4 +D_5)$.
We have $V^*_{D_4 +D_5}=\mathsf\Lambda^2 V^*\otimes \mathsf S^2W$,
the equivariant map
$$
\pi\colon \big(V^* \otimes \mathsf\Lambda_0^2 W\big)  \otimes \big(V^* \otimes \mathsf\Lambda_0^2 W\big) 
\lra \mathsf\Lambda^2 V^*\otimes \mathsf S^2W 
$$
given by
\begin{align*}
\pi&\Big(\big(\varphi\otimes a\wedge b \big)\otimes \big(\psi\otimes c\wedge d \big) \Big) = \\
&=\Big(\varphi\wedge \psi\Big) \otimes \Big( \omega(a,c) bd - \omega(b,c) ad -\omega(a,d)bc+\omega(b,d)ac \Big)
\end{align*}
and  
$$
\pi(h_2\otimes h_2)=
2\big(
(\grf_3\wedge \grf_2)\otimes e'_1e'_{-2}
+(\grf_3\wedge \grf_1)\otimes e'_1e'_{2}
-(\grf_2\wedge \grf_1)\otimes (e'_1)^2
\big)
\neq 0.
$$

$(D_2,D_3,2 D_5)$.
We have $V^*_{2D_5}=\mathsf S^2W$,
the equivariant map
$$
\pi\colon \big(V^* \otimes \mathsf\Lambda_0^2 W\big)  \otimes \big(V \otimes \mathsf\Lambda_0^2 W\big) 
\lra \mathsf S^2W 
$$
given by
\begin{align*}
\pi&\Big(\big(\varphi\otimes a\wedge b \big)\otimes \big(v\otimes c\wedge d \big) \Big) = \\
&=\varphi(v) \Big( \omega(a,c) bd - \omega(b,c) ad -\omega(a,d)bc+\omega(b,d)ac \Big)
\end{align*}
and  
$$
\pi(h_2\otimes h_3)=
-2 (e'_1)^2
\neq 0.
$$

$(D_3,D_4, D_1+D_5)$.
We have $V^*_{D_1+D_5}=\mathsf\Lambda^2 V\otimes W$,
the equivariant map
$$
\pi\colon \big(V \otimes \mathsf\Lambda_0^2 W\big)  \otimes \big(V \otimes W\big) 
\lra \mathsf\Lambda^2 V\otimes W
$$
given by
$$
\pi\Big(\big(u\otimes a\wedge b \big)\otimes \big(v\otimes c \big) \Big) = 
\Big(u\wedge v\Big)\otimes \Big( \omega(a,c) b - \omega(b,c) a \Big)
$$
and  
$$
\pi(h_3\otimes h_4)=
-2 (e_1\wedge e_2)\otimes e'_1
\neq 0.
$$

$(D_4, D_4, D_1)$.
We have 
the equivariant map
$$
\pi\colon \big(V \otimes W\big)  \otimes \big(V \otimes W\big) 
\lra \mathsf\Lambda^2 V
$$
given by
$$
\pi\Big(\big(u\otimes a\big)\otimes \big(v\otimes b \big) \Big) = 
\omega(a,b)\Big(u\wedge v\Big)
$$
and  
\[
\pi(h_4\otimes h_4)=
-2 (e_1\wedge e_2)
\neq 0.	\qedhere
\]
\end{proof}


\subsection{Projective normality of $\mathsf{a^y}(s,s) + \mathsf{b}'(t)$}
Consider the wonderful variety $X$ for a semisimple group $G$ of type $\sfA_s \times \sfB_{s+t}$
with $s,t\geq 1$ defined by the following spherical system.
\[\begin{picture}(23800,3600)(-300,-1800)
\multiput(0,0)(9100,0){2}{
\put(0,0){\usebox{\edge}}
\multiput(0,0)(1800,0){2}{\usebox{\aone}}
\put(5400,0){\usebox{\aone}}
}
\put(1800,0){\usebox{\susp}}
\multiput(10900,0)(7200,0){2}{\usebox{\susp}}
\multiput(14500,0)(1800,0){2}{\usebox{\edge}}
\put(21700,0){\usebox{\rightbiedge}}
\put(16300,0){\usebox{\gcircletwo}}
\multiput(0,900)(9100,0){2}{\line(0,1){900}}
\put(0,1800){\line(1,0){9100}}
\multiput(1800,900)(9100,0){2}{\line(0,1){600}}
\put(1800,1500){\line(1,0){7200}}
\put(9200,1500){\line(1,0){1700}}
\multiput(5400,900)(9100,0){2}{\line(0,1){300}}
\put(5400,1200){\line(1,0){3600}}
\put(9200,1200){\line(1,0){1600}}
\put(11000,1200){\line(1,0){3500}}
\multiput(1800,-900)(7300,0){2}{\line(0,-1){900}}
\put(1800,-1800){\line(1,0){7300}}
\multiput(3600,-1500)(0,300){3}{\line(0,1){150}}
\put(3600,-1500){\line(1,0){5400}}
\put(9200,-1500){\line(1,0){1700}}
\put(10900,-1500){\line(0,1){600}}
\put(5400,-900){\line(0,-1){300}}
\put(5400,-1200){\line(1,0){3600}}
\put(9200,-1200){\line(1,0){1600}}
\put(11000,-1200){\line(1,0){1700}}
\multiput(12700,-1200)(0,300){2}{\line(0,1){150}}
\multiput(0,600)(1800,0){2}{\usebox{\toe}}
\multiput(10800,600)(3600,0){2}{\usebox{\tow}}
\end{picture}\]
The spherical data and the Cartan pairing associated to this Luna diagram are described in Section~\ref{sss:typeB}. 
For convenience we number the spherical roots in the following way:
$$
\sigma_{2i-1} = \alpha_i, \quad \sigma_{2i} = \alpha'_i\quad
\text{for }i=1,\dots,s; \quad 
\sigma_{2s+1} = 
\sum_{i=s+1}^{s+t} 2\alpha'_i.
$$
There are $2s+2$ colors that we label in the following way:
\begin{align*}
D_{2i-1} &= D_{\alpha_i}^-, \quad\text{for }i=1,\dots,s, & D_{2s+1}& = D_{\gra'_{s}}^- \\
D_{2i}   &= D_{\alpha_i}^+, \quad\text{for }i=1,\dots,s, & D_{2s+2}& = D_{\gra'_{s+1}}.
\end{align*}  
The weights of these colors are the following:
\begin{align*}
\omega_{D_{2i-1}}&=\omega_{i}+\omega_{i-1}' \quad\text{ for }i=2,\dots,s & 
\omega_{D_{1}}& = \omega_1,\quad \omega_{D_{2s+1}} =\omega'_s,  \\
\omega_{D_{2i}}&=\omega_i+\omega_i' \quad\text{ for }i=1,\dots,s,& 
\omega_{D_{2s+2}}&=\tilde \omega'_{s+1}.
\end{align*}
where $\tilde \omega'_{s+1}=\omega'_{s+1}$ if $t>1$ and $\tilde\omega'_{s+1}=2\omega'_{s+1}$ if $t=1$.

Notice that $X$ has the same Cartan matrix of the spherical nilpotent orbit studied in \cite[Section~7.3]{BGM}. It follows that the covering differences and the fundamental low triples are the same as those computed therein, since they only depend on the Cartan matrix. In particular, every covering difference $\grg$ satisfies $\height(\grg^+) = 2$, and every fundamental triple is low. In order to prove the projective normality of $X$, in the following lemma we summarize some properties of its fundamental triples.

\begin{lemma}	\label{lemma:fund-triples-YYY}
Let $(D_p, D_q, F)$ be a fundamental triple, denote $\grg = D_p + D_q -F$ and suppose that $\grs_{2s+1} \in \supp_\grS(\grg)$. Then $p,q$ are even integers and $\grs_1 \not \in \supp_\grS(\grg)$. If moreover $\grs_2 \in \supp_\grS(\grg)$, then $p+q-3 \leq 2s+1$ and $F = D_1 + D_{p+q-3}$.
\end{lemma}

\begin{proof}
Take a sequence of coverings in $\mathbb N\Delta$
\[ F = F_{n+1} <_\Sigma F_n <_\Sigma\ldots <_\Sigma F_1 = D_p + D_q\] 
Denote $\grg_i = F_i- F_{i+1}$. By Proposition~3.2 and Proposition~7.3 in \cite{BGM} we have the following three possibilities:
\begin{enumerate}
\item $\grg_i = \grs_{p_i}+\grs_{p_i+2}+\ldots+\grs_{q_i-1}=D_{p_i}+D_{q_i}-D_{p_i-1}-D_{q_i+1}$, for some integers $p_i,q_i$ of different parity with $1\leq p_i< q_i\leq 2s+1$,
\item $\grg_i =\grs_{p_i-1}+\grs_{p_i} +\ldots+\grs_{q_i} = D_{p_i} + D_{q_i} - D_{p_i-2} - D_{q_i+2}$, for some integers $p_i,q_i$ of the same parity with $2\leq p_i\leq q_i \leq 2s$,
\item $\grg_i = \grs_{p_i} + \grs_{p_i+2}+ \ldots +\grs_{q_i-2} + 2(\grs_{q_i} + \grs_{q_i+2} + \ldots + \grs_{2s}) + \grs_{2s+1} = D_{p_i} + D_{q_i} - D_{p_i-1} - D_{q_i-1}$, for some even integers $p_i,q_i$ with $2\leq p_i\leq q_i \leq 2s$.
\end{enumerate}

Since $\grs_{2s+1} \in \supp_\grS(\grg)$, there is at least one $\grg_i$ of type 3. Let $k$ be minimal with $\grg_k$ of type 3, because of the parity of $p_k$ and $q_k$, the previous description implies that every $\grg_j$ with $j \neq k$ is of type 2. Moreover, it follows that $p_{i+1} = p_i-2$ and $q_{i+1} = q_i+2$ for all $i \neq k$, and that $p_i$, $q_i$ are even (resp.\ odd) for all $i \leq k$ (resp.\ $i > k$).

Therefore $p= p_1$ and $q= q_1$ are even integers and $2\leq p \leq q \leq 2s+2$, and we get the equalities $p_{n+1}= p-2n-1$ and $q_{n+1} = q+2n-1$. Suppose that $k = n$: then $p_n$ and $q_n$ are even and $2 \leq p_n \leq q_n \leq 2s+2$, hence $1 \leq p_{n+1} \leq q_{n+1} \leq 2s+1$. Suppose instead $k < n$, then $p_n$ and $q_n$ are odd and $2 \leq p_n \leq q_n \leq 2s$, and again we get $1 \leq p_{n+1} \leq q_{n+1} \leq 2s+1$.

To show the first claim, notice that $\grs_1 \in \supp_\grS(\grg)$ if and only if $\grs_1 \in \supp_\grS(\grg_{n})$: this is not the case if $k=n$, and if $k <n$ it cannot happen as well, since then $p_n$ and $q_n$ would be odd. Similarly, $\grs_2 \in \supp_\grS(\grg)$ if and only if $\grs_2 \in \supp_\grS(\grg_n)$ if and only if $p_{n+1} = 1$. This means $n=\frac p2-1$, which implies $q_{n+1} = p+q -3$.
\end{proof}

To prove the projective normality of $X$ we will apply Lemma~\ref{lemma: supporto moltiplicazione}. First we describe the invariants. Let $V = \mC^{s+1}$ with standard basis given by $e_1,\dots,e_{s+1}$.
Let $W=\mC^{2n+1}$ where $n=s+t$ and we choose a basis $e'_1,\dots,e'_n,e'_0,e'_{-n},\dots, e'_{-1}$   
and fix a bilinear symmetric form such that $\grb(e'_i,e'_j)=\grd_{i,-j}$ for all $i,j\geq 0$. 
Set $G=\mathrm{SL}(V^*)\times \mathrm{SO}(W,\grb)$, so that we can take $H$ as the stabilizer of the line 
spanned by the vector $e = e_1\otimes e'_0 + \sum_{i=2}^{s+1} e_{i}\otimes e'_{s-i+2}$.
We have
$$V^*_{D_{2i-1}}= \mathsf\Lambda^i V \otimes \mathsf\Lambda^{i-1} W \qquad V^*_{D_{2i}}  = \mathsf\Lambda^i V \otimes \mathsf\Lambda^i W$$
for $i=1,\dots,s+1$. If we denote by $h_i$ the vector $h_{D_i}\in V_{D_i}^*$ then in coordinates the vectors $h_i$ 
are given as follows:
\begin{align*}
h_{2i-1} =& \sum_{2\leq j_1<\ldots<j_{i-1}\leq s+1} e_1\wedge e_{j_1}\wedge\ldots\wedge e_{j_{i-1}} 
\otimes e'_{s-j_{i-1}+2}\wedge\ldots\wedge e'_{s-j_1+2},\\
h_{2i}   =& \sum_{2\leq j_1<\ldots<j_{i-1}\leq s+1} e_1\wedge e_{j_1}\wedge\ldots\wedge e_{j_{i-1}} 
\otimes e'_{s-j_{i-1}+2}\wedge\ldots\wedge e'_{s-j_1+2}\wedge e'_0 +\\
&+ \sum_{2\leq j_1<\ldots<j_{i}\leq s+1} e_{j_1}\wedge\ldots\wedge e_{j_{i}} 
\otimes e'_{s-j_{i}+2}\wedge\ldots\wedge e'_{s-j_1+2}.
\end{align*}

\begin{proposition}
The multiplication $m_{D,E}$ is surjective for all $D,E \in \mN \grD$.
\end{proposition}

\begin{proof}
As already noticed, Propositions~3.2 and~7.3 in \cite{BGM} show that every covering difference $\grg \in \mN\grS$ satisfies $\height(\grg^+) = 2$. It follows that every $D \in \grD$ is minimal in $\mN \grD$ w.r.t.\ $\leq_\grS$, hence every fundamental triple is low. Therefore by Lemma~\ref{lem:riduzionetriple} we have to check that $s^{D+E-F} V_F \subset V_D\cdot V_E$ for all fundamental triples $(D,E,F)$. Let $(D,E,F)$ be such a triple and denote $\grg = D+E-F$.

Suppose that $\grs_{2s+1} \notin \supp_\grS(\grg)$, and let $X'$ be the $G$-stable divisor of $X$ corresponding to the spherical root $\grs_{2s+1}$. Then $X'$ is a parabolic induction of a comodel wonderful variety of cotype $\sfA_{2s+1}$ (see \cite[Section~5]{BGM}). Hence the inclusion $s^\grg V_F \subset V_D\cdot V_E$ follows by Lemma~\ref{lem: projnorm parabolic induction} together with \cite[Theorem~5.2]{BGM}.

Suppose that $\grs_{2s+1} \in \supp_\grS(\grg)$, and assume $D = D_p$ and $E = D_q$, then $\grs_1 \not \in \supp_\grS(\grg)$ by Lemma~\ref{lemma:fund-triples-YYY}. We show that $s^\grg V_F \subset V_D\cdot V_E$ proceeding by induction on $s$.

Suppose that $\grs_2 \not \in \supp_\grS(\grg)$. Then $\supp_\grS(\grg) \subset \{\grs_3, \ldots, \grs_{2s+1}\}$. Let $X''$ be the $G$-stable subvariety of $X$ obtained by intersecting the $G$-stable divisors corresponding to $\grs_1$ and to $\grs_2$. If $s > 1$, then $X''$ is a parabolic induction of the wonderful variety of type $\mathsf{a}^\mathsf{y}(s-1, s-1)+\mathsf{b}'(t)$, therefore the multiplication of sections of globally generated line bundles on $X''$ is always surjective by the inductive hypothesis thanks to Lemma~\ref{lem: projnorm parabolic induction}. If instead $s =1$, then $X''$ is a parabolic induction of a rank 1 wondeful symmetric variety $Y$, which is homogeneous under its automorphism group, therefore the multiplication of sections of globally generated line bundles on $Y$ is always surjective, and the same holds for $X''$ by Lemma~\ref{lem: projnorm parabolic induction} again. In particular, since $(D,E,F)$ is a low triple, it follows the inclusion $s^\grg V_F \subset V_D\cdot V_E$.

Suppose now that $\grs_2 \in \supp_\grS(\grg)$. Then by Lemma \ref{lemma:fund-triples-YYY} it follows that $p = 2 \ell$ and $q = 2m$ are even integers, and $F = D_1+D_{2\ell+2m-3}$ with $2\ell+2m-3\leq2s+1$. Hence by Lemma~\ref{lemma: supporto moltiplicazione}
we need to find an equivariant map
$$ \grf\colon\big(\mathsf\Lambda^{\ell}V \otimes \mathsf\Lambda^{\ell}W\big) \otimes
\big(\mathsf\Lambda^{m}V\otimes \mathsf\Lambda^{m}W\big)
\lra V_{\omega_1+\omega_{\ell+m-1}}\otimes \mathsf\Lambda^{\ell+m-2}W$$
such that $\grf(h_{2\ell}\otimes h_{2m})\neq 0$ 
(the formula makes sense also when $\ell+m-1=s+1$, by setting $\omega_{s+1} = 0$).
Notice that $V\otimes \mathsf\Lambda^{\ell+m-1}V\isocan V_{\omega_1+\omega_{\ell+m-1}} \oplus \mathsf\Lambda^{\ell+m}V$ 
and we denote by $\rho_1$ and $\rho_2$ the projection respectively onto the first and onto the second factor. 
In particular the map $\rho_2$, up to a scalar factor, 
is just the wedge product. 
We will construct a map $$ \psi\colon\big(\mathsf\Lambda^{\ell}V \otimes \mathsf\Lambda^{\ell}W\big) \otimes
\big(\mathsf\Lambda^{m}V\otimes \mathsf\Lambda^{m}W\big)
\lra \big(V\otimes \mathsf\Lambda^{\ell+m-1}V \big)\otimes \mathsf\Lambda^{\ell+m-2}W$$
such that $\psi(h_{2\ell}\otimes h_{2m})\neq 0$ and $(\rho_2\otimes \mathrm{Id})\circ \psi(h_{2\ell}\otimes h_{2m})=0$ so that the map 
$\grf=(\rho_1\otimes \mathrm{Id})\circ \psi$ will have the desired properties.

Let $\pi_1\colon\mathsf\Lambda^{\ell}W\otimes \mathsf\Lambda^m W\lra \mathsf\Lambda^{\ell+m-2} W$ be defined by
\begin{align*}
\pi_1&(u_1\wedge \dots \wedge u_{\ell}\otimes v_1\wedge \dots \wedge v_m )\\ 
&=\sum_{i,j}(-1)^{i+j}\grb(u_i,v_j)
u_1\wedge \dots\wedge \hat u_i \wedge \dots\wedge u_{\ell+1}\wedge v_1\wedge \dots \wedge \hat v_j \wedge \dots \wedge v_m 
\end{align*}
Let $\pi_2\colon\mathsf\Lambda^{\ell}V\otimes \mathsf\Lambda^m V\lra V\otimes \mathsf\Lambda^{\ell+m-1} V$ be defined by
\begin{align*}
\pi_2&(u_1\wedge \dots \wedge u_{\ell}\otimes v_1\wedge \dots \wedge v_m )\\ 
&=\sum_{i}(-1)^{i}
u_i\otimes u_1\wedge \dots\wedge \hat u_i \wedge \dots\wedge u_{\ell}\wedge v_1\wedge \dots \wedge v_m 
\end{align*}
Let $\pi_3\colon\mathsf\Lambda^{\ell}V\otimes \mathsf\Lambda^m V\lra \Lambda^{\ell+m} V$ be defined by $\pi_3(x\otimes y)=x\wedge y$. Finally, set $\psi=\pi_2\otimes \pi_1$, so that $(\rho_2\otimes \mathrm{Id})\circ \psi = \pi_3\otimes \pi_1$.

Notice that the value of $\pi_2\otimes\pi_1$ (resp.\ $\pi_3\otimes\pi_1$) on $h_{2\ell}\otimes h_{2m}$ is the same as that on
\begin{align*}
&\sum_{2\leq j_1<\ldots<j_{\ell-1}\leq s+1} e_1\wedge e_{j_1}\wedge\ldots\wedge e_{j_{\ell-1}} 
\otimes e'_{s-j_{\ell-1}+2}\wedge\ldots\wedge e'_{s-j_1+2}\wedge e'_0\\
&\otimes\sum_{2\leq k_1<\ldots<k_{m-1}\leq s+1} e_1\wedge e_{k_1}\wedge\ldots\wedge e_{k_{m-1}} 
\otimes e'_{s-k_{m-1}+2}\wedge\ldots\wedge e'_{s-k_1+2}\wedge e'_0.
\end{align*} 

The first is equal to
$$
{\ell+m-2\choose \ell-1}\sum \big(e_1\otimes e_1\wedge e_{i_1}\wedge\ldots\wedge e_{i_{\ell+m-2}}\big) 
\otimes e'_{s-i_{\ell+m-2}+2}\wedge\ldots\wedge e'_{s-i_1+2}
$$
(the sum being over $2\leq i_1<\ldots<i_{\ell+m-2}\leq s+1$), the second is equal to zero.
\end{proof}


\subsection{Projective normality of $\mathsf{ab^y}(s, s)$} \label{ss:abyss}

Consider the wonderful variety $X$ for a semisimple group $G$ of type $\sfA_s \times \sfB_{s}$ 
with $s\geq2$
defined by the following spherical system
\[\begin{picture}(17700,4200)(-300,-2100)
\put(0,0){\usebox{\edge}}
\put(1800,0){\usebox{\susp}}
\put(5400,0){\usebox{\edge}}
\put(9900,0){\usebox{\edge}}
\put(11700,0){\usebox{\susp}}
\put(15300,0){\usebox{\rightbiedge}}
\multiput(0,0)(9900,0){2}{\multiput(0,0)(5400,0){2}{\multiput(0,0)(1800,0){2}{\usebox{\aone}}}}
\multiput(0,900)(9900,0){2}{\line(0,1){1200}}
\put(0,2100){\line(1,0){9900}}
\multiput(1800,900)(9900,0){2}{\line(0,1){900}}
\put(1800,1800){\line(1,0){8000}}
\put(10000,1800){\line(1,0){1700}}
\multiput(5400,900)(9900,0){2}{\line(0,1){600}}
\put(5400,1500){\line(1,0){4400}}
\put(10000,1500){\line(1,0){1600}}
\put(11800,1500){\line(1,0){3500}}
\multiput(7200,900)(9900,0){2}{\line(0,1){300}}
\put(7200,1200){\line(1,0){2600}}
\put(10000,1200){\line(1,0){1600}}
\put(11800,1200){\line(1,0){3400}}
\put(15400,1200){\line(1,0){1700}}
\multiput(1800,-900)(8100,0){2}{\line(0,-1){1200}}
\put(1800,-2100){\line(1,0){8100}}
\multiput(3600,-1800)(0,300){3}{\line(0,1){150}}
\put(3600,-1800){\line(1,0){6200}}
\put(10000,-1800){\line(1,0){1700}}
\put(11700,-1800){\line(0,1){900}}
\put(5400,-900){\line(0,-1){600}}
\put(5400,-1500){\line(1,0){4400}}
\put(10000,-1500){\line(1,0){1600}}
\put(11800,-1500){\line(1,0){1700}}
\multiput(13500,-1500)(0,300){2}{\line(0,1){150}}
\multiput(7200,-900)(8100,0){2}{\line(0,-1){300}}
\put(7200,-1200){\line(1,0){2600}}
\put(10000,-1200){\line(1,0){1600}}
\put(11800,-1200){\line(1,0){1600}}
\put(13600,-1200){\line(1,0){1700}}
\multiput(0,600)(1800,0){2}{\usebox{\toe}}
\put(5400,600){\usebox{\toe}}
\put(11700,600){\usebox{\tow}}
\multiput(15300,600)(1800,0){2}{\usebox{\tow}}
\end{picture}\]
The spherical data and the Cartan pairing associated to this Luna diagram are described in Section~\ref{sss:typeB}. 
The spherical roots are simple roots, for convenience we enumerate them in the following way:
$$
\sigma_{2i-1} = \alpha_i, \quad \sigma_{2i} = \alpha'_i\quad
\text{for }i=1,\dots,s.
$$
There are $2s+1$ colors that we label in the following way:
$$
D_{2i-1}= D_{\alpha_i}^-, \quad D_{2i} = D_{\alpha_i}^+ \quad\text{for }i=1,\dots,s;
\quad D_{2s+1} = D_{\gra'_{s}}^-.
$$  
The weights of these colors are the following:
\begin{eqnarray*}
&&\omega_{D_{1}} = \omega_1,\qquad\qquad
\omega_{D_{2i}}=\omega_i+\omega_i' \quad\text{ for }i=1,\dots,s,\\
&&\omega_{D_{2i-1}}=\omega_{i}+\omega_{i-1}' \quad\text{ for }i=2,\dots,s,\qquad\qquad
\omega_{D_{2s+1}} =\omega'_s.
\end{eqnarray*}

Notice that the $G$-stable divisor of $X$ corresponding to $\grs_{2s}$ is a parabolic induction of a comodel wonderful variety of cotype $\sfA_{2s}$ (see \cite[Section~5]{BGM}). Therefore we can restrict our study to the covering differences and the low triples of $X$ which contain $\grs_{2s}$.

\begin{lemma} 	\label{lemma:covering-ab^y(s,s)}
Let $\gamma\in\mathbb N\Sigma$ be a covering difference in $\mathbb N\Delta$ 
with $\sigma_{2s}\in\supp_\Sigma \gamma$.
Then either $\gamma=\sigma_{2s}=D_{2s}+D_{2s+1}-D_{2s-1}$, or $\grg = \grs_{2s-1}+ \grs_{2s} = -D_{2s-2} + 2D_{2s}$, or $\gamma=\sum_{i=\ell}^{s}\sigma_{2i}=D_{2\ell}-D_{2\ell-1}$
for some $1\leq\ell< s$. Every other covering difference satisfies $\height(\grg^+) = 2$.
\end{lemma}

\begin{proof}
Recall that the Cartan pairing is as follows 
(we also set $D_i=0$ for all $i\leq0$ and all $i\geq 2s+2$):
\begin{eqnarray*}
\sigma_i & = & -D_{i-1}+D_i+D_{i+1}-D_{i+2} \mbox{ for }i\neq 2s-2, \\
\sigma_{2s-2} & = & -D_{2s-3}+D_{2s-2}+D_{2s-1}-D_{2s}-D_{2s+1}.
\end{eqnarray*}

Set $\gamma=\sum_{i=1}^{2s}a_i\sigma_i=\sum_{i=1}^{2s+1}c_iD_i$. 
We have the following identities (we set $a_i=0$ for $i\leq0$ and $i\geq2s+1$)
\begin{eqnarray*}
c_i & = & -a_{i-2}+a_{i-1}+a_{i}-a_{i+1} \mbox{ for }i<2s+1, \\
c_{2s+1} & = & -a_{2s-2}-a_{2s-1}+a_{2s}.
\end{eqnarray*}

By hypothesis, we have $a_{2s}>0$.

Let $k \geq 0$ be minimal with $c_{2s-2k}>0$. Then $c_{2s-2i}\leq0$ for all $0\leq i<k$, and it follows
\[a_{2s-2j}-a_{2s-2j+1}\geq a_{2s-2j+2}-a_{2s-2j+3}\] 
for all $1\leq j\leq k$, and $a_{2s-2j}\geq a_{2s}>0$ for all $0\leq j\leq k$. 

If $k>0$, set $\gamma_0=\sum_{j=0}^k\sigma_{2s-2j}=-D_{2s-2k-1}+D_{2s-2k}$. Then $\gamma_0\leq_{\Sigma}\gamma$ and $\gamma^+-\gamma_0\in\mathbb N\Delta$, hence $\gamma=\gamma_0$ since $\grg$ is a covering difference.

We are left with the case $k=0$, in particular $c_{2s}>0$. We claim that $\gamma$ is necessarily equal to $\sigma_{2s}$ or to $\sigma_{2s-1}+\sigma_{2s}$.

Assume that $\gamma \neq \sigma_{2s}$ and $\grg \neq \sigma_{2s-1}+\sigma_{2s}$. Since $a_{2s}>0$ and $\sigma_{2s}=-D_{2s-1}+D_{2s}+D_{2s+1}$, it follows $c_{2s+1}\leq0$, hence $a_{2s-1}>0$.

Since $\sigma_{2s-1}=-D_{2s-2}+D_{2s-1}+D_{2s}-D_{2s+1}$, it must be $c_{2s-1}\leq0$, and since $\sigma_{2s-1}+\sigma_{2s}=-D_{2s-2}+2D_{2s}$, it must be $c_{2s}=1$.
The latter implies 
\[a_{2s-2}-a_{2s}=a_{2s-1}-1\geq0,\] 
hence $a_{2s-2}>0$. Notice that $c_{2s-2} \leq 0$: if indeed $c_{2s-2} > 0$, then $\grg^+ - 
(\sigma_{2s-2} + \sigma_{2s}) = \grg^+ - (-D_{2s-3} + D_{2s-2}) \in \mN \grD$, hence $\grg = -D_{2s-3} + D_{2s-2}$, contradicting $c_{2s} > 0$.

Let $j \geq 3$ be such that $a_{2s-j+1}-a_{2s-j+3}\geq0$, and suppose that $a_{2s-i+1}>0$ and $c_{2s-i+1}\leq0$ for all $i$ with $2\leq i\leq j$ (notice that these conditions have just been proved for $j=3$). As $c_{2s-j+2}\leq0$, it follows
\[a_{2s-j}-a_{2s-j+2}\geq a_{2s-j+1}-a_{2s-j+3}\geq0,\]
hence $a_{2s-j}>0$. 
This implies $c_{2s-j}\leq0$. If indeed $c_{2s-j}>0$, set 
\begin{eqnarray*}
\gamma_0&=&\sum_{i=0}^{j/2}\sigma_{2s-2i}=-D_{2s-j-1}+D_{2s-j}\mbox{ if $j$ is even},\\
\gamma_0&=&\sum_{i=1}^{(j+1)/2}\sigma_{2s-2i+1}=-D_{2s-j-1}+D_{2s-j}+D_{2s}-D_{2s+1}\mbox{ if $j$ is odd},
\end{eqnarray*}
then $\gamma_0\leq_{\Sigma}\gamma$ and $\gamma^+-\gamma_0\in\mathbb N\Delta$, hence $\gamma = \gamma_0$, contradicting $c_{2s}>0$ in the first case and $a_{2s}>0$ in the second case.

Applying this argument recursively for $j=3, \ldots, s-1$, it follows that $a_i>0$ for all $1\leq i\leq 2s$, $c_i\leq0$ for all $1\leq i\leq 2s-1$ and $a_1-a_3\geq0$. In particular $a_1+a_2-a_3=c_2\leq0$ and $a_1-a_3\geq0$, which are in contradiction. 

As already noticed, the $G$-stable divisor of $X$ corresponding to $\grs_{2s}$ is a parabolic induction of a comodel wonderful variety of cotype $\sfA_{2s}$. Therefore the covering differences $\grg$ with $\grs_{2s} \not \in \supp_\grS \grg$ coincide with those studied in \cite[Proposition~3.2]{BGM}, and they all satisfy $\height(\grg^+) = 2$.
\end{proof}

\begin{lemma}	\label{lemma:lowtriples-ab^y}
Let $(D,E,F)$ be a low fundamental triple, denote $\gamma=D+E-F$ 
and suppose that $\sigma_{2s}\in\supp_\Sigma\gamma$. 
Then we have the following possibilities:
\begin{itemize}
\item[-] $(D_{2m+1},D_{2s},D_{2m-1}+D_{2s+1})$ for $1\leq m< s$, $\gamma=\sum_{i=2m}^{2s}\sigma_i$;
\item[-] $(D_{2s},D_{2s},D_{2s-3})$, $\gamma=\sigma_{2s-2}+\sigma_{2s-1}+2\sigma_{2s}$;
\item[-] $(D_{2s},D_{2s},D_{2s-2})$, $\gamma=\sigma_{2s-1}+\sigma_{2s}$;
\item[-] $(D_{2s},D_{2s+1},D_{2s-1})$, $\gamma=\sigma_{2s}$.
\end{itemize}
\end{lemma}

\begin{proof}
Set $D+E-F=\sum_{i=1}^{2s}a_i\sigma_i=\sum_{i=1}^{2s+1}c_iD_i$, 
and set also $D=D_{2s-p+1}$ and $E=D_{2s-q+1}$. By hypothesis, we have $a_{2s}>0$.

At least one of the two indices $p$ and $q$ must be odd. 
Indeed, if both $p$ and $q$ were even, taking a sequence
\[ F = F_n <_\Sigma F_{n-1} <_\Sigma\ldots <_\Sigma F_0 = D+E\] 
of coverings in $\mathbb N\Delta$, 
$F_{i-1}-F_i$ would necessarily be a covering difference of a comodel spherical system of cotype $\mathsf A$,  
see \cite[Proposition~3.2.(2)]{BGM}, hence $\sigma_{2s}\not\in\supp_\grS(D+E-F)$.

We claim that at least one of the two indices $p$ and $q$ must be equal to $1$. 

Let us prove the claim. Assume both $p$ and $q$ are different from $1$. 
We can assume that $q$ is the minimal odd number between $p$ and $q$. 
Since $c_{2s-2i}\leq0$ for every $0\leq i <(q-1)/2$, as in the above proof, 
it follows that $a_{2s-2i} \geq a_{2s} > 0$ for every $0\leq i \leq (q-1)/2$.
Set 
\[\gamma_0 = \sum_{i= 0}^{(q-1)/2} \sigma_{2s-2i} = - D_{2s-q} +D_{2s-q+1}\] 
and $E'=D_{2s-q+1}-\gamma_0$, 
then $E'\in\mathbb N\Delta$ and $F\leq_\Sigma D+E' <_\Sigma D+E$, 
hence $(D,E,F)$ is not a low triple.

Therefore, we can assume $q=1$.

Suppose $p=0$. We have $D_{2s}+D_{2s+1}-\sigma_{2s}= D_{2s-1}$, but the latter is minuscule 
therefore we get only $(D_{2s},D_{2s+1},D_{2s-1})$.

Suppose $p=1$. We have $-a_{2s-2}-a_{2s-1}+a_{2s}\leq0$ and $-a_{2s-2}+a_{2s-1}+a_{2s}=2$, hence $a_{2s-1}>0$. 
Now, we have $2D_{2s}-(\sigma_{2s-1}+\sigma_{2s})=D_{2s-2}$, 
but in $\mathbb N\Delta$ the latter covers only $D_{2s-3}$, with $D_{2s-2}-(\sigma_{2s-2}+\sigma_{2s})=D_{2s-3}$.
Therefore, if $a_{2s-2}=0$ or $a_{2s}=1$ we get only $(D_{2s},D_{2s},D_{2s-2})$. 
If $a_{2s-2}>0$ and $a_{2s}>1$, since $D_{2s-3}$ is minuscule, we get only $(D_{2s},D_{2s},D_{2s-3})$.

Suppose $p>1$. We have $-a_{2s-2}-a_{2s-1}+a_{2s}\leq0$ and $-a_{2s-2}+a_{2s-1}+a_{2s}=1$, 
hence $a_{2s-1}>0$ and, since 
\[a_{2s-2}-a_{2s}=a_{2s-1}-1\geq0,\]
also $a_{2s-2}>0$.    
For every $1<i<p$, as $c_{2s-i+1}\leq0$, we have
\[a_{2s-i-1}-a_{2s-i+1}\geq a_{2s-i}-a_{2s-i+2}.\]
Therefore, $a_{2s-j+1}>0$ for all $1\leq j\leq p+1$.

If $p$ is odd, set
\[\gamma_0=\sum_{i=0}^{(p-1)/2}\sigma_{2s-2i}=-D_{2s-p}+D_{2s-p+1}\]
and $D'=D_{2s-p+1}-\gamma_0$, 
then $D'\in\mathbb N\Delta$ and $F\leq_\Sigma D'+E <_\Sigma D+E$, 
hence $(D,E,F)$ is not a low triple.

If $p$ is even, we have
\[D_{2s-p+1}+D_{2s}-\sum_{i=1}^{p+1}\sigma_{2s-i+1}=D_{2s-p-1}+D_{2s+1},\]
but the latter is minuscule. 
We get $(D_{2s-p+1},D_{2s},D_{2s-p-1}+D_{2s+1})$.
\end{proof}

To prove the projective normality of $X$ we will apply Lemma \ref{lemma: supporto moltiplicazione}. This requires some 
computations, and we will need an explicit description of the invariants $h_3$ and $h_{2s}$. 

Fix $V=\mC^{s+1}$ with standard basis $e_1,\dots,e_{s+1}$ and $W =\mC^{2s+1}$ with basis
$e'_1,\dots,e'_s$, $e'_0,e'_{-s},\dots,e'_{-1}$ and with a symmetric bilinear form $\grb$ such that $\grb(e_i,e_j)=\grd_{i,-j}$.
Set $G =\mathrm{SL}(V^*)\times\mathrm{Spin}(W,\grb)$, so that we can take $H$ as the stabilizer of the line spanned by the vector 
$e = e_1\otimes e'_0 + \sum_{i=2}^{s+1} e_{i}\otimes e'_{s-i+2}$.

We will need to use the spin module for the group $\mathrm{Spin}(W)$. 
Let us recall its construction. 
Let $W=U\oplus \mC e'_0 \oplus U^*$ 
where $U$ is the span of $e'_1,\dots,e'_s$ and 
$U^*$ is the span of $e'_{-s},\dots,e'_{-1}$ identified with the dual of $U$ by the form $\grb$. 
Let $S=\mathsf\Lambda U^*$ and rename the basis of $U^*$ as 
$\psi_n=e'_{-n},\ldots,\psi_1=e'_{-1}$. 
Define a map 
$\pi_S\colon W\otimes S\lra S$ by setting
$\pi_S(e'_i\otimes \psi_{i_1}\wedge \dots\wedge \psi_{i_k})$ equal to
\[\begin{array}{ll}
\sum_{j=1}^k (-1)^{j-1} \beta(e'_i,e'_{-i_j}) \psi_{i_1}\wedge \dots\wedge \hat \psi_{i_j} \wedge \dots\wedge \psi_{i_k}
& \mbox{if }i>0,\\
(-1)^k\psi_{i_1}\wedge \dots\wedge \psi_{i_k} & \mbox{if }i=0,\\
\psi_{-i} \wedge\psi_{i_1}\wedge \dots\wedge \psi_{i_k} & \mbox{if }i<0.
\end{array}\]
Then $\pi_S$ is $G$-equivariant, and its alternating square $\pi_S^2\colon\mathsf\Lambda^2W\otimes S\lra S$ corresponds to the spin representation via the isomorphism $\mathsf\Lambda^2W\cong\mathfrak{so}(W,\beta)$.
We have
\[V^*_{D_1}=V,\quad
V^*_{D_3}=\mathsf\Lambda^2V\otimes W,\quad
V^*_{D_{2s}}=\mathsf\Lambda^sV\otimes S,\quad
V^*_{D_1+D_{2s+1}}=V\otimes S.\]
The invariants, in coordinates, are
\begin{eqnarray*}
h_3&=&\sum_{i=2}^{s+1}e_1\wedge e_i\otimes e'_{s-i+2}\\
h_{2s}&=&e_2\wedge\ldots\wedge e_{s+1}\otimes 1 + \sum_{i=2}^{s+1}(-1)^{i-1}e_1\wedge\ldots\wedge\hat e_i\wedge\ldots\wedge e_{s+1}\otimes \psi_{s-i+2}
\end{eqnarray*}

\begin{proposition}	\label{prop:projnorm-ab^y}
The multiplication $m_{D,E}$ is surjective for all $D,E \in \mN \grD$.
\end{proposition}

\begin{proof}
By Lemma \ref{lemma:covering-ab^y(s,s)} every covering difference $\grg \in \mN \grS$ satisfies $\height(\grg^+) \leq 2$, therefore by Lemma~\ref{lem:riduzionetriple} it is enough to check that $s^{D+E-F}V_F\subset V_D\cdot V_E$ for all low fundamental triples $(D,E,F)$.

Suppose that $\grs_{2s} \notin \supp_\grS(D+E-F)$ and let $X'$ be the $G$-stable divisor of $X$ corresponding to $\grs_{2s}$. Then $X'$ is a parabolic induction of a comodel wonderful variety of cotype $\sfA_{2s}$, hence the inclusion $s^\grg V_F \subset V_D\cdot V_E$ follows by Lemma~\ref{lem: projnorm parabolic induction} together with \cite[Theorem~5.2]{BGM}.

We are left to check that $s^{D+E-F}V_F\subset V_D V_E$ for all low fundamental triples $(D,E,F)$ with $\sigma_{2s}\in\supp_\Sigma(D+E-F)$.

Consider first the triple $(D_3,D_{2s},D_1+D_{2s+1})$.
Then we have the projection $\pi\colon(\mathsf\Lambda^2V\otimes W)\otimes(\mathsf\Lambda^sV\otimes S)\to V\otimes S$
given by 
\[\pi((u_1\wedge u_2\otimes w)\otimes(v\otimes\psi))=((u_2\wedge v)u_1-(u_1\wedge v)u_2)\otimes\pi_s(w\otimes\psi),\]
where $\mathsf\Lambda^{s+1}V\cong\mathbb C$ via $e_1\wedge\ldots\wedge e_{s+1}\mapsto1$, and we get $\pi(h_3\otimes h_{2s})=s(e_1\otimes 1)\neq 0$.

We now proceed by induction on $s$.

Assume $s=2$. Then we are left to check the triples $(D_4,D_4,D_1)$, $(D_4,D_4,D_2)$ and $(D_4,D_5,D_3)$.

$(D_4,D_4,D_1)$. We have the projection $\pi\colon (\mathsf\Lambda^2 V\otimes S)\otimes(\mathsf\Lambda^2 V\otimes S)\to V$ given by $\pi(u_1\wedge u_2\otimes \varphi)\otimes(v_1\wedge v_2\otimes\psi))=\pi'(\varphi\otimes\psi)((u_1\wedge u_2\wedge v_1)v_2-(u_1\wedge u_2\wedge v_2)v_1)$, with $\mathsf\Lambda^3 V\cong \mathbb C$ given by the identification $e_1\wedge e_2\wedge e_3=1$, notice that $S$ is selfdual and set $\pi'\colon S\otimes S\to\mathbb C$ given by $\varphi\otimes\psi\mapsto\varphi\wedge\psi$ followed by the identification $\psi_2\wedge\psi_1=1$. We get $\pi(h_4\otimes h_4)=-2e_1\neq0$.

$(D_4,D_4,D_2)$ and $(D_4,D_5,D_3)$. Since $\grs_1,\grs_2 \notin \supp_\grS(D+E-F)$, 
we can consider the intersection $X'$ of the $G$-stable divisors of $X$ 
corresponding to the spherical roots $\grs_1,\grs_2$. Then the sections in $V_D,V_E, s^{D+E-F}V_F$ do not vanish on $X'$, so it is enough to prove that
$s^{D+E-F}V_F\subset m'_{D,E}(V_D\otimes V_E)$ where $m'$ denotes the
multiplication of sections on $X'$. 
Consider in $G$ the Levi subgroup $L$ associated to the roots $\gra_2,\gra'_2$, 
which has semisimple factor of type
$\sfA_{1}\times \sfA_{1}$,  and consider the comodel $L$-variety $Y$ of cotype $\mathsf A_3$. 
The wonderful $G$-variety $X'$ is obtained by parabolic induction from $Y$. Hence our claim follows  
by Lemma~\ref{lem: projnorm parabolic induction}. 

Assume $s>2$. Then we are left to check the triples $(D_{2m+1},D_{2s},D_{2m-1}+D_{2s+1})$ with $1< m< s$, $(D_{2s},D_{2s},D_{2s-3})$, $(D_{2s},D_{2s},D_{2s-2})$, $(D_{2s},D_{2s+1},D_{2s-1})$. Let $(D,E,F)$ be such a triple, then $\grs_1,\grs_2 \notin \supp_\grS(D+E-F)$ and we can consider the intersection $X'$ of the $G$-stable divisors of $X$ corresponding to the spherical roots $\grs_1,\grs_2$. 
Take the Levi subgroup $L$ of $G$ associated to the roots $\gra_2,\dots,\gra_s,\gra'_2,\dots,\gra'_{s}$, 
of semisimple type
$\sfA_{s-1}\times \sfB_{s-1}$,  and consider the comodel $L$-variety $Y$ of type $\mathsf{ab}^\mathsf y (s-1,s-1)$. 
The wonderful $G$-variety $X'$ is obtained by parabolic induction from $Y$. Hence our claim follows  
by Lemma~\ref{lem: projnorm parabolic induction} and the induction hypothesis.
\end{proof}

\section{Normality and semigroups}

Recall that we have fixed a maximal torus $T$ in $K$ and Borel subgroup $B$ of $K$ containing $T$. 
We use $\mathcal X(T)$ for the weight lattice of $T$.
 
Let us denote by $\Gamma(Z)$ the weight semigroup of a $K$-spherical variety $Z$,
$$
	\grG(Z) = \{\grl \in \mathcal X(T) \; : \; \Hom(\mC[Z], V(\grl)) \neq 0 \}.
$$
Let $Ke$ be a spherical nilpotent orbit in $\gop$, and let $\Sigma$ and $\Delta$ be the set of spherical roots and the set of colors of the wonderful compactification of $K/K_{[e]}$.
Let us denote by $D_\gop$ the element of $\mN\Delta$ such that $\gop=V^*_{D_\gop}$. Provided that the multiplication of sections of globally generated line bundles on the wonderful compactification of $K/K_{[e]}$ is surjective, we have that $\ol{Ke}\subset\gop$ is normal if and only if $D_{\gop}$ is minuscule in $\mN\Delta$ with respect to the partial order $\leq_\Sigma$. If moreover $\wt{Ke}$ is the normalization of $\ol{Ke}$, then
$$
	\grG(\wt{Ke}) = \bigcup_{n\in\mN}\{ \omega_E \, : \, E \in \mN\grD, \ E \leq_\grS nD_\gop\},
$$
that is, $\grG(\wt{Ke}) = \gro(\grG_{D_\gop})$ where $\grG_{D_\gop}$ is the subsemigroup of $\mN\Delta$ given by
$$\grG_{D_\gop} = \bigcup_{n\in\mN} \{E \in \mN \grD \; : \; E \leq_\grS nD_\gop\}.$$

In the present section we will study the normality of $\ol{Ke}$ and we will compute the weight semigroups $\grG(\wt{Ke})$ by computing the corresponding semigroups $\grG_{D_\gop}$. In particular we will prove the following theorem.

\begin{theorem}
Let $(\gog, \gok)$ be a classical symmetric pair of non-Hermitian type. Then $\ol{Ke}$ is not normal if and only if $(\gog,\gok) = (\mathfrak{so}(m+n),\mathfrak{so}(m) + \mathfrak{so}(n))$  and the signed partition of $Ke$ is $(+3,+2^{n-1},+1^{m-n-1})$, with $n>1$ odd, or $(-3,+2^{m-1},-1^{n-m-1})$, with $m>1$ odd. 
\end{theorem}

In Appendix \ref{A} these are the cases \ref{sss:BDbay}, with $r = p$, \ref{sss:BBaby}, with $r = q$, and \ref{sss:BBbay}, with $r = p$.

The normality or non-normality of $\ol{Ke}$, as well as the generators of the weight semigroup $\grG(\wt{Ke})$ are given in  Tables 2--11, in Appendix~\ref{B}. 

In the tables we also provide the codimension of  $\overline{Ke} \smallsetminus Ke$ in $\overline{Ke}$. Notice that, if $\overline{Ke}$ is normal and the codimension of $\overline{Ke} \smallsetminus Ke$ in $\overline{Ke}$ is greater than 1, then $\mathbb C[\overline{Ke}] = \mathbb C[Ke]$, so that the weight semigroup of $Ke$ actually coincides with $\grG(\wt{Ke})$. 

 Notice also that in all cases where $\ol{Ke}$ is not normal, the normalization $\wt{Ke} \lra \ol{Ke}$ is not even bijective (see \cite{Ga} for a general procedure to compute the $K$-orbits in $\wt{Ke}$ and in $\ol{Ke}$).

\begin{remark}
The normality of $\overline{Ge}$ is well known and may be deduced from \cite{KP2} (see also  \cite{Pa2}, when $Ge$ is spherical under the action of $G$). In particular, if $(\gog, \gok)$ is a classical symmetric pair of non-Hermitian type, then $\overline{Ge}$ is normal in all but the following cases:
\begin{itemize}
	\item[-] $\gog = \mathfrak{sp}(2n)$ with $n > 3$ and the partition of $Ge$ is $(3^2,1^{2n-6})$ (Cases \ref{sss:CCaac} and \ref{sss:CCaac2}, with $p+q > 3$),
           \item[-] $\gog = \mathfrak{sp}(2n)$ with $n > 6$ and the partition of $Ge$ is $(3^4,1^{2n-12})$ (Cases \ref{sss:CCabx} and \ref{sss:CCbax}, with $p+q>6$),
	\item[-] $\gog = \mathfrak{so}(2n+1)$ and the partition of $Ge$ is $(3,2^{n-1})$ (Case \ref{sss:BDbay}, with $r = p = q-1$).	
\end{itemize}
\end{remark}

We now report the details of the computation of the semigroup $\grG_{D_\gop}$. We omit the cases where $K/K_{[e]}$ is a flag variety (Cases \ref{sss:trvial1} and \ref{sss:trvial2} in Appendix \ref{A}) or a parabolic induction of a symmetric variety (Cases \ref{sss:symm1}, \ref{sss:symm2}, \ref{sss:Dao}, \ref{sss:symm4.1}, \ref{sss:BDaa}, \ref{sss:symm8}, \ref{sss:DDaa}, as well as Cases \ref{sss:BDady}, \ref{sss:BDbay}, \ref{sss:BBaby}, \ref{sss:BBbay}, \ref{sss:DDady}, \ref{sss:DDday} when $r=0$): in these cases the combinatorics of spherical systems is easier. By \cite{Ki06}, the normality of $\overline{Ke}$ was already known in all these cases (see the discussion at the beginning of Appendix \ref{B}), and the corresponding weight semigroups $\grG(\ol{Ke})$ were obtained in \cite{Bi} by using different techniques.

\subsection{Symplectic cases}\label{ss:symplecticcases}

\paragraph{\textbf{Cases \ref{sss:CCaac} $(q>1)$ and \ref{sss:CCaac2} $(p>1)$.}}
Let us deal with the case \ref{sss:CCaac} $(q>1)$, the other one is similar.
We have two spherical roots $\sigma_1=\alpha_1+\alpha'_1$ and $\sigma_2=\alpha'_1+2(\alpha'_2+\ldots+\alpha'_{q-1})+\alpha'_q$,
and three colors $D_1=D_{\alpha_1}$, $D_2=D_{\alpha'_2}$ and $D_3=D_{\alpha_2}$.

We have $D_\gop=D_1$, which is minuscule in $\mN \grD$, therefore $\ol{Ke}$ is normal. Furthermore $D_2+D_3=2D_1-\sigma_1$ and $D_3=2D_1-\sigma_1-\sigma_2$,
therefore $D_1, D_2+D_3, D_3\in\Gamma_{D_1}$.

Let us set $\grs = \sum a_i \grs_i\in\mN\Sigma$ and $nD_1 - \grs = \sum c_i D_i\in\mN\Delta$, we have
$$
	nD_1 - \grs = (n-2a_1) D_1 + (a_1-a_2)D_2 + a_1 D_3,
$$
and therefore $c_3-c_2 = a_2$. It follows that
$$
	\grG_{D_1} = \langle D_1, D_2+D_3, D_3\rangle_\mN.
$$

\paragraph{\textbf{Cases \ref{sss:CCacy} $(q>2)$ and \ref{sss:CCcay} $(p>2)$.}}
Let us deal with the case \ref{sss:CCacy} ($q>2$), the other one is similar. 
Let us keep the notation of Sections~\ref{sss:typeC} and~\ref{ss:ay22ct}, therefore $D_1 = D_{\gra_2}^-$, $D_2 = D_{\gra_2}^+$, $D_3 = D_{\gra_1}^-$, $D_4 = D_{\gra_1}^+$, $D_5 = D_{\gra_1'}^-$, $D_6 = D_{\gra'_3}$ and $D_7 = D_{\gra_3}$.

Then we have $D_\gop = D_4$, which is minuscule in $\mN \grD$, therefore $\ol{Ke}$ is normal. Moreover $D_2 = 2D_4 - \grs_3 - \grs_4$, $D_1 = D_2 - \grs_2 - \grs_4 - \grs_5$, therefore $D_1, D_2, D_4 \in \grG_{D_4}$. Moreover $D_3 + D_7 = D_1 + D_2 - \grs_1$, $D_5 + D_7 = D_2 + D_4 - \grs_1 - \grs_2 - \grs_3 - \grs_4$, $D_6 + D_7 = D_2 + D_4 - \grs_1 - \grs_2 - \grs_3 - \grs_4 - \grs_5$, therefore  $D_3 + D_7, D_5 + D_7, D_6 + D_7 \in \grG_{D_4}$ as well.

Let us set $\grs = \sum a_i \grs_i\in\mN\Sigma$ and $nD_4 - \grs = \sum c_i D_i\in\mN\Delta$, we have
\begin{align*}
	nD_4 - \grs = \, &(-a_1+a_2) D_1 + (-a_1-a_2+a_3)D_2 + (a_1-a_2-a_3+a_4)D_3+ \\
	& +(n+a_2-a_3-a_4)D_4 + (a_3-a_4+a_5)D_5 +(a_2-a_5) D_6+ a_1 D_7,
\end{align*}
and therefore $c_3+c_5 + c_6 = c_7$. It follows that
$$
	\grG_{D_4} = \langle D_1, D_2, D_4, D_3+D_7, D_5+D_7, D_6+D_7 \rangle_\mN.
$$

\paragraph{\textbf{Cases \ref{sss:CCacy} $(q=2)$ and \ref{sss:CCcay} $(p=2)$.}}
Let us deal with the case \ref{sss:CCacy} ($q=2$), the other one is similar. Label the spherical roots $\grs_1 = \gra_2$, $\grs_2 = \gra_2'$, $\grs_3 = \gra_1$, $\grs_4 = \gra_1'$, and label the colors $D_1 = D_{\gra_2}^-$, $D_2 = D_{\gra_2}^+$, $D_3 = D_{\gra_1}^-$, $D_4 = D_{\gra_1}^+$, $D_5 = D_{\gra_1'}^-$, $D_6 = D_{\gra_3}$.

Then we have $D_\gop = D_4$, which is minuscule in $\mN \grD$, therefore $\ol{Ke}$ is normal. Moreover $D_2 = 2D_4 - \grs_3 - \grs_4$, $D_1 = D_2 - \grs_2 - \grs_4$, therefore $D_1, D_2, D_4 \in \grG_{D_4}$. Similarly $D_3 + D_6 = D_1+ D_2 - \grs_1$ and $D_5 + D_6 = D_2 + D_4 - \grs_1 - \grs_2 - \grs_3 - \grs_4$, therefore  $D_3 + D_6, D_5 + D_6 \in \grG_{D_4}$ as well.

Let us set $\grs = \sum a_i \grs_i\in\mN\Sigma$ and $nD_4 - \grs = \sum c_i D_i\in\mN\Delta$, we have
\begin{align*}
	nD_4 - \grs = (-a_1+a_2) D_1 &+ (-a_1-a_2+a_3)D_2 + (a_1-a_2-a_3+a_4)D_3+ \\
	&+(n+a_2-a_3-a_4)D_4 + (a_2+a_3-a_4)D_5 +a_1 D_6,
\end{align*}
and therefore $c_3+c_5 = c_6$. It follows that 
$$
	\grG_{D_4} = \langle D_1, D_2, D_4, D_3+D_6, D_5+D_6 \rangle_\mN.
$$

\paragraph{\textbf{Cases \ref{sss:CCabx} and \ref{sss:CCbax}.}}

Let us deal with the case \ref{sss:CCabx}, the other one is similar. In this case $D_\gop$ is minuscule in $\mN\grD$. Following Example~\ref{example:counterexample}, this does not imply that the ring $\bigoplus_{n \in \mN} \grG(X,\calL_{nD_\gop})$ is generated by its degree one component $V_{D_\gop} = \grG(X,\calL_{D_\gop})$, indeed the multiplication of sections of globally generated line bundles on $X$ is not necessarily surjective. However using our methods we are still able to compute the normality and the weight semigroups of $\overline{Ke}$.

Enumerate the spherical roots and the colors of $X$ as in Example~\ref{example:counterexample}. Then $D_\gop = D_4$, and by definition
$$
	\grG(\overline{Ke}) = \bigcup_{n\in\mN}\{ \omega_E \, : \, E \in \mN\grD, \ V_E \subset V_{D_4}^n\}.
$$



\begin{lemma} The following inclusions hold:
\begin{itemize}
	\item[(1)] $V_{D_1} \subset V_{D_4}^2$ (where $D_1 = 2D_4 - \grs_2 - \grs_3 - 2\grs_4$),
	\item[(2)] $V_{D_2} \subset V_{D_4}^2$ (where $D_2 = 2D_4 - \grs_3 - \grs_4$),
	\item[(3)] $V_{D_3} \subset V_{D_1} V_{D_2}$ (where $D_3 = D_1+ D_2 - \grs_1$),
	\item[(4)] $V_{D_5} \subset V_{D_1} V_{D_4}$ (where $D_5 = D_1+D_4 - \grs_1 - \grs_3$),
	\item[(5)] $V_{D_6} \subset V_{D_1}^2$ (where $D_6 = 2D_1 - 2\grs_1 - \grs_3 -\grs_5$).
\end{itemize}
\end{lemma}

\begin{proof}
Consider the $G$-stable divisor $X' \subset X$ corresponding to $\grs_5$: it is a parabolic induction of a wonderful variety of type $\mathsf{ab^y}(2,2)$. By Proposition~\ref{prop:projnorm-ab^y} together with Lemma~\ref{lem: projnorm parabolic induction} it follows that the multiplication of sections is surjective for all pairs of globally generated line bundles on $X'$. Denote by $\rho' : \Pic(X) \lra \Pic(X')$ the restriction of line bundles and for $D \in \mN\grD$ set $D' = \rho'(D)$. By Lemma \ref{lemma:lowtriples-ab^y} $(D'_4, D'_4, D'_1)$ and $(D'_4, D'_4, D'_2)$ are low triples for $X'$, and since $D'_5 \leq D'_1 + D'_4$ and $D'_3 \leq D'_1 + D'_2$ are coverings in $\Pic(X')$ it follows that $(D'_1, D'_4, D'_5)$ and $(D'_1, D'_2, D'_3)$ are low triples for $X'$ as well. On the other hand, for all $D \in \mN\grD$ the $G$-modules $V_D$ and $V_{D'}$ are canonically identified, and since the restriction $\grG(X, \calL_D) \lra \grG(X', \calL_{D'})$ is surjective we get the inclusions (1), (2), (3), (4).

We are left with the inclusion (5). Consider the distinguished subset of colors $\grD_0 = \{D_2, D_3, D_4, D_5\}$ and denote by $Y$ the quotient of $X$ by $\grD_0$. Then $Y$ is a rank 1 wonderful variety with spherical root $2\grs_1 + \grs_3 + \grs_5$ whose set of colors is identified with $\{D_1, D_6\}$. By Lemma \ref{lem: projnorm quotients} we have that $\grG(X,\calL_{nD_1}) = \grG(Y,\calL_{nD_1})$ for all $n \in \mN$. Since $D_6 \leq 2D_1$ is a covering in $\Pic(Y)$, the triple $(D_1, D_1, D_6)$ is low in $\Pic(Y)$. On the other hand $Y$ is a parabolic induction of a rank one symmetric variety, and for such a variety the multiplication of sections of globally generated line bundles is known to be always surjective. By Lemma~\ref{lem: projnorm parabolic induction} the same holds for $Y$, and since it corresponds to a low triple we get the inclusion $V_{D_6} \subset V_{D_1}^2$.
\end{proof}

\begin{proposition}
$\ol{Ke}$ is normal, and $\grG(\ol{Ke})$ is generated by the weights
$$\gro_2, \ \gro_4, \ \gro_1 + \gro_1', \ \gro_2 + \gro_2', \ \gro_1 + \gro_3 + \gro_2', \ \gro_3 + \gro_1'.$$
\end{proposition}

\begin{proof}
Clearly, $\grG(\overline{Ke}) \subset \omega(\mN\grD)$. On the other hand by the previous lemma we have that $\gro(D) \in \grG(\overline{Ke})$ for all $D \in \grD$, therefore $\grG(\overline{Ke}) = \omega(\mN\grD)$ and the description of the generators follows by the description of the map $\gro$.

Notice that the weights $\gro(D_1), \ldots, \gro(D_6)$ are linearly independent. Therefore $\grG(\overline{Ke})$ is a saturated semigroup of weights (that is, if $\grG_\mZ \subset \calX(T)$ is the sublattice generated by $\grG(\overline{Ke})$ and if $\grG_{\mQ^+} \subset \calX(T) \otimes \mQ$ is the cone generated by $\grG(\overline{Ke})$, then $\grG(\overline{Ke}) = \grG_\mZ \cap \grG_{\mQ^+}$). It follows that $\ol{Ke}$ is normal. 
\end{proof}

\subsection{Orthogonal cases}

\subsubsection{Tail cases}

\paragraph{\textbf{Cases without Roman numerals.}} Let us deal with the case \ref{sss:BDady} ($r<q-1$). The cases \ref{sss:BDbay} ($r<\min\{p,q-1\}$), \ref{sss:BBaby} $(r<q)$, \ref{sss:BBbay} $(r<p)$, \ref{sss:DDady} ($r<\min\{p-1,q-1\}$) and \ref{sss:DDday} ($r<\min\{p-1,q-1\}$) are similar. Suppose $r>0$, otherwise we have a symmetric case. Let us keep the notation of Section~\ref{sss:typeD}, therefore, for all $i=1,\ldots,r$, $D_{2i-1} = D_{\gra_i}^-$ and $D_{2i} = D_{\gra_i}^+$, furthermore $D_{2r+1} = D_{\gra'_r}^-$, $D_{2r+2} = D_{\gra'_{r+1}}$ and $D_{2r+3} = D_{\gra_{r+1}}$. To have a uniform notation, we denote
$$
\tilde D_{2r+3} = \left\{
	\begin{array}{ll}
		D_{2r+3} & \text{if } r<p-1  \\
		2D_{2r+3} & \text{if } r=p-1 
	\end{array} \right..
$$
We have $D_\gop = D_2$, which is minuscule in $\mN \grD$, therefore $\ol{Ke}$ is normal. 

Set $\tilde D_1 = D_1$, $\tilde D_2 = D_2$ and, for all $k = 3, \ldots, 2r+2$,
$$
	\tilde D_k = D_2 + D_{k-2} - (\grs_1 + \ldots + \grs_{k-2}).
$$
Notice that
$$
\tilde D_k = \left\{ \begin{array}{ll}
		D_k  & \text{if } k \leq 2r \\
		D_k + \tilde D_{2r+3} & \text{if } k = 2r+1, 2r+2
\end{array} \right..
$$

\begin{proposition}
The semigroup $\grG_{D_2}$ is generated by $\tilde D_{2i}$ and $\tilde D_{2i-1}+\tilde D_{2j-1}$ for all $i,j=1, \ldots, r+1$.
\end{proposition}

\begin{proof}
Since $D_2 \in \grG_{D_2}$, by induction on the even indices, 
it follows that $\tilde D_{2i} \in \grG_{D_2}$ for all $i \leq r+1$. 
On the other hand
$$
	D_1 = D_2 - (\grs_2 + \grs_4 + \ldots + \grs_{2r-2} + \grs_{2r} + \frac{\grs_{2r+1}}{2}).
$$
Therefore, for the odd indices, we get $\tilde D_{2i-1} + \tilde D_{2j-1} \in \grG_{D_2}$ for all $i,j \leq r+1$.

Let $\grs \in \mN \grS$ and suppose that $nD_2 - \grs \in \mN \grD$, let $\grs = \sum a_i \grs_i$ and $nD_2 - \grs = \sum c_i D_i$. Notice that if $r=p-1$ then $c_{2r+3}$ is even. Therefore $nD_2 - \grs \in \langle D_1, \ldots, D_{2r+2}, \tilde D_{2r+3}\rangle_\mN$, and write
$$nD_2 - \grs = b_1 D_1 + \ldots + b_{2r+2} D_{2r+2} + b_{2r+3} \tilde D_{2r+3}.$$

Expressing the coefficients $b_1, \ldots, b_{2r+3}$ with respect to $a_1, \ldots, a_{2r+1}$ we get that $\sum_{i=1}^{r+1} b_{2i-1} = 2 a_{2r+1}$, and $b_{2r+1} + b_{2r+2} = b_{2r+3}$. The claim follows.
\end{proof}

\paragraph{\textbf{Cases with Roman numerals.}} Let us deal with the case \ref{sss:DDady} ($r=p-1<q-1$), the case (I). The case (II), as well as the cases \ref{sss:BDbay} ($r=q-1<p$) and \ref{sss:DDday} ($r=q-1<p-1$) are similar. First, let us suppose $r>1$. Let us keep the notation of Section~\ref{sss:typeD}, therefore, for all $i=1,\ldots,r$, $D_{2i-1} = D_{\gra_i}^-$ and $D_{2i} = D_{\gra_i}^+$, furthermore $D_{2r+1} = D_{\gra'_r}^-$, $D_{2r+2} = D_{\gra'_{r+1}}$ and $D_{2r+3} = D_{\gra_{p}}$. 

We have $D_\gop = D_2$, which is minuscule in $\mN \grD$, therefore $\ol{Ke}$ is normal. 

Set $\tilde D_1 = D_1$ and $\tilde D_2 = D_2$, and define inductively, for all $k = 3, \ldots, 2r+2$,
$$
	\tilde D_k = D_2 + \tilde D_{k-2} - (\grs_1 + \ldots + \grs_{k-2}).
$$
Notice that
$$
\tilde D_k = \left\{ \begin{array}{ll}
		D_k  & \text{if } k \leq 2r-2 \\
		D_k + D_{2r+3} & \text{if } k= 2r-1, 2r \\
		D_k + 2D_{2r+3} & \text{if } k= 2r+1, 2r+2
\end{array} \right..
$$

\begin{proposition}
The semigroup $\grG_{D_2}$ is generated by $\tilde D_{2i}$ and $\tilde D_{2i-1}+\tilde D_{2j-1}$ for all $i,j=1, \ldots, r+1$.
\end{proposition}

\begin{proof}
Since $D_2 \in \grG_{D_2}$, for the even indices, 
it follows that $\tilde D_{2i} \in \grG_{D_2}$ for all $i \leq r+1$. 
On the other hand
$$
	D_1 = D_2 - (\grs_2 + \grs_4 + \ldots + \grs_{2r-2} + \grs_{2r} + \frac{\grs_{2r+1}}{2}).
$$
Therefore, for the odd indices, we get $\tilde D_{2i-1} + \tilde D_{2j-1} \in \grG_{D_2}$ for all $i,j \leq r+1$.

Let $\grs \in \mN \grS$ and suppose that $nD_2 - \grs \in \mN \grD$, denote $\grs = \sum a_i \grs_i$ and $nD_2 - \grs = \sum c_i D_i$. Expressing the coefficients $c_1, \ldots, c_{2r+3}$ with respect to $a_1, \ldots, a_{2r+1}$ we get that $\sum_{i=1}^{r+1} c_{2i-1} = 2 a_{2r+1}$, and $c_{2r-1}+c_{2r}+2c_{2r+1} + 2c_{2r+2} = c_{2r+3}$. The claim follows.
\end{proof}

We are left with the case $r=1$, the case $r=0$ being symmetric. Here we have $D_\gop=D_2+D_5$, which is minuscule in $\mN \grD$, therefore $\ol{Ke}$ is normal. Proceeding as above we get the same semigroup
$$
\Gamma_{D_2+D_5}=\langle D_2+D_5, D_4+2D_5, 2D_1+2D_5, D_1+D_3+3D_5, 2D_3+4D_5\rangle_\mN.  
$$

\subsubsection{Collapsed tails of type $\mathsf B$}

\paragraph{\textbf{Cases without Roman numerals.}}
Let us deal with the case \ref{sss:BBaby} ($r=q$). 
The cases \ref{sss:BDbay} ($r=p<q-1$) and \ref{sss:BBbay} ($r=p$) are similar. 
Let us keep the notation of Sections~\ref{sss:typeB} and~\ref{ss:abyss}, 
therefore $D_{2i-1} = D_{\gra_i}^-$ and $D_{2i} = D_{\gra_i}^+$ for all $i = 1, \ldots, r$,
$D_{2r+1}=D_{\gra'_q}^-$ and $D_{2r+2}=D_{\gra_{r+1}}$.

We have $D_\gop = D_2$. Notice that $D_2$ is not minuscule, indeed $D_1 = D_2 - \sum_{i=1}^r \grs_{2i}$. Therefore $\ol{Ke}$ is not normal. 

To have a uniform notation set 
$$
\tilde D_{2r+2} = \left\{
	\begin{array}{ll}
		D_{2r+2} & \text{if } r<p-1\\
		2D_{2r+2} & \text{if } r=p-1 
	\end{array} \right.
$$
Set $\tilde D_1 = D_1$ and $\tilde D_2 = D_2$, and define inductively for all $i=3,\ldots,2r+1$
$$
	\tilde D_i = D_2 + \tilde D_{i-2} - (\grs_1 + \ldots + \grs_{i-2}).
$$
Notice that 
$$
\tilde D_i = \left\{\begin{array}{ll}
	D_i & \text{if } i\leq 2r-1 \\
	D_{2r} + D_{2r+1} & \text{if } i=2r \\
	2D_{2r+1}+\tilde D_{2r+2} & \text{if } i=2r+1
	\end{array}	\right..
$$

\begin{proposition}
The semigroup $\grG_{D_2}$ is generated by $\tilde D_1, \ldots, \tilde D_{2r+1}$.
\end{proposition}

\begin{proof}
Let $\grs \in \mN \grS$ and suppose that $nD_2 - \grs \in \mN \grD$, denote $\grs = \sum a_i \grs_i$ and $nD_2 - \grs = \sum c_i D_i$. Notice that if $r=p-1$ then $c_{2r+3}$ is even. Therefore $nD_2 - \grs \in \langle D_1, \ldots, D_{2r+2}, \tilde D_{2r+3}\rangle_\mN$, and write
$$nD_2 - \grs = b_1 D_1 + \ldots + b_{2r+2} D_{2r+2} + b_{2r+3} \tilde D_{2r+3}.$$
Expressing the spherical roots in terms of colors it follows that $b_{2r+1} = b_{2r} + 2b_{2r+2}$. 
The claim follows.
\end{proof}

%
%

\paragraph{\textbf{Case with Roman numerals.}}
We deal here with the case \ref{sss:BDbay} ($r=p=q-1$), the case (I), the case (II) is similar. The notation will be slightly different than before: let us enumerate the spherical roots $\grS = \{\grs_1, \ldots, \grs_{2r}\}$ as follows
$$
	\grs_{2i-1} = \alpha'_i, \qquad \grs_{2i} = \gra_i, \qquad \mbox{for all }i = 1, \ldots, r.
$$
Accordingly, we enumerate the colors as follows: 
$D_{2i-1} = D_{\gra'_i}^-$ and $D_{2i} = D_{\gra'_i}^+$ for all $i = 1, \ldots, r$, 
$D_{2r+1} = D_{\gra_p}^-$ and $D_{2r+2} = D_{\gra'_{q}}$.

We have $D_\gop = D_2$. Notice that $D_2$ is not minuscule, indeed $D_1 = D_2 - \sum_{i=1}^r \grs_{2i}$. Therefore $\ol{Ke}$ is not normal. 

Set $\tilde D_1 = D_1$ and $\tilde D_2 = D_2$, and define inductively for all $i=3,\ldots,2r+1$
$$
	\tilde D_i = D_2 + \tilde D_{i-2} - (\grs_1 + \ldots + \grs_{i-2}).
$$
Notice that 
$$
\tilde D_i = \left\{\begin{array}{ll}
	D_i & \text{if } i\leq 2r-2 \\
	D_{2r-1} + D_{2r+2} & \text{if } i=2r-1 \\
	D_{2r} + D_{2r+1} + D_{2r+2} & \text{if } i=2r \\
	2D_{2r+1} + 2D_{2r+2} & \text{if } i=2r+1
	\end{array}	\right..
$$

\begin{proposition}
The semigroup $\grG_{D_2}$ is generated by $\tilde D_1, \ldots, \tilde D_{2r+1}$.
\end{proposition}

\begin{proof}
Let $\grs \in \mN \grS$ and suppose that $nD_2 - \grs \in \mN \grD$, let $\grs = \sum a_i \grs_i$ and $nD_2 - \grs = \sum c_i D_i$. Expressing the spherical roots in terms of colors it follows that 
$c_{2r+1} - c_{2r}=2a_{2r-1}$ and 
$c_{2r-1} + c_{2r+1} = c_{2r+2}$.
The claim follows.
\end{proof}

\subsubsection{Collapsed tails of type $\mathsf D$}

\paragraph{\textbf{Cases without Roman numerals.}} Let us deal with the case \ref{sss:BDady} ($r=q-1$). The cases  \ref{sss:DDady} ($r=q-1<p-1$) and \ref{sss:DDday} ($r=p-1<q-1$) are similar. Let us keep the notation of Section~\ref{sss:typeD}, therefore, for all $i=1,\ldots,r$, $D_{2i-1} = D_{\gra_i}^-$ and $D_{2i} = D_{\gra_i}^+$, furthermore $D_{2r+1} = D_{\gra'_r}^-$, $D_{2r+2} = D_{\gra'_{r+1}}^-$ and $D_{2r+3} = D_{\gra_{r+1}}$. To have a uniform notation, we denote
$$
\tilde D_{2r+3} = \left\{
	\begin{array}{ll}
		D_{2r+3} & \text{if } r<p-1  \\
		2D_{2r+3} & \text{if } r=p-1 
	\end{array} \right..
$$
We have $D_\gop = D_2$, which is minuscule in $\mN \grD$, therefore $\ol{Ke}$ is normal.

Set $\tilde D_1 = D_1$, $\tilde D_2 = D_2$ and, for all $k = 3, \ldots, 2r+2$,
$$
	\tilde D_k = D_2 + D_{k-2} - (\grs_1 + \ldots + \grs_{k-2}).
$$
Notice that
$$
\tilde D_k = \left\{ \begin{array}{ll}
		D_k  & \text{if } k \leq 2r \\
		D_k + D_{2r+2} + \tilde D_{2r+3} & \text{if } k = 2r+1, 2r+2
\end{array} \right..
$$
Furthermore, set
$$
	\tilde D'_{2r+2} = D_2 + D_{2r} - (\grs_1 + \ldots + \grs_{2r-1} + \grs_{2r+1})=2D_{2r+1}+\tilde D_{2r+3}.
$$

\begin{proposition}
The semigroup $\grG_{D_2}$ is generated by $\tilde D_{2i}$, $\tilde D_{2i-1}+\tilde D_{2j-1}$ for all $i,j=1, \ldots, r+1$ with $i+j \leq 2r+1$, and $\tilde D'_{2r+2}$.
\end{proposition}

\begin{proof}
Since $D_2 \in \grG_{D_2}$, for the even indices, 
it follows that $\tilde D_{2i} \in \grG_{D_2}$ for all $i \leq r+1$, and $\tilde D'_{2r+2}\in\grG_{D_2}$ as well. 
On the other hand
$$
	D_1 = D_2 - (\grs_2 + \grs_4 + \ldots + \grs_{2r-2} + \frac{\grs_{2r} + \grs_{2r+1}}{2}).
$$
Therefore, for the odd indices, we get $\tilde D_{2i-1} + \tilde D_{2j-1} \in \grG_{D_2}$ for all $i,j \leq r+1$.

Let $\grs \in \mN \grS$ and suppose that $nD_2 - \grs \in \mN \grD$, denote $\grs = \sum a_i \grs_i$ and $nD_2 - \grs = \sum c_i D_i$. Notice that if $r=p-1$ then $c_{2r+3}$ is even. Therefore $nD_2 - \grs \in \langle D_1, \ldots, D_{2r+2}, \tilde D_{2r+3}\rangle_\mN$, and write
$$nD_2 - \grs = b_1 D_1 + \ldots + b_{2r+2} D_{2r+2} + b_{2r+3} \tilde D_{2r+3}.$$

Expressing the coefficients $b_1, \ldots, b_{2r+3}$ with respect to $a_1, \ldots, a_{2r+1}$ we get that $\sum_{i=1}^{r+1} b_{2i-1} = 2 a_{2r+1}$, and $b_{2r+1} + b_{2r+2} = 2b_{2r+3}$. The claim follows.
\end{proof}

\paragraph{\textbf{Cases with Roman numerals.}} Let us deal with the case \ref{sss:DDady} ($r=p-1=q-1$), the case (I). The case (II) as well as the cases \ref{sss:DDday} ($r=q-1=p-1$) are similar. First, let us suppose $r>1$. Let us keep the notation of Section~\ref{sss:typeD}, therefore, for all $i=1,\ldots,r$, $D_{2i-1} = D_{\gra_i}^-$ and $D_{2i} = D_{\gra_i}^+$, furthermore $D_{2r+1} = D_{\gra'_{q-1}}^-$, $D_{2r+2} = D_{\gra'_{q}}^-$ and $D_{2r+3} = D_{\gra_{p}}$. 

We have $D_\gop = D_2$, which is minuscule in $\mN \grD$, therefore $\ol{Ke}$ is normal. 

Set $\tilde D_1 = D_1$ and $\tilde D_2 = D_2$, and define inductively, for all $k = 3, \ldots, 2r+2$,
$$
	\tilde D_k = D_2 + \tilde D_{k-2} - (\grs_1 + \ldots + \grs_{k-2}).
$$
Notice that
$$
\tilde D_k = \left\{ \begin{array}{ll}
		D_k  & \text{if } k \leq 2r-2 \\
		D_k + D_{2r+3} & \text{if } k= 2r-1, 2r \\
		D_k + D_{2r+2}+2D_{2r+3} & \text{if } k= 2r+1, 2r+2
\end{array} \right..
$$
Furthermore, set
$$
	\tilde D'_{2r+2} = D_2 + \tilde D_{2r} - (\grs_1 + \ldots + \grs_{2r-1} + \grs_{2r+1})=2D_{2r+1}+2D_{2r+3}.
$$

\begin{proposition}
The semigroup $\grG_{D_2}$ is generated by $\tilde D_{2i}$, $\tilde D_{2i-1}+\tilde D_{2j-1}$ for all $i,j=1, \ldots, r+1$ with $i+j \leq 2r+1$, and $\tilde D'_{2r+2}$.
\end{proposition}

\begin{proof}
Since $D_2 \in \grG_{D_2}$, for the even indices, 
it follows that $\tilde D_{2i} \in \grG_{D_2}$ for all $i \leq r+1$, and $\tilde D'_{2r+2}\in\grG_{D_2}$ as well. 
On the other hand
$$
	D_1 = D_2 - (\grs_2 + \grs_4 + \ldots + \grs_{2r-2} + \frac{\grs_{2r} + \grs_{2r+1}}{2}).
$$
Therefore, for the odd indices, we get $\tilde D_{2i-1} + \tilde D_{2j-1} \in \grG_{D_2}$ for all $i,j \leq r+1$.

Let $\grs \in \mN \grS$ and suppose that $nD_2 - \grs \in \mN \grD$, let $\grs = \sum a_i \grs_i$ and $nD_2 - \grs = \sum c_i D_i$. Expressing the coefficients $c_1, \ldots, c_{2r+3}$ with respect to $a_1, \ldots, a_{2r+1}$ we get that $\sum_{i=1}^{r+1} c_{2i-1} = 2 a_{2r+1}$, and $c_{2r-1}+c_{2r}+c_{2r+1} + c_{2r+2} = c_{2r+3}$. The claim follows.
\end{proof}

We are left with the case $r=1$, the case $r=0$ being symmetric. Here we have $D_\gop=D_2+D_5$, which is minuscule in $\mN \grD$, therefore $\ol{Ke}$ is normal. Proceeding as above we get the same semigroup, that is, $\Gamma_{D_2+D_5}$ is generated by $D_2+D_5$, $2D_4+2D_5$, $2D_3+2D_5$, $2D_1+2D_5$, $D_1+D_3+D_4+3D_5$.


\appendix

\section{List of spherical nilpotent $K$-orbits  in $\mathfrak p$ in the classical non-Hermitian cases}\label{A}

\renewcommand{\thesubsection}{\arabic{subsection}}

Here we report the list of the spherical nilpotent
$K$-orbits in $\mathfrak p$ for all symmetric pairs $(\mathfrak
g,\mathfrak k)$ of classical non-Hermitian type.

Every $K$-orbit in $\mathfrak p$ is labelled with the signed partition of the corresponding real nilpotent orbit, via the Kostant-Sekiguchi-\makebox[0pt]{\rule{3pt}{0pt}\rule[4pt]{3pt}{0.8pt}}Dokovi\'c bijection.

For every orbit we provide a representative $e$, as well as a normal triple containing it $\{h,e,f\}$.

For all $i\in\mathbb Z$, let $\mathfrak k(i)$ be the $\mathrm{ad}h$-eigenspace in $\mathfrak k$ of eigenvalue $i$. We denote by $Q$ the parabolic subgroup of $K$ whose Lie algebra is equal to
\[\mathrm{Lie}\,Q=\bigoplus_{i\geq0}\mathfrak k(i).\]

In each case we describe the centralizer of $h$, which we denote by $K_h$ or by $L$, which is a Levi subgroup of $Q$. We denote by $Q^\mathrm u$ the unipotent radical of $Q$. Then we describe the centralizer of $e$, which we denote by $K_e$. A Levi subgroup of $K_e$ is always given by $L_e$, the centralizer of $e$ in $L$. The unipotent radical of $K_e$ is either equal to $Q^\mathrm u$ or equal to a co-simple $L_e$-submodule of $Q^\mathrm u$. In the latter case, there always exist some simple $L_e$-submodules in $\mathfrak k(1)$, say $W_0,\ldots,W_d$, which we determine, with the following properties. They are isomorphic as $L_e$-modules but lie in pairwise distinct isotypical $L$-components. Denoting by $V$ the $L_e$-complement of $W_0\oplus \ldots \oplus W_d$ in $\mathrm{Lie}\,Q^\mathrm u$, as $L_e$-module, 
\[\mathrm{Lie}\,K_e^\mathrm u=W\oplus V,\]
where $W$ is a co-simple $L_e$-submodule of $W_0\oplus\ldots\oplus W_d$ which projects non-trivially on every summand $W_0,\ldots,W_d$. Actually, the integer $d+1$, the number of the above simple $L_e$-modules $W_0,\ldots,W_d$, will only be equal to 2 or 3.

\begin{remark}\label{remark:missingcase}
As already mentioned, the list of spherical nilpotent $K$-orbits in $\mathfrak p$ is in \cite{Ki04}, and all the data in our list, such as a representative and its centralizer, can be directly computed using the information contained therein, with one exception. There is one missing case in \cite{Ki04}, corresponding to the signed partition $(+3^4,+1^{2n-8})$ for the symmetric pair $(\mathfrak{sp}(2n+4), \mathfrak{sp}(2n) + \mathfrak{sp}(4))$ with $n\geq4$ (cases \ref{sss:CCabx} and \ref{sss:CCbax} in Appendix \ref{A}). The lack comes from a small mistake in \cite[Lemma~7.2]{Ki04}, we have checked that there is no further missing case arising from that lemma. The smallest case of this family, which is for $n=4$, was already present in \cite[Example~5.8]{Pa1}.
\end{remark}

\subsection{$\mathfrak{sl}(2n)/\mathfrak{sp}(2n)$}

$K=\mathrm{Sp}(2n)$, $n\geq 2$, $\mathfrak p=V(\omega_2)$.

Let us fix a basis $e_1,\ldots,e_n,e_{-n},\ldots,e_{-1}$ of $\mathbb C^{2n}$, a skew-symmetric bilinear form $\omega$ such that $\omega(e_i,e_j)=\delta_{i,-j}$ for $1\leq i\leq n$ and $K=\mathrm{Sp}(\mathbb C^{2n},\omega)$. Then $\omega$ can be seen as a linear form on $\mathsf\Lambda^2\mathbb C^{2n}$, $\omega(e_i\wedge e_j)=\omega(e_i,e_j)$, and 
\[\mathfrak p=\ker\omega\subset\mathsf\Lambda^2\mathbb C^{2n}.\]

\subsubsection{$\mathbf{(2^{r},1^{n-2r})}$, $r\geq1$}\label{sss:symm1}

\[e=\sum_{i=1}^re_i\wedge e_{2r-i+1},\quad
h(e_i)=\left\{\begin{array}{cl}
e_i & \mbox{if $1\leq i\leq 2r$}\\
-e_i & \mbox{if $-2r\leq i\leq -1$}\\
0 & \mbox{otherwise}
\end{array}\right.,\quad
f=\sum_{i=1}^re_{-2r+i-1}\wedge e_{-i}.\] 
Let $Q=L\,Q^\mathrm u$ be the corresponding parabolic subgroup of $K$,
so that $L=K_h\cong\mathrm{GL}(2r)\times\mathrm{Sp}(2n-4r)$.

The centralizer of $e$ is $K_e=L_eQ^\mathrm u$ where
$L_e\cong\mathrm{Sp}(2r)\times\mathrm{Sp}(2n-4r)$.

\subsection{$\mathfrak{sl}(2n+1)/\mathfrak{so}(2n+1)$}

$K=\mathrm{SO}(2n+1)$, $n\geq2$, $\mathfrak p=V(2\omega_1)$.

Let us fix a basis $e_1,\ldots,e_n,e_0,e_{-n},\ldots,e_{-1}$ of $\mathbb C^{2n+1}$, a symmetric bilinear form $\beta$ such that $\beta(e_i,e_j)=\delta_{i,-j}$ for all $i,j$ and $K=\mathrm{SO}(\mathbb C^{2n+1},\beta)$. Then $\beta$ can be seen as a linear form on $\mathsf S^2\mathbb C^{2n+1}$, $\beta(e_i e_j)=\beta(e_i,e_j)$, and 
\[\mathfrak p=\ker\beta\subset\mathsf S^2\mathbb C^{2n+1}.\]

\subsubsection{$\mathbf{(2^{r},1^{2n-2r+1})}$, $r\geq1$}\label{sss:symm2}

\[e=\sum_{i=1}^re_i e_{r-i+1},\quad
h(e_i)=\left\{\begin{array}{cl}
e_i & \mbox{if $1\leq i\leq r$}\\
-e_i & \mbox{if $-r\leq i\leq -1$}\\
0 & \mbox{otherwise}
\end{array}\right.,\quad
f=\sum_{i=1}^re_{-r+i-1} e_{-i}.\] 
Let $Q=L\,Q^\mathrm u$ be the corresponding parabolic subgroup of $K$,
so that $L=K_h\cong\mathrm{GL}(r)\times\mathrm{SO}(2n-2r+1)$.

The centralizer of $e$ is $K_e=L_eQ^\mathrm u$ where
$L_e\cong\mathrm{O}(r)\times\mathrm{SO}(2n-2r+1)$.

\subsection{$\mathfrak{sl}(2n)/\mathfrak{so}(2n)$}

$K=\mathrm{SO}(2n)$, $n\geq3$, $\mathfrak p=V(2\omega_1)$. If $n=2$, $\mathfrak p=V(2\omega+2\omega')$.

Let us fix a basis $e_1,\ldots,e_n,e_{-n},\ldots,e_{-1}$ of $\mathbb C^{2n}$, a symmetric bilinear form $\beta$ such that $\beta(e_i,e_j)=\delta_{i,-j}$ for all $i,j$ and $K=\mathrm{SO}(\mathbb C^{2n},\beta)$. Then $\beta$ can be seen as a linear form on $\mathsf S^2\mathbb C^{2n}$, $\beta(e_i e_j)=\beta(e_i,e_j)$, and 
\[\mathfrak p=\ker\beta\subset\mathsf S^2\mathbb C^{2n}.\] 

Let us denote by $\tau$ the linear endomorphism of $\mathbb C^{2n}$ switching $e_n$ and $e_{-n}$ and fixing all the other basis vectors. The conjugation by $\tau$ is an involutive external automorphism of $\mathfrak g$, living $\mathfrak k$ and $\mathfrak p$ stable, and inducing the nontrivial involution of the Dynkin diagram of $\mathfrak k$.

\subsubsection{$\mathbf{(2^{r},1^{2n-2r})}$, $r \geq 1$}\label{sss:Dao}

If $r< n$,
\[e=\sum_{i=1}^re_i e_{r-i+1},\quad
h(e_i)=\left\{\begin{array}{cl}
e_i & \mbox{if $1\leq i\leq r$}\\
-e_i & \mbox{if $-r\leq i\leq -1$}\\
0 & \mbox{otherwise}
\end{array}\right.,\quad
f=\sum_{i=1}^re_{-r+i-1} e_{-i}.\] 
Let $Q=L\,Q^\mathrm u$ be the corresponding parabolic subgroup of $K$,
so that $L=K_h\cong\mathrm{GL}(r)\times\mathrm{SO}(2n-2r)$.

The centralizer of $e$ is $K_e=L_eQ^\mathrm u$ where
$L_e\cong\mathrm{O}(r)\times\mathrm{SO}(2n-2r)$.

If $r=n$, there exist two orbits labelled I and II. 
The case (I) can be described as above by specializing $r$ equal to $n$.
The case (II) can be obtained from the case (I) conjugating by $\tau$.

\subsection{$\mathfrak{sp}(2p+2q)/\mathfrak{sp}(2p) + \mathfrak{sp}(2q)$}\label{sse:sp}

$K=\mathrm{Sp}(2p)\times\mathrm{Sp}(2q)$, $p,q\geq1$, $\mathfrak p=V(\omega_1+\omega_1')$.

Let us fix a basis $e_1,\ldots,e_p,e_{-p},\ldots,e_{-1}$ of $\mathbb C^{2p}$ and a skew-symmetric bilinear form $\omega$ such that $\omega(e_i,e_j)=\delta_{i,-j}$ for $1\leq i\leq p$. Similarly, let us fix a basis $e'_1,\ldots,e'_q,e'_{-q},\ldots,e'_{-1}$ of $\mathbb C^{2q}$ and a skew-symmetric bilinear form $\omega'$ such that $\omega'(e'_i,e'_j)=\delta_{i,-j}$ for $1\leq i\leq q$. Then $K=\mathrm{Sp}(\mathbb C^{2p},\omega)\times\mathrm{Sp}(\mathbb C^{2q},\omega')$ and 
\[\mathfrak p=\mathbb C^{2p}\otimes\mathbb C^{2q}.\]

\subsubsection{$\mathbf{(+2^{2r},+1^{2p-2r},-1^{2q-2r})}$, $r\geq 1$}\label{sss:symm4.1}

\[e=\sum_{i=1}^re_i\otimes e'_{r-i+1},\qquad
f=-\sum_{i=1}^re_{-r+i-1}\otimes e'_{-i},\] 
\[h(e_i)=\left\{\begin{array}{cl}
e_i & \mbox{if $1\leq i\leq r$}\\
-e_i & \mbox{if $-r\leq i\leq -1$}\\
0 & \mbox{otherwise}
\end{array}\right.,\quad
h(e'_i)=\left\{\begin{array}{cl}
e'_i & \mbox{if $1\leq i\leq r$}\\
-e'_i & \mbox{if $-r\leq i\leq -1$}\\
0 & \mbox{otherwise}
\end{array}\right..\]
Let $Q=L\,Q^\mathrm u$ be the corresponding parabolic subgroup of $K$,
so that $L=K_h\cong\mathrm{GL}(r)\times\mathrm{Sp}(2p-2r)\times\mathrm{GL}(r)\times\mathrm{Sp}(2q-2r)$.

The centralizer of $e$ is $K_e=L_eQ^\mathrm u$ where
$L_e\cong\mathrm{GL}(r)\times\mathrm{Sp}(2p-2r)\times\mathrm{Sp}(2q-2r)$,
the $\mathrm{GL}(r)$ factor of $L_e$ is embedded skew-diagonally, $A\mapsto(A,A^{-1})$, into the $\mathrm{GL}(r)\times\mathrm{GL}(r)$ factor of $L$.

\subsubsection{$\mathbf{(+3^{2},+1^{2p-4},-1^{2q-2})}$}\label{sss:CCaac}

\[e=e_1\otimes e'_{-1}-e_2\otimes e'_{1},\qquad
f=2e_{-2}\otimes e'_{-1}+2e_{-1}\otimes e'_{1},\] 
\[h(e_i)=\left\{\begin{array}{cl}
2e_i & \mbox{if $1\leq i\leq 2$}\\
-2e_i & \mbox{if $-2\leq i\leq -1$}\\
0 & \mbox{otherwise}
\end{array}\right.,\quad
h(e'_i)=0\ \forall\,i.\]
Let $Q=L\,Q^\mathrm u$ be the corresponding parabolic subgroup of $K$,
so that $L=K_h\cong\mathrm{GL}(2)\times\mathrm{Sp}(2p-4)\times\mathrm{Sp}(2q)$.

The centralizer of $e$ is $K_e=L_eQ^\mathrm u$ where
$L_e\cong\mathrm{SL}(2)\times\mathrm{Sp}(2p-4)\times\mathrm{Sp}(2q-2)$,
the $\mathrm{SL}(2)\times\mathrm{Sp}(2q-2)$ factor of $L_e$ is embedded as 
\[(A,B)\mapsto(A,A,B)\] into $\mathrm{SL}(2)\times\mathrm{Sp}(2)\times\mathrm{Sp(2q-2)}$ 
where the $\mathrm{SL}(2)$ factor is included in the $\mathrm{GL}(2)$ factor of $L$
and the $\mathrm{Sp}(2)\times\mathrm{Sp(2q-2)}$ factor is included in the $\mathrm{Sp}(2q)$ factor of $L$.

\subsubsection{$\mathbf{(-3^{2},+1^{2p-2},-1^{2q-4})}$}\label{sss:CCaac2}

This case can be obtained from the case \ref{sss:CCaac} by switching the role of $p$ and $q$.

\subsubsection{$\mathbf{(+3^{2},+2^{2},+1^{2p-6},-1^{2q-4})}$}\label{sss:CCacy}

\[e=e_1\otimes e'_{-2}-e_2\otimes e'_{2}-e_3\otimes e'_{1},\quad
f=e_{-3}\otimes e'_{-1}+2e_{-2}\otimes e'_{-2}+2e_{-1}\otimes e'_{2},\] 
\[h(e_i)=\left\{\begin{array}{cl}
2e_i & \mbox{if $1\leq i\leq 2$}\\
e_i & \mbox{if $i=3$}\\
-e_i & \mbox{if $i=-3$}\\
-2e_i & \mbox{if $-2\leq i\leq -1$}\\
0 & \mbox{otherwise}
\end{array}\right.,\quad
h(e'_i)=\left\{\begin{array}{cl}
e'_i & \mbox{if $i=1$}\\
-e'_i & \mbox{if $i=-1$}\\
0 & \mbox{otherwise}
\end{array}\right..\]
Let $Q=L\,Q^\mathrm u$ be the corresponding parabolic subgroup of $K$,
so that $L=K_h\cong\mathrm{GL}(2)\times\mathrm{GL}(1)\times\mathrm{Sp}(2p-6)\times\mathrm{GL}(1)\times\mathrm{Sp}(2q-2)$.

The centralizer of $e$ is $K_e=L_eK_e^\mathrm u$ where
$L_e\cong\mathrm{SL}(2)\times\mathrm{GL}(1)\times\mathrm{Sp}(2p-3)\times\mathrm{Sp}(2q-4)$,
the $\mathrm{SL}(2)\times\mathrm{Sp}(2q-4)$ factor of $L_e$ is embedded as 
\[(A,B)\mapsto(A,A,B)\] into $\mathrm{SL}(2)\times\mathrm{Sp}(2)\times\mathrm{Sp(2q-4)}$ 
where the $\mathrm{SL}(2)$ factor is included in the $\mathrm{GL}(2)$ factor of $L$
and the $\mathrm{Sp}(2)\times\mathrm{Sp(2q-4)}$ factor is included in the $\mathrm{Sp}(2q-2)$ factor of L,
the $\mathrm{GL}(1)$ factor of $L_e$ is embedded skew-diagonally
\[z\mapsto(z,z^{-1})\]
into the $\mathrm{GL}(1)\times\mathrm{GL}(1)$ factor of $L$.
The quotient $\mathrm{Lie}\,Q^\mathrm u/\mathrm{Lie}\,K_e^\mathrm u$ is a simple $L_e$-module of dimension 2 as follows. In $\mathfrak k(1)$ there are exactly two simple $L_e$-submodules, $W_0,W_1$, of highest weight $\omega_1$ w.r.t.\ the $\mathrm{SL}(2)$ factor, isomorphic as $L_e$-modules but lying in two distinct isotypical $L$-components. Let $V$ be the $L_e$-complement of $W_0\oplus W_1$ in $\mathrm{Lie}\,Q^\mathrm u$. As $L_e$-module, $\mathrm{Lie}\,K_e^\mathrm u$ is the direct sum of $V$ and a simple $L_e$-submodule of $W_0\oplus W_1$ which projects non-trivially on both summands $W_0$ and $W_1$.

\subsubsection{$\mathbf{(-3^{2},+2^{2},+1^{2p-4},-1^{2q-6})}$}\label{sss:CCcay}

This case can be obtained from the case \ref{sss:CCacy} by switching the role of $p$ and $q$.

\subsubsection{$\mathbf{(+3^4,+1^{2p-8})}$, $q=2$}\label{sss:CCabx}

\[e=e_1\otimes e'_{-1}+e_2\otimes e'_{-2}-e_3\otimes e'_2-e_4\otimes e'_1,\]
\[h(e_i)=\left\{\begin{array}{cl}
2e_i & \mbox{if $1\leq i\leq 4$}\\
-2e_i & \mbox{if $-4\leq i\leq -1$}\\
0 & \mbox{otherwise}
\end{array}\right.,\quad
h(e'_i)=0\ \forall\,i,\]
\[f=2(e_{-4}\otimes e'_{-1}+e_{-3}\otimes e'_{-2}+e_{-2}\otimes e'_{2}+e_{-1}\otimes e'_{1}).\] 
Let $Q=L\,Q^\mathrm u$ be the corresponding parabolic subgroup of $K$,
so that $L=K_h\cong\mathrm{GL}(4)\times\mathrm{Sp}(2p-8)\times\mathrm{Sp}(4)$.

The centralizer of $e$ is $K_e=L_eQ^\mathrm u$ where
$L_e\cong\mathrm{Sp}(4)\times\mathrm{Sp}(2p-8)$,
the $\mathrm{Sp}(4)$ factor of $L_e$ is embedded diagonally,  
$A\mapsto(A,A)$, into the $\mathrm{GL}(4)\times\mathrm{Sp}(4)$ 
factor of $L$.

\subsubsection{$\mathbf{(-3^4,-1^{2q-8})}$, $p=2$}\label{sss:CCbax}

This case can be obtained from the case \ref{sss:CCabx} by switching the role of $p$ and $q$.

\subsection{$\mathfrak{so}(2n+1)/\mathfrak{so}(2n)$}

$K=\mathrm{SO}(2n)$, $n\geq3$, $\mathfrak p=V(\omega_1)$. If $n=2$, $\mathfrak p=V(\omega+\omega')$.

Let us fix a basis $e_1,\ldots,e_n,e_{-n},\ldots,e_{-1}$ of $\mathbb C^{2n}$, a symmetric bilinear form $\beta$ such that $\beta(e_i,e_j)=\delta_{i,-j}$ for all $i,j$ and $K=\mathrm{SO}(\mathbb C^{2n},\beta)$. Then  
\[\mathfrak p=\mathbb C^{2n}.\]

\subsubsection{$\mathbf{(+3,+1^{2n-2})}$}\label{sss:trvial1}

\[e=e_1,\quad
h(e_i)=\left\{\begin{array}{cl}
2e_i & \mbox{if $i=1$}\\
-2e_i & \mbox{if $i=-1$}\\
0 & \mbox{otherwise}
\end{array}\right.,\quad
f=-2e_{-1}.\] 
Let $Q=L\,Q^\mathrm u$ be the corresponding parabolic subgroup of $K$,
so that $L=K_h\cong\mathrm{GL}(1)\times\mathrm{SO}(2n-2)$.

The centralizer of $e$ is $K_e=L_eQ^\mathrm u$ where
$L_e\cong\mathrm{SO}(2n-2)$.

\subsection{$\mathfrak{so}(2n+2)/\mathfrak{so}(2n+1)$}

$K=\mathrm{SO}(2n+1)$, $n\geq2$, $\mathfrak p=V(\omega_1)$. If $n=1$, $\mathfrak p=V(2\omega)$.

Let us fix a basis $e_1,\ldots,e_n,e_0,e_{-n},\ldots,e_{-1}$ of $\mathbb C^{2n+1}$, a symmetric bilinear form $\beta$ such that $\beta(e_i,e_j)=\delta_{i,-j}$ for all $i,j$ and $K=\mathrm{SO}(\mathbb C^{2n+1},\beta)$. Then  
\[\mathfrak p=\mathbb C^{2n+1}.\]

\subsubsection{$\mathbf{(+3,+1^{2n-1})}$}\label{sss:trvial2}

\[e=e_1,\quad
h(e_i)=\left\{\begin{array}{cl}
2e_i & \mbox{if $i=1$}\\
-2e_i & \mbox{if $i=-1$}\\
0 & \mbox{otherwise}
\end{array}\right.,\quad
f=-2e_{-1}.\] 
Let $Q=L\,Q^\mathrm u$ be the corresponding parabolic subgroup of $K$,
so that $L=K_h\cong\mathrm{GL}(1)\times\mathrm{SO}(2n-1)$.

The centralizer of $e$ is $K_e=L_eQ^\mathrm u$ where
$L_e\cong\mathrm{SO}(2n-1)$.

\subsection{$\mathfrak{so}(2p+2q+1)/\mathfrak{so}(2p+1) + \mathfrak{so}(2q)$}

$K=\mathrm{SO}(2p+1)\times\mathrm{SO}(2q)$, $p\geq2$, $q\geq3$, $\mathfrak p=V(\omega_1+\omega'_1)$. 
If $p=1$ and $q\geq3$, $\mathfrak p=V(2\omega+\omega'_1)$.
If $p\geq2$ and $q=2$, $\mathfrak p=V(\omega_1+\omega'+\omega'')$.
If $p=1$ and $q=2$, $\mathfrak p=V(2\omega+\omega'+\omega'')$.

Let us fix a basis $e_1,\ldots,e_p,e_0,e_{-p},\ldots,e_{-1}$ of $\mathbb C^{2p+1}$ and a symmetric bilinear form $\beta$ such that $\beta(e_i,e_j)=\delta_{i,-j}$ for all $i,j$. Similarly, let us fix a basis $e'_1,\ldots,e'_q,e'_{-q},\ldots,e'_{-1}$ of $\mathbb C^{2q}$ and a symmetric bilinear form $\beta'$ such that $\beta'(e'_i,e'_j)=\delta_{i,-j}$ for all $i,j$. Then $K=\mathrm{SO}(\mathbb C^{2p+1},\beta)\times\mathrm{SO}(\mathbb C^{2q},\beta')$ and 
\[\mathfrak p=\mathbb C^{2p+1}\otimes\mathbb C^{2q}.\] 

Let us denote by $\tau$ the linear endomorphism of $\mathbb C^{2p+2q+1}$ switching $e'_q$ and $e'_{-q}$ and fixing all the other basis vectors. The conjugation by $\tau$ is an involutive external automorphism of $\mathfrak g$, living $\mathfrak k$ and $\mathfrak p$ stable, and inducing the nontrivial involution of the Dynkin diagram of $\mathfrak k$.

\subsubsection{$\mathbf{(+2^{2r},+1^{2p+1-2r},-1^{2q-2r})}$, $r \geq 1$}\label{sss:BDaa}

If $r<q$,
\[e=\sum_{i=1}^re_i\otimes e'_{r-i+1},\qquad
f=-\sum_{i=1}^re_{-r+i-1}\otimes e'_{-i},\] 
\[h(e_i)=\left\{\begin{array}{cl}
e_i & \mbox{if $1\leq i\leq r$}\\
-e_i & \mbox{if $-r\leq i\leq -1$}\\
0 & \mbox{otherwise}
\end{array}\right.,\quad
h(e'_i)=\left\{\begin{array}{cl}
e'_i & \mbox{if $1\leq i\leq r$}\\
-e'_i & \mbox{if $-r\leq i\leq -1$}\\
0 & \mbox{otherwise}
\end{array}\right..\]
Let $Q=L\,Q^\mathrm u$ be the corresponding parabolic subgroup of $K$,
so that $L=K_h\cong\mathrm{GL}(r)\times\mathrm{SO}(2p-2r+1)\times\mathrm{GL}(r)\times\mathrm{SO}(2q-2r)$.

The centralizer of $e$ is $K_e=L_eQ^\mathrm u$ where
$L_e\cong\mathrm{GL}(r)\times\mathrm{SO}(2p-2r+1)\times\mathrm{SO}(2q-2r)$,
the $\mathrm{GL}(r)$ factor of $L_e$ is embedded skew-diagonally, $A\mapsto(A,A^{-1})$, into the $\mathrm{GL}(r)\times\mathrm{GL}(r)$ factor of $L$.

If $r=q$, there exist two orbits labelled I and II. 
The case (I) can be described as above by specializing $r$ equal to $q$.
The case (II) can be obtained from the case (I) conjugating by $\tau$.

\subsubsection{$\mathbf{(+3,+2^{2r},+1^{2p-1-2r},-1^{2q-1-2r})}$}\label{sss:BDady}

If $r\leq q-2$,
\[e=e_1\otimes (e'_{q}+e'_{-q})+\sum_{i=1}^re_{i+1}\otimes e'_{r-i+1},\quad
f=-\big(\sum_{i=1}^re_{-r+i-2}\otimes e'_{-i}\big)-e_{-1}\otimes(e'_q+e'_{-q}),\] 
\[h(e_i)=\left\{\begin{array}{cl}
2e_i & \mbox{if $i=1$}\\
e_i & \mbox{if $2\leq i\leq r+1$}\\
-e_i & \mbox{if $-r-1\leq i\leq -2$}\\
-2e_i & \mbox{if $i=-1$}\\
0 & \mbox{otherwise}
\end{array}\right.,\quad
h(e'_i)=\left\{\begin{array}{cl}
e'_i & \mbox{if $1\leq i\leq r$}\\
-e'_i & \mbox{if $-r\leq i\leq -1$}\\
0 & \mbox{otherwise}
\end{array}\right..\]
Let $Q=L\,Q^\mathrm u$ be the corresponding parabolic subgroup of $K$,
so that $L=K_h\cong\mathrm{GL}(1)\times\mathrm{GL}(r)\times\mathrm{SO}(2p-2r-1)
\times\mathrm{GL}(r)\times\mathrm{SO}(2q-2r)$.

The centralizer of $e$ is $K_e=L_eK_e^\mathrm u$ where
$L_e\cong\mathrm{GL}(r)\times\mathrm{SO}(2p-2r-1)\times\mathrm S(\mathrm{O}(1)\times\mathrm{O}(2q-2r-1))$,
the $\mathrm S(\mathrm{O}(1)\times\mathrm{O}(2p-2r-1))$ factor of $L_e$ is embedded as 
\[(z,A)\mapsto(z,z,A)\] into $\mathrm{GL}(1)\times\mathrm S(\mathrm{O}(1)\times\mathrm{O}(2q-2r-1))$ 
where the $\mathrm S(\mathrm{O}(1)\times\mathrm{O}(2q-2r-1))$ factor 
is included in the $\mathrm{SO}(2q-2r)$ factor of L,
the $\mathrm{GL}(r)$ factor of $L_e$ is embedded skew-diagonally
\[B\mapsto(B,B^{-1})\]
into the $\mathrm{GL}(r)\times\mathrm{GL}(r)$ factor of $L$.
The quotient $\mathrm{Lie}\,Q^\mathrm u/\mathrm{Lie}\,K_e^\mathrm u$ is a simple $L_e$-module of dimension $r$ as follows. In $\mathfrak k(1)$ there are exactly two simple $L_e$-submodules, $W_0,W_1$, of highest weight $\omega_{r-1}$ w.r.t.\ the $\mathrm{GL}(r)$ factor, isomorphic as $L_e$-modules but lying in two distinct isotypical $L$-components. Let $V$ be the $L_e$-complement of $W_0\oplus W_1$ in $\mathrm{Lie}\,Q^\mathrm u$. As $L_e$-module, $\mathrm{Lie}\,K_e^\mathrm u$ is the direct sum of $V$ and a simple $L_e$-submodule of $W_0\oplus W_1$ which projects non-trivially on both summands $W_0$ and $W_1$.

If $r=q-1$,
the normal triple $h,e,f$, the parabolic subgroup $Q=L\,Q^\mathrm u$ 
and $L_e$ have the same description, with $K_e=L_eK_e^\mathrm u$.  
The quotient $\mathrm{Lie}\,Q^\mathrm u/\mathrm{Lie}\,K_e^\mathrm u$ remains a simple $L_e$-module of dimension $q-1$ but here in $\mathfrak k(1)$ there are exactly three simple $L_e$-submodules, $W_0,W_1,W_2$, of highest weight $\omega_{q-2}$ w.r.t.\ the $\mathrm{GL}(q-1)$ factor, isomorphic as $L_e$-modules but lying in three distinct isotypical $L$-components. Let $V$ be the $L_e$-complement of $W_0\oplus W_1\oplus W_2$ in $\mathrm{Lie}\,Q^\mathrm u$. As $L_e$-module, $\mathrm{Lie}\,K_e^\mathrm u$ is the direct sum of $V$ and a co-simple $L_e$-submodule of $W_0\oplus W_1\oplus W_2$ which projects non-trivially on every summand $W_0$, $W_1$ and $W_2$.

\subsubsection{$\mathbf{(-3,+2^{2r},+1^{2p-2r},-1^{2q-2-2r})}$}\label{sss:BDbay}
If $r\leq q-2$,
\[e=\big(\sum_{i=1}^re_i\otimes e'_{r-i+2}\big)+e_0\otimes e'_{1},\quad
f=-2e_0\otimes e'_{-1}-\sum_{i=1}^re_{-r+i-1}\otimes e'_{-i-1},\] 
\[h(e_i)=\left\{\begin{array}{cl}
e_i & \mbox{if $1\leq i\leq r$}\\
-e_i & \mbox{if $-r\leq i\leq -1$}\\
0 & \mbox{otherwise}
\end{array}\right.,\quad
h(e'_i)=\left\{\begin{array}{cl}
2e'_i & \mbox{if $i=1$}\\
e'_i & \mbox{if $2\leq i\leq r+1$}\\
-e'_i & \mbox{if $-r-1\leq i\leq -2$}\\
-2e'_i & \mbox{if $i=-1$}\\
0 & \mbox{otherwise}
\end{array}\right..\]
Let $Q=L\,Q^\mathrm u$ be the corresponding parabolic subgroup of $K$,
so that $L=K_h\cong\mathrm{GL}(r)\times\mathrm{SO}(2p-2r+1)\times\mathrm{GL}(1)\times\mathrm{GL}(r)\times\mathrm{SO}(2q-2r-2)$.

The centralizer of $e$ is $K_e=L_eK_e^\mathrm u$ where
$L_e\cong\mathrm{GL}(r)\times\mathrm S(\mathrm{O}(1)\times\mathrm{O}(2p-2r))\times\mathrm{SO}(2q-2r-2)$,
the $\mathrm S(\mathrm{O}(1)\times\mathrm{O}(2p-2r))$ factor of $L_e$ is embedded as 
\[(z,A)\mapsto(z,A,z)\] into $\mathrm S(\mathrm{O}(1)\times\mathrm{O}(2p-2r))\times\mathrm{GL}(1)$ 
where the $\mathrm S(\mathrm{O}(1)\times\mathrm{O}(2p-2r))$ factor is included in the $\mathrm{SO}(2p-2r+1)$ factor of L,
the $\mathrm{GL}(r)$ factor of $L_e$ is embedded skew-diagonally
\[B\mapsto(B,B^{-1})\]
into the $\mathrm{GL}(r)\times\mathrm{GL}(r)$ factor of $L$.
The quotient $\mathrm{Lie}\,Q^\mathrm u/\mathrm{Lie}\,K_e^\mathrm u$ is a simple $L_e$-module of dimension $r$ as follows. In $\mathfrak k(1)$ there are exactly two simple $L_e$-submodules, $W_0,W_1$, of highest weight $\omega_1$ w.r.t.\ the $\mathrm{GL}(r)$ factor, isomorphic as $L_e$-modules but lying in two distinct isotypical $L$-components. Let $V$ be the $L_e$-complement of $W_0\oplus W_1$ in $\mathrm{Lie}\,Q^\mathrm u$. As $L_e$-module, $\mathrm{Lie}\,K_e^\mathrm u$ is the direct sum of $V$ and a simple $L_e$-submodule of $W_0\oplus W_1$ which projects non-trivially on both summands $W_0$ and $W_1$.

If $r=q-1$, there exist two orbits labelled I and II. 
The case (I) can be described as above by specializing $r$ equal to $q-1$.
The case (II) can be obtained from the case (I) conjugating by $\tau$.

\subsection{$\mathfrak{so}(2p+2q+2)/\mathfrak{so}(2p+1) + \mathfrak{so}(2q+1)$}

$K=\mathrm{SO}(2p+1)\times\mathrm{SO}(2q+1)$, $p,q\geq2$, $\mathfrak p=V(\omega_1+\omega'_1)$. 
If $p=1$ and $q\geq2$, $\mathfrak p=V(2\omega+\omega'_1)$.  
If $p\geq2$ and $q=1$, $\mathfrak p=V(\omega_1+2\omega')$.
If $p=q=1$, $\mathfrak p=V(2\omega+2\omega')$.  

Let us fix a basis $e_1,\ldots,e_p,e_0,e_{-p},\ldots,e_{-1}$ of $\mathbb C^{2p+1}$ and a symmetric bilinear form $\beta$ such that $\beta(e_i,e_j)=\delta_{i,-j}$ for all $i,j$. Similarly, let us fix a basis $e'_1,\ldots,e'_q,e'_0$, $e'_{-q},\ldots,e'_{-1}$ of $\mathbb C^{2q+1}$ and a symmetric bilinear form $\beta'$ such that $\beta'(e'_i,e'_j)=\delta_{i,-j}$ for all $i,j$. Then $K=\mathrm{SO}(\mathbb C^{2p+1},\beta)\times\mathrm{SO}(\mathbb C^{2q+1},\beta')$ and 
\[\mathfrak p=\mathbb C^{2p+1}\otimes\mathbb C^{2q+1}.\]

\subsubsection{$\mathbf{(+2^{2r},+1^{2p+1-2r},-1^{2q+1-2r})}$, $r\geq 1$}\label{sss:symm8}

\[e=\sum_{i=1}^re_i\otimes e'_{r-i+1},\qquad
f=-\sum_{i=1}^re_{-r+i-1}\otimes e'_{-i},\] 
\[h(e_i)=\left\{\begin{array}{cl}
e_i & \mbox{if $1\leq i\leq r$}\\
-e_i & \mbox{if $-r\leq i\leq -1$}\\
0 & \mbox{otherwise}
\end{array}\right.,\quad
h(e'_i)=\left\{\begin{array}{cl}
e'_i & \mbox{if $1\leq i\leq r$}\\
-e'_i & \mbox{if $-r\leq i\leq -1$}\\
0 & \mbox{otherwise}
\end{array}\right..\]
Let $Q=L\,Q^\mathrm u$ be the corresponding parabolic subgroup of $K$,
so that $L=K_h\cong\mathrm{GL}(r)\times\mathrm{SO}(2p-2r+1)\times\mathrm{GL}(r)\times\mathrm{SO}(2q-2r+1)$.

The centralizer of $e$ is $K_e=L_eQ^\mathrm u$ where
$L_e\cong\mathrm{GL}(r)\times\mathrm{SO}(2p-2r+1)\times\mathrm{SO}(2q-2r+1)$,
the $\mathrm{GL}(r)$ factor of $L_e$ is embedded skew-diagonally, $A\mapsto(A,A^{-1})$, into the $\mathrm{GL}(r)\times\mathrm{GL}(r)$ factor of $L$.

\subsubsection{$\mathbf{(+3,+2^{2r},+1^{2p-1-2r},-1^{2q-2r})}$}\label{sss:BBaby}

\[e=e_1\otimes e'_0+\sum_{i=1}^re_{i+1}\otimes e'_{r-i+1},\quad
f=-\big(\sum_{i=1}^re_{-r+i-2}\otimes e'_{-i}\big)-2e_{-1}\otimes e'_0,\] 
\[h(e_i)=\left\{\begin{array}{cl}
2e_i & \mbox{if $i=1$}\\
e_i & \mbox{if $2\leq i\leq r+1$}\\
-e_i & \mbox{if $-r-1\leq i\leq -2$}\\
-2e_i & \mbox{if $i=-1$}\\
0 & \mbox{otherwise}
\end{array}\right.,\quad
h(e'_i)=\left\{\begin{array}{cl}
e'_i & \mbox{if $1\leq i\leq r$}\\
-e'_i & \mbox{if $-r\leq i\leq -1$}\\
0 & \mbox{otherwise}
\end{array}\right..\]
Let $Q=L\,Q^\mathrm u$ be the corresponding parabolic subgroup of $K$,
so that $L=K_h\cong\mathrm{GL}(1)\times\mathrm{GL}(r)\times\mathrm{SO}(2p-2r-1)\times\mathrm{GL}(r)\times\mathrm{SO}(2q-2r+1)$.

The centralizer of $e$ is $K_e=L_eK_e^\mathrm u$ where
$L_e\cong\mathrm{GL}(r)\times\mathrm{SO}(2p-2r-1)\times\mathrm S(\mathrm{O}(1)\times\mathrm{O}(2q-2r))$,
the $\mathrm S(\mathrm{O}(1)\times\mathrm{O}(2q-2r))$ factor of $L_e$ is embedded as 
\[(z,A)\mapsto(z,z,A)\] into $\mathrm{GL}(1)\times\mathrm S(\mathrm{O}(1)\times\mathrm{O}(2q-2r))$ 
where the $\mathrm S(\mathrm{O}(1)\times\mathrm{O}(2q-2r))$ factor is included in the $\mathrm{SO}(2q-2r+1)$ factor of L,
the $\mathrm{GL}(r)$ factor of $L_e$ is embedded skew-diagonally
\[B\mapsto(B,B^{-1})\]
into the $\mathrm{GL}(r)\times\mathrm{GL}(r)$ factor of $L$.
The quotient $\mathrm{Lie}\,Q^\mathrm u/\mathrm{Lie}\,K_e^\mathrm u$ is a simple $L_e$-module of dimension $r$ as follows. In $\mathfrak k(1)$ there are exactly two simple $L_e$-submodules, $W_0,W_1$, of highest weight $\omega_{r-1}$ w.r.t.\ the $\mathrm{GL}(r)$ factor, isomorphic as $L_e$-modules but lying in two distinct isotypical $L$-components. Let $V$ be the $L_e$-complement of $W_0\oplus W_1$ in $\mathrm{Lie}\,Q^\mathrm u$. As $L_e$-module, $\mathrm{Lie}\,K_e^\mathrm u$ is the direct sum of $V$ and a simple $L_e$-submodule of $W_0\oplus W_1$ which projects non-trivially on both summands $W_0$ and $W_1$.

\subsubsection{$\mathbf{(-3,+2^{2r},+1^{2p-2r},-1^{2q-1-2r})}$}\label{sss:BBbay}

This case can be obtained from the case \ref{sss:BBaby} by switching the role of $p$ and $q$.

\subsection{$\mathfrak{so}(2p+2q)/\mathfrak{so}(2p) + \mathfrak{so}(2q)$}

$K=\mathrm{SO}(2p)\times\mathrm{SO}(2q)$, $p,q\geq3$, $\mathfrak p=V(\omega_1+\omega'_1)$. 
If $p=2$ and $q\geq3$, $\mathfrak p=V(\omega+\omega'+\omega''_1)$.
If $p\geq3$ and $q=2$, $\mathfrak p=V(\omega_1+\omega'+\omega'')$.
If $p=2$ and $q=2$, $\mathfrak p=V(\omega+\omega'+\omega''+\omega''')$.

Let us fix a basis $e_1,\ldots,e_p,e_{-p},\ldots,e_{-1}$ of $\mathbb C^{2p}$ and a symmetric bilinear form $\beta$ such that $\beta(e_i,e_j)=\delta_{i,-j}$ for all $i,j$. Similarly, let us fix a basis $e'_1,\ldots,e'_q,e'_{-q},\ldots,e'_{-1}$ of $\mathbb C^{2q}$ and a symmetric bilinear form $\beta'$ such that $\beta'(e'_i,e'_j)=\delta_{i,-j}$ for all $i,j$. Then $K=\mathrm{SO}(\mathbb C^{2p},\beta)\times\mathrm{SO}(\mathbb C^{2q},\beta')$ and 
\[\mathfrak p=\mathbb C^{2p}\otimes\mathbb C^{2q}.\] 

Let us denote by $\tau$ the linear endomorphism of $\mathbb C^{2p+2q}$ switching $e_p$ and $e_{-p}$ and fixing all the other basis vectors. Similarly, let  us denote by $\tau'$ the linear endomorphism of $\mathbb C^{2p+2q}$ switching $e'_q$ and $e'_{-q}$ and fixing all the other basis vectors. The conjugation by $\tau$ (and by $\tau'$, respectively) is an involutive external automorphism of $\mathfrak g$, living $\mathfrak k$ and $\mathfrak p$ stable, and inducing the nontrivial involution of the first (the second, respectively) connected component of the Dynkin diagram of $\mathfrak k$.

\subsubsection{$\mathbf{(+2^{2r},+1^{2p-2r},-1^{2q-2r})}$, $r \geq 1$}\label{sss:DDaa}

If $r<p$ and $r<q$,
\[e=\sum_{i=1}^re_i\otimes e'_{r-i+1},\qquad
f=-\sum_{i=1}^re_{-r+i-1}\otimes e'_{-i},\] 
\[h(e_i)=\left\{\begin{array}{cl}
e_i & \mbox{if $1\leq i\leq r$}\\
-e_i & \mbox{if $-r\leq i\leq -1$}\\
0 & \mbox{otherwise}
\end{array}\right.,\quad
h(e'_i)=\left\{\begin{array}{cl}
e'_i & \mbox{if $1\leq i\leq r$}\\
-e'_i & \mbox{if $-r\leq i\leq -1$}\\
0 & \mbox{otherwise}
\end{array}\right..\]
Let $Q=L\,Q^\mathrm u$ be the corresponding parabolic subgroup of $K$,
so that $L=K_h\cong\mathrm{GL}(r)\times\mathrm{SO}(2p-2r)\times\mathrm{GL}(r)\times\mathrm{SO}(2q-2r)$.

The centralizer of $e$ is $K_e=L_eQ^\mathrm u$ where
$L_e\cong\mathrm{GL}(r)\times\mathrm{SO}(2p-2r)\times\mathrm{SO}(2q-2r)$,
the $\mathrm{GL}(r)$ factor of $L_e$ is embedded skew-diagonally, $A\mapsto(A,A^{-1})$, into the $\mathrm{GL}(r)\times\mathrm{GL}(r)$ factor of $L$.

If $r=p$ and $r<q$, there exist two orbits labelled I and II. 
The case (I) can be described as above by specializing $r$ equal to $p$.
The case (II) can be obtained from the case (I) conjugating by $\tau$.

If $r<p$ and $r=q$, there exist two orbits labelled I and II. 
The case (I) can be described as above by specializing $r$ equal to $q$.
The case (II) can be obtained from the case (I) conjugating by $\tau'$.

If $r=p=q$, there exist four orbits with a double label I or II. 
The case (I,I) can be described as above by specializing $r$ equal to $p=q$.
The case (I,II) can be obtained from the case (I,I) conjugating by $\tau'$.
The case (II,I) can be obtained from the case (I,I) conjugating by $\tau$.
The case (II,II) can be obtained from the case (I,I) conjugating by $\tau$ and $\tau'$.

\subsubsection{$\mathbf{(+3,+2^{2r},+1^{2p-2-2r},-1^{2q-1-2r})}$}\label{sss:DDady}

If $r\leq p-2$ and $r\leq q-2$,
\[e=e_1\otimes (e'_{q}+e'_{-q})+\sum_{i=1}^re_{i+1}\otimes e'_{r-i+1},\quad
f=-\big(\sum_{i=1}^re_{-r+i-2}\otimes e'_{-i}\big)-e_{-1}\otimes(e'_q+e'_{-q}),\] 
\[h(e_i)=\left\{\begin{array}{cl}
2e_i & \mbox{if $i=1$}\\
e_i & \mbox{if $2\leq i\leq r+1$}\\
-e_i & \mbox{if $-r-1\leq i\leq -2$}\\
-2e_i & \mbox{if $i=-1$}\\
0 & \mbox{otherwise}
\end{array}\right.,\quad
h(e'_i)=\left\{\begin{array}{cl}
e'_i & \mbox{if $1\leq i\leq r$}\\
-e'_i & \mbox{if $-r\leq i\leq -1$}\\
0 & \mbox{otherwise}
\end{array}\right..\]
Let $Q=L\,Q^\mathrm u$ be the corresponding parabolic subgroup of $K$,
so that $L=K_h\cong\mathrm{GL}(1)\times\mathrm{GL}(r)\times\mathrm{SO}(2p-2r-2)
\times\mathrm{GL}(r)\times\mathrm{SO}(2q-2r)$.

The centralizer of $e$ is $K_e=L_eK_e^\mathrm u$ where
$L_e\cong\mathrm{GL}(r)\times\mathrm{SO}(2p-2r-2)\times\mathrm S(\mathrm{O}(1)\times\mathrm{O}(2q-2r-1))$,
the $\mathrm S(\mathrm{O}(1)\times\mathrm{O}(2p-2r-1))$ factor of $L_e$ is embedded as 
\[(z,A)\mapsto(z,z,A)\] into $\mathrm{GL}(1)\times\mathrm S(\mathrm{O}(1)\times\mathrm{O}(2q-2r-1))$ 
where the $\mathrm S(\mathrm{O}(1)\times\mathrm{O}(2q-2r-1))$ factor 
is included in the $\mathrm{SO}(2q-2r)$ factor of L,
the $\mathrm{GL}(r)$ factor of $L_e$ is embedded skew-diagonally
\[B\mapsto(B,B^{-1})\]
into the $\mathrm{GL}(r)\times\mathrm{GL}(r)$ factor of $L$.
The quotient $\mathrm{Lie}\,Q^\mathrm u/\mathrm{Lie}\,K_e^\mathrm u$ is a simple $L_e$-module of dimension $r$ as follows. In $\mathfrak k(1)$ there are exactly two simple $L_e$-submodules, $W_0,W_1$, of highest weight $\omega_{r-1}$ w.r.t.\ the $\mathrm{GL}(r)$ factor, isomorphic as $L_e$-modules but lying in two distinct isotypical $L$-components. Let $V$ be the $L_e$-complement of $W_0\oplus W_1$ in $\mathrm{Lie}\,Q^\mathrm u$. As $L_e$-module, $\mathrm{Lie}\,K_e^\mathrm u$ is the direct sum of $V$ and a simple $L_e$-submodule of $W_0\oplus W_1$ which projects non-trivially on both summands $W_0$ and $W_1$.

If $r\leq p-2$ and $r=q-1$, 
the normal triple $h,e,f$, the parabolic subgroup $Q=L\,Q^\mathrm u$ 
and $L_e$ have the same description, with $K_e=L_eK_e^\mathrm u$.  
The quotient $\mathrm{Lie}\,Q^\mathrm u/\mathrm{Lie}\,K_e^\mathrm u$ remains a simple $L_e$-module of dimension $q-1$ but here in $\mathfrak k(1)$ there are exactly three simple $L_e$-submodules, $W_0,W_1,W_2$, of highest weight $\omega_{q-2}$ w.r.t.\ the $\mathrm{GL}(q-1)$ factor, isomorphic as $L_e$-modules but lying in three distinct isotypical $L$-components. Let $V$ be the $L_e$-complement of $W_0\oplus W_1\oplus W_2$ in $\mathrm{Lie}\,Q^\mathrm u$. As $L_e$-module, $\mathrm{Lie}\,K_e^\mathrm u$ is the direct sum of $V$ and a co-simple $L_e$-submodule of $W_0\oplus W_1\oplus W_2$ which projects non-trivially on every summand $W_0$, $W_1$ and $W_2$.

If $r=p-1$, there exist two orbits labelled I and II. 
The case (I) can be described as above by specializing $r$ equal to $p-1$.
The case (II) can be obtained from the case (I) conjugating by $\tau$.

\subsubsection{$\mathbf{(-3,+2^{2r},+1^{2p-1-2r},-1^{2q-2-2r})}$}\label{sss:DDday}

This case can be obtained from the case \ref{sss:DDady} by switching the role of $p$ and $q$.

\section{Tables of spherical nilpotent $K$-orbits  in $\mathfrak p$ in the classical non-Hermitian cases}\label{B}

Let $e \in \mathcal N_{\mathfrak p}$ and let $\{h,e,f\}$ be a normal triple containing it. The action of the semisimple element $h$ on $\mathfrak g$ induces a $\mathbb Z$-grading $\mathfrak g = \bigoplus_{i \in \mathbb Z} \mathfrak g(i)$, where we denote $\mathfrak g(i) = \{x \in \mathfrak g \ : \ [h,x] = ix \}$. This defines the \textit{height} of $e$ (which actually depends only on $Ge$), defined as $\height(e) = \max\{i \ : \ \mathfrak g(i) \neq 0\}$. By \cite[Theorem~2.6]{Pa1}, the orbit $Ge$ is spherical if and only if $\height(e) \leq 3$.

Similarly, one may consider the action of $h$ on $\mathfrak p$, and the corresponding $\mathbb Z$-grading $\mathfrak p = \bigoplus_{i \in \mathbb Z} \mathfrak p(i)$, where $\mathfrak p(i) = \mathfrak p \cap \mathfrak g(i)$. This defines the $\mathfrak p$-\textit{height} of $e$ (which actually depends only on $Ke$), defined as $\height_{\mathfrak p}(e) = \max\{i \ : \ \mathfrak p(i) \neq 0\}$. By \cite[Theorems~5.1 and 5.6]{Pa1}, the orbit $Ke$ is spherical if $\height(e) \leq 3$, whereas if $Ke$ is spherical then $\height(e) \leq 4$ and $\height_{\mathfrak p}(e) \leq 3$. Similarly to the adjoint case, $\overline{Ke}$ is normal if $\height_{\mathfrak p}(e) = 2$, in which case Hesselink's proof \cite{He} of the normality of the closure of a nilpotent adjoint $G$-orbit of height 2 essentially applies (see \cite[Proposition 2.1]{Ki06}).

In Tables~2--11, for every spherical orbit $Ke \subset \mathcal N_{\mathfrak p}$, we report its signed partition (column~2), the Kostant-Dynkin diagram and the height of $Ge$ (columns~3 and 4), the Kostant-Dynkin diagram and the $\mathfrak p$-height of $Ke$ (columns~5 and 6), the normality of $\ol{Ke}$ (column~7), the codimension of $\ol{Ke} \smallsetminus Ke$ in $\ol{Ke}$ (column~8) and the weight semigroup of $\wt{Ke}$ (column~9).

In the orthogonal cases, the generators of the weight semigroups given in the tables are expressed in terms of the following variation of the fundamental weights of an irreducible root system $R$
$$
\varpi_i = \left\{
	\begin{array}{ll}
	2\gro_i & \text{if } i = n, \; R = \sfB_n\\
	\gro_{n-1} + \gro_n & \text{if } i = n-1, n, \; R = \sfD_n\\
	\gro_i & \text{otherwise}
	\end{array} \right. \\
$$
and we set $\varpi_0 = 0$. 

In all cases with a Roman numeral, (I) or (II), one $K$-orbit is obtained from the other one by applying an involutive automorphism of a factor of $K$ of type $\mathsf D$. In some of these cases, the generators of $\grG(\wt{Ke})$ are given just for the $K$-orbit labeled with (I), the generators for the other one are obtained by switching  $\gro_{p-1}$ and $\gro_p$ (resp.\ $\gro'_{q-1}$ and $\gro'_q$) if the first (resp.\ the second) component is the one involved by the mentioned automorphism. Which is the component that is involved by the automorphism is clear from the Kostant-Dynkin diagrams of the two orbits.

In Tables~12--20, for every spherical nilpotent orbit $Ke$ in $\gop$, we report the Luna diagram and the set of spherical roots of the spherical system of $K[e]$.

For every family of $K$-orbits, we draw the Luna diagram as it looks like for values of the parameters $n,p,q$ big enough with respect to $r$. When $r$ becomes close to $n$, $p$ or $q$ the diagram may change. Let us explain how it changes.

Whenever $K$ has a factor of type $\mathsf D_t$, where the diagram ends with 
\[\begin{picture}(8100,2400)(-1500,-1200)
\put(0,0){\usebox{\edge}}
\put(1800,0){\usebox{\susp}}
\put(5400,0){\usebox{\bifurc}}
\put(0,0){\usebox{\wcircle}}
\thicklines
\put(-1500,0){\line(1,0){1500}}
\end{picture}\]
(the corresponding simple root $\alpha_s$ moving a color of type $\mathrm b$, 
with $\alpha_{s+1},\ldots,\alpha_t$ belonging to $S^\mathrm p$) as in case~3.1, the given diagram is for $s<t-1$. 
If $s=t-1$, both the simple roots $\alpha_{t-1}$ and $\alpha_t$ move colors of type $\mathrm b$.
If $s=t$, with Roman numeral (I), $\alpha_{t-1}\in\supp\Sigma$ and $\alpha_t$ moves a color of type $\mathrm b$.
If $s=t$, with Roman numeral (II), $\alpha_{t}\in\supp\Sigma$ and $\alpha_{t-1}$ moves a color of type $\mathrm b$.
For example, the diagram of the case~3.1 becomes\\ 
if $r=n-1$
\[\begin{picture}(9000,3000)(-300,-1500)
\multiput(0,0)(5400,0){2}{\usebox{\edge}}
\put(1800,0){\usebox{\susp}}
\put(7200,0){\usebox{\bifurc}}
\multiput(0,0)(5400,0){2}{\multiput(0,0)(1800,0){2}{\usebox{\aprime}}}
\multiput(8400,-1200)(0,2400){2}{\usebox{\wcircle}}
\end{picture}\]
if $r=n$ (I)
\[\begin{picture}(9000,2700)(-300,-1500)
\multiput(0,0)(5400,0){2}{\usebox{\edge}}
\put(1800,0){\usebox{\susp}}
\put(7200,0){\usebox{\bifurc}}
\multiput(0,0)(5400,0){2}{\multiput(0,0)(1800,0){2}{\usebox{\aprime}}}
\put(8400,1200){\usebox{\aprime}}
\put(8400,-1200){\usebox{\wcircle}}
\end{picture}\]
if $r=n$ (II)
\[\begin{picture}(9000,3600)(-300,-2100)
\multiput(0,0)(5400,0){2}{\usebox{\edge}}
\put(1800,0){\usebox{\susp}}
\put(7200,0){\usebox{\bifurc}}
\multiput(0,0)(5400,0){2}{\multiput(0,0)(1800,0){2}{\usebox{\aprime}}}
\put(8400,-1200){\usebox{\aprime}}
\put(8400,1200){\usebox{\wcircle}}
\end{picture}\]

Whenever $K$ has a factor of type $\mathsf D_t$, where the diagram ends with a tail
\[\begin{picture}(8100,2400)(-1500,-1200)
\put(0,0){\usebox{\edge}}
\put(1800,0){\usebox{\susp}}
\put(5400,0){\usebox{\bifurc}}
\put(0,0){\usebox{\gcircle}}
\thicklines
\put(-1500,0){\line(1,0){1500}}
\end{picture}\]
(the corresponding simple root $\alpha_s$ moving a color of type $\mathrm b$, 
with $\alpha_{s+1},\ldots,\alpha_t$ belonging to $S^\mathrm p$ 
and $2(\alpha_s+\ldots+\alpha_{t-2})+\alpha_{t-1}+\alpha_t$ belonging to $\Sigma$) as in case~7.2, 
the given diagram is for $s<t-1$. 
If $s=t-1$, the simple roots $\alpha_{t-1}$ and $\alpha_t$ move the same color of type $\mathrm b$
($\alpha_{t-1}+\alpha_t$ is a spherical root).
For example, the diagram of the case~7.2 for $r=q-2$ becomes
\[\begin{picture}(24750,3600)(-300,-1800)
\multiput(0,0)(17100,0){2}{
\put(0,0){\usebox{\edge}}
\multiput(0,0)(1800,0){2}{\usebox{\aone}}
\put(5400,0){\usebox{\aone}}
}
\multiput(1800,0)(7200,0){2}{\usebox{\susp}}
\put(18900,0){\usebox{\susp}}
\multiput(5400,0)(1800,0){2}{\usebox{\edge}}
\put(7200,0){\usebox{\wcircle}}
\put(12600,0){\usebox{\rightbiedge}}
\put(22500,0){\usebox{\bifurc}}
\multiput(23700,-1200)(0,2400){2}{\usebox{\wcircle}}
\multiput(24000,-1200)(0,2400){2}{\line(1,0){450}}
\put(24450,-1200){\line(0,1){2400}}
\multiput(0,900)(17100,0){2}{\line(0,1){900}}
\put(0,1800){\line(1,0){17100}}
\multiput(1800,900)(17100,0){2}{\line(0,1){600}}
\put(1800,1500){\line(1,0){15200}}
\put(17200,1500){\line(1,0){1700}}
\multiput(5400,900)(17100,0){2}{\line(0,1){300}}
\put(5400,1200){\line(1,0){11600}}
\put(17200,1200){\line(1,0){1600}}
\put(19000,1200){\line(1,0){3500}}
\multiput(1800,-900)(15300,0){2}{\line(0,-1){900}}
\put(1800,-1800){\line(1,0){15300}}
\multiput(3600,-1500)(0,300){3}{\line(0,1){150}}
\put(3600,-1500){\line(1,0){13400}}
\put(17200,-1500){\line(1,0){1700}}
\put(18900,-1500){\line(0,1){600}}
\put(5400,-900){\line(0,-1){300}}
\put(5400,-1200){\line(1,0){11600}}
\put(17200,-1200){\line(1,0){1600}}
\put(19000,-1200){\line(1,0){1700}}
\multiput(20700,-1200)(0,300){2}{\line(0,1){150}}
\multiput(0,600)(1800,0){2}{\usebox{\toe}}
\multiput(18900,600)(3600,0){2}{\usebox{\tow}}
\end{picture}\]

Whenever $K$ has a factor of type $\mathsf B_t$, where the diagram ends with a tail
\[\begin{picture}(8700,1800)(-1500,-900)
\put(0,0){\usebox{\edge}}
\put(1800,0){\usebox{\susp}}
\put(5400,0){\usebox{\rightbiedge}}
\put(0,0){\usebox{\gcircletwo}}
\thicklines
\put(-1500,0){\line(1,0){1500}}
\end{picture}\]
(the corresponding simple root $\alpha_s$ moving a color of type $\mathrm b$, 
with $\alpha_{s+1},\ldots,\alpha_t$ belonging to $S^\mathrm p$ 
and $2(\alpha_s+\ldots+\alpha_{t})$ belonging to $\Sigma$) as in case~7.3, 
the given diagram is for $s<t$. 
If $s=t$, the simple root $\alpha_t$ moves a color of type $2\mathrm a$
($2\alpha_t$ is a spherical root).
For example, the diagram of the case~7.3 for $r=p-1$ becomes
\[\begin{picture}(24000,3600)(-300,-1800)
\put(0,0){\usebox{\edge}}
\put(1800,0){\usebox{\susp}}
\multiput(0,0)(1800,0){2}{\usebox{\aone}}
\put(5400,0){\usebox{\aone}}
\put(9900,0){
\put(0,0){\usebox{\edge}}
\multiput(1800,0)(7200,0){2}{\usebox{\susp}}
\multiput(5400,0)(1800,0){2}{\usebox{\edge}}
\multiput(0,0)(1800,0){2}{\usebox{\aone}}
\put(5400,0){\usebox{\aone}}
}
\put(7200,0){\usebox{\aprime}}
\put(5400,0){\usebox{\rightbiedge}}
\put(-7200,0){
\put(24300,0){\usebox{\wcircle}}
\put(29700,0){\usebox{\bifurc}}
}
\multiput(0,900)(9900,0){2}{\line(0,1){900}}
\put(0,1800){\line(1,0){9900}}
\multiput(1800,900)(9900,0){2}{\line(0,1){600}}
\put(1800,1500){\line(1,0){8000}}
\put(10000,1500){\line(1,0){1700}}
\multiput(5400,900)(9900,0){2}{\line(0,1){300}}
\put(5400,1200){\line(1,0){4400}}
\put(10000,1200){\line(1,0){1600}}
\put(11800,1200){\line(1,0){3500}}
\multiput(0,-900)(11700,0){2}{\line(0,-1){900}}
\put(0,-1800){\line(1,0){11700}}
\put(1800,-900){\line(0,-1){600}}
\put(1800,-1500){\line(1,0){9800}}
\put(11800,-1500){\line(1,0){1700}}
\multiput(13500,-1500)(0,300){3}{\line(0,1){150}}
\multiput(3600,-1200)(0,300){2}{\line(0,1){150}}
\put(3600,-1200){\line(1,0){8000}}
\put(11800,-1200){\line(1,0){1600}}
\put(13600,-1200){\line(1,0){1700}}
\put(15300,-1200){\line(0,1){300}}
\multiput(1800,600)(3600,0){2}{\usebox{\tow}}
\multiput(9900,600)(1800,0){2}{\usebox{\toe}}
\end{picture}\]


\includepdf[fitpaper,pages=-,landscape]{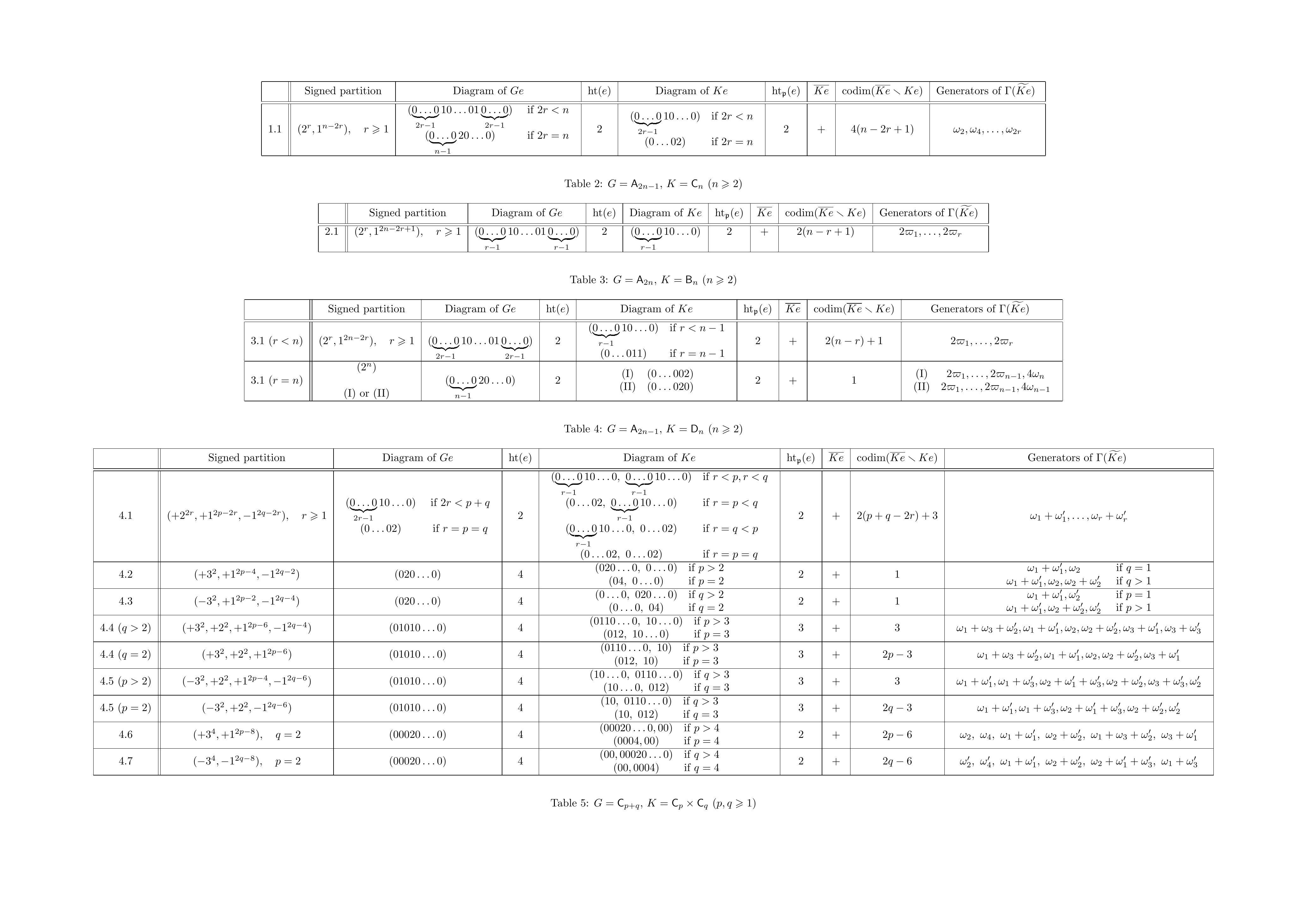}
\includepdf[fitpaper,pages=-,landscape]{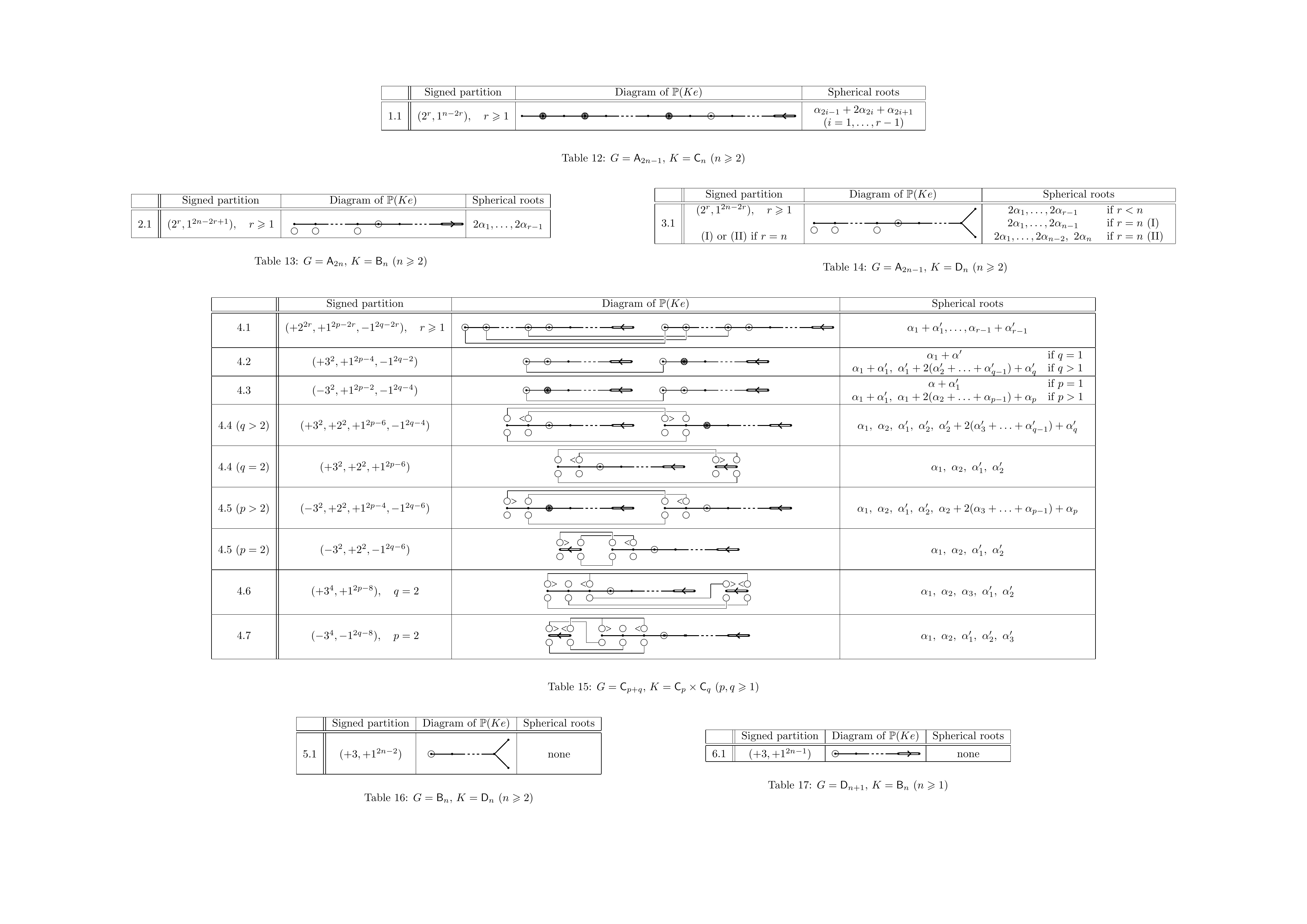}

\end{document}